\documentclass[dvipsnames]{amsart}
\usepackage{pict2e}
\usepackage{amsmath,amssymb,amsthm}
\usepackage{extarrows}
\usepackage{fullpage}
\usepackage{todonotes}

\allowdisplaybreaks
\usepackage{soul}

\usepackage{color}
\usepackage{hyperref}
\hypersetup{
	colorlinks=true,
	linkcolor=blue,
	citecolor=blue,
	urlcolor=blue
}

\usepackage{graphicx}
\usepackage{mathrsfs}
\usepackage{tikz-cd}
\usepackage{array}
\usepackage{multirow}
\usepackage{soul}
\usepackage{quiver}

\usepackage{xcolor}

\usepackage{tikz}

\newtheorem{theorem}{Theorem}[section]
\newtheorem{lemma}[theorem]{Lemma}	
\newtheorem{proposition}[theorem]{Proposition}
\newtheorem{corollary}[theorem]{Corollary}
\theoremstyle{definition}
\newtheorem{definition}[theorem]{Definition} 
\newtheorem{remark}[theorem]{Remark}	
\newtheorem{example}[theorem]{Example}

\newcommand{\longdownarrow}{\lower 1.4ex\hbox{\begin{picture}(18,18)(0,0)
\thicklines
\put(0,18){\vector(0,-1){18}}
\end{picture}}}
\newcommand{\longsearrow}{\lower 1.4ex\hbox{\begin{picture}(18,18)(0,0)
\thicklines
\put(0,18){\vector(1,-1){18}}
\end{picture}}}
\newcommand{\longssearrow}{\lower 1.4ex\hbox{\begin{picture}(18,18)(0,0)
\thicklines
\put(0,18){\vector(1,-2){18}}
\end{picture}}}
\newcommand{\longeearrow}{\lower 1.4ex\hbox{\begin{picture}(0,0)(9,9)
\thicklines
\put(0,18){\vector(1,0){18}}
\end{picture}}}
 \usepackage{relsize}
\theoremstyle{definition} 
\newtheorem*{ack}{Acknowledgements}

\numberwithin{equation}{section}
\newcommand{\C}{\mathcal{C}}
\newcommand{\R}{\mathbb{R}}
\newcommand{\N}{\mathbb{N}}
\newcommand{\Z}{\mathbb{Z}}
\title[Pansu pullback and spectral complexes]{Pansu pullback and spectral complexes} 
 
\author[F.~Lo Biundo]{Filippa Lo Biundo}
\author[F.~Tripaldi]{Francesca Tripaldi}

\address[F. Lo Biundo]%
{Department of Pure Mathematics, University of Leeds, Woodhouse, LS2 9JT Leeds, UK} 
\email{f.lobiundo@leeds.ac.uk}

\address[F. Tripaldi]%
{Department of Pure Mathematics, University of Leeds, Woodhouse, LS2 9JT Leeds, UK} 
\email{f.tripaldi@leeds.ac.uk}

\date{}

\begin{document}
\maketitle
\begin{abstract}
  In this paper, we prove the commutativity between the Pansu pullback of a smooth contact map between Carnot groups and the differentials appearing in the spectral complexes. As a direct application, we also present a way of ``lifting'' a Pansu derivative (viewed as a Lie algebra homomorphism) from  Carnot groups to their central extensions. 
\end{abstract}
\section{Introduction}
The naturality of the exterior derivative is a fundamental property of the de Rham complex: for every smooth map $\varphi:M_1\to M_2$ between smooth manifolds, $d\varphi^*\alpha=\varphi^*d\alpha$ for all $\alpha\in\Omega^\bullet(M_2)$.
Consequently, pullback defines a cochain map and induces a homomorphism on de Rham cohomology \cite{bott-tu}. In this paper, we establish an analogous naturality property for the Pansu pullback on Carnot groups, using the differentials of the spectral complexes introduced in \cite{tripaldi2026spectralcomplexestruncatedmulticomplexes}.

The stratification of a Carnot Lie algebra equips differential forms with a decomposition by weight and provides an intrinsic notion of differentiability adapted to the group dilations. The corresponding Pansu derivative is a homogeneous group homomorphism and, at the Lie algebra level, preserves the layers of the stratification \cite{Pansu-quasiconformal}. Replacing the classical differential by the Pansu derivative in the definition of pullback gives the \emph{Pansu pullback}, which preserves the weight of forms. This construction is particularly relevant to the analysis of mappings between Carnot groups: Lipschitz maps are Pansu differentiable almost everywhere, and Pansu differentiation and pullback play a central role in the study of Sobolev mappings and rigidity phenomena \cite{Pansu-quasiconformal,kleiner2020pansu,kleiner2021sobolev}. We refer to \cite{LeDonne+2017+116+137,ABB,donne2024metric} for background on Carnot groups and subRiemannian geometry.

Despite this compatibility with the grading, the Pansu pullback does not generally commute with the exterior derivative. A related question concerns the Rumin complex $(E_0^\bullet,d_c)$, a differential complex adapted to the subRiemannian structure and computing the same cohomology as the de Rham complex \cite{rumin1999differential,rumin_palermo,CompensatedCompactness,FT}. On Heisenberg groups, the naturality of the Rumin complex under smooth contact maps and its compatibility with Pansu pullback are well understood \cite{canarecci2019insightrumincohomologyorientability}. Kleiner, M\"uller, and Xie established weighted distributional pullback identities on general Carnot groups and distributional commutation with the Rumin differential on Heisenberg groups under suitable Sobolev assumptions \cite{kleiner2020pansu,kleiner2021sobolevmappingsrumincomplex}. Crucially, on general Carnot groups, Pansu pullback need not preserve the space of Rumin forms, but even projecting onto this space $E_0^\bullet$ does not in general restore commutation with $d_c$. Explicit examples are discussed in Section~\ref{chapter pansu pullback}.

The spectral complexes introduced in
\cite{tripaldi2026spectralcomplexestruncatedmulticomplexes}
provide another family of differential complexes adapted
to the graded structure. Starting from a truncated
multicomplex, their construction assembles differentials
from different pages of its spectral sequence into
complexes through suitable mixed quotients. For a Carnot
group $G$, the decomposition of differential forms by
degree and weight makes the de Rham complex into such
a multicomplex. Its weight filtration determines the
cycle spaces $Z_i^{p,k-p}$, the boundary spaces
$B_j^{p,k-p}$, and the spectral differentials $\Delta_i$,
while the weight components of the associated Rumin
differential determine the indices used in the assembly.
We retain the resulting spectral complexes in their
quotient form $E_{i,j}^{p,k-p}
    =
    Z_i^{p,k-p}/B_j^{p,k-p}$,
as recalled in
Section~\ref{chapter truncated multicomplex}.

Our main result establishes the naturality of these
spectral complexes under Pansu pullback. For every
Euclidean smooth Pansu differentiable map
$\varphi:G_1\to G_2$,
Proposition~\ref{prop:Pansu-preserves-Zr-Br} shows that
the Pansu pullback preserves the cycle and boundary
spaces, and therefore induces maps between the
corresponding mixed quotients.
Theorem~\ref{thm: commutativity} proves that these maps
commute with the spectral differentials:
\[
    \varphi_P^*\bigl(\Delta_i[\alpha]\bigr)
    =
    \Delta_i\bigl(\varphi_P^*[\alpha]\bigr).
\]
Here $[\alpha]\in
Z_i^{p,k-p}(G_2)/B_j^{p,k-p}(G_2)$, the indices are
chosen so that the corresponding arrows are defined
on both groups, and the equality holds in the target
quotient by $B_i^{p+i,k+1-p-i}(G_1)$.

The mixed quotients are essential to the scope of this
statement. On the ordinary $i$-th page, the cycle and
boundary indices coincide, considering only the quotients $Z_i/B_i$. In a
spectral complex these indices can differ, and
successive arrows can arise from different pages. Our theorem
establishes compatibility with this more refined
structure: the Pansu pullback acts on the mixed
quotients themselves and commutes with the
differentials assembled into the spectral complexes.

The proof relies on the relation between the classical
and Pansu pullbacks, together with the naturality of
the exterior derivative. By
Proposition~\ref{prop:weight-behaviour-pullbacks},
the classical pullback of such a map does not decrease
weight, and its component preserving weight is
precisely the Pansu pullback. Decomposing
$d\varphi^*=\varphi^*d$ according to weight gives
compatibility relations between the weight components
of the classical pullback and those of $d$; see
Lemma~\ref{lem:weight-components-classical-pullback}.
These relations allow us to transport the smooth
correction forms defining the cycle and boundary
spaces and to prove commutation with the spectral
differentials.

The same argument also establishes commutation on
the ordinary spectral sequences of the smooth
de Rham complexes. For Sobolev mappings,
Kleiner, M\"uller, and Xie established a
spectral-sequence pullback theorem with values in
the spectral sequence of a rough de Rham complex
\cite{kleiner2022sobolev}.
Our method gives a direct proof within the smooth
category and treats the mixed quotients and
differentials constituting the spectral complexes.

As an application, we investigate lifting problems for
1-dimensional central extensions. The approach developed
here stems from the second author's idea of combining central
extensions of both the source and the target with commutation
properties of the Pansu pullback. The central-extension
description of the first Heisenberg group gives an elementary
model: for every smooth map
$\varphi:\mathbb R^2\to\mathbb R^2$, the derivative
$D\varphi(x)$ admits the graded lift
$$\widehat{D\varphi}(x)
    =
    \begin{pmatrix}
        D\varphi(x) & 0\\
        0 & \det D\varphi(x)
    \end{pmatrix}
    \colon\mathfrak h^1\longrightarrow\mathfrak h^1\,.$$
    
Its action on the centre is the multiplication by
$\det D\varphi(x)$, and is therefore the identity precisely
when $\det D\varphi(x)=1$. If a smooth contact map
$\Phi(x,t)=(\varphi(x),F(x,t))$ covers $\varphi$, then
$D_P\Phi(x,t)=\widehat{D\varphi}(x)$. The existence of the
pointwise homomorphisms $\widehat{D\varphi}(x)$ alone,
however, does not ensure the existence of such a map.

For general Carnot groups, the algebraic lifting problem
is governed by the Lie algebra cohomology in degree $2$, which
classifies central extensions with the base and central
kernel fixed. Our spectral commutation formula provides
an explicit expression for the cohomological data involved.
Let $\omega\in\ker\Box_0\cap\Omega_L^2(G_2)$
be a left-invariant harmonic cocycle, where $\Box_0$ is the
algebraic Hodge Laplacian associated with $d_0$, and let
$\Pi_0$ denote the fibrewise orthogonal projection on $\ker\Box_0\cap\Omega^2(G_1)$.
If $\Pi_0(\varphi_P^*\omega)$ is left-invariant, then every
$D_P\varphi(x)$ lifts to a Lie algebra homomorphism between
the fixed central extensions defined by
$\Pi_0(\varphi_P^*\omega)$ and $\omega$, acting as the identity
on their central kernels. Proposition~\ref{prop: extensions}
expresses $\Pi_0(\varphi_P^*\omega)$ in terms of the weight
components of the Rumin differential applied to Pansu
pullbacks of suitable horizontal $1$-forms.

This viewpoint is also applied in
\cite{hakavuori2025smoothcontactliftscentral} in the setting
of smooth contact lifts. Their characterisation uses the
full Rumin differential of the pullback of a de Rham
potential for the target extension cocycle. They also
establish a cohomological condition involving Pansu
pullback and prove its sufficiency under additional
hypotheses. Under the homogeneity and weight assumptions
of Corollary~\ref{cor: extensions to stratifiable}, their
results yield a smooth contact lift between the stratified
central extensions constructed here. The pointwise
algebraic lifting construction can also be applied to cocycles
with several weight components, including examples whose
central extensions are not stratifiable.

The paper is organised as follows. Section~\ref{chapter truncated multicomplex} recalls the graded structure of differential forms on Carnot groups, the resulting de Rham multicomplex, and the associated Rumin and spectral complexes. Section~\ref{chapter pansu pullback} compares the classical and Pansu differentials and pullbacks, and presents examples illustrating the failure of commutation with the exterior derivative and Rumin differentials. Section~\ref{chapter commutativity} proves the main commutation theorem and discusses its formulation using harmonic representatives. Finally, Section~\ref{section application} develops the application to central extensions of nilpotent Lie groups.

\ack{The authors would like to thank B. Kleiner, S. M\"uller, and X. Xie for their insightful feedback on an earlier draft, which helped shape the current version of this paper. They are also grateful to V. Magnani for taking the time to explain his work on Pansu differentiability, which forms the basis of Section 3, as well as his feedback on the paper. We thank A.~Baldi and A.~Rosa for drawing our attention to
potential closedness issues in the orthogonal realisation
of the spectral complexes, discussed in
Subsection~\ref{subsect: Orthogonal realisations of the spectral complexes}.
Their observations led us to retain the quotient formulation
of the spectral complexes and to develop an alternative
proof of the main commutativity theorem, based entirely on
the relation between the classical and Pansu differentials
and the naturality of the exterior derivative.}

\section{The de Rham complex on Carnot groups and its associated spectral complexes}\label{chapter truncated multicomplex}
The main result of this paper concerns the commutativity of the Pansu pullback with the differentials of the spectral complexes introduced in \cite{tripaldi2026spectralcomplexestruncatedmulticomplexes}. In order to make this paper as self-contained as possible, we recall these complexes in the setting of Carnot groups. Starting from the decomposition of the de Rham differential according to weight increase, we describe the resulting truncated multicomplex and its associated spectral sequence. We also recall the Rumin complex, whose differential determines the weight increments used to select the spectral complexes. These are then presented through mixed quotients of cycle and boundary spaces, connected by differentials drawn from the corresponding pages of the spectral sequence. Unlike in \cite{tripaldi2026spectralcomplexestruncatedmulticomplexes}, we keep the spectral complexes as quotient spaces (see Subsection \ref{subsect: Orthogonal realisations of the spectral complexes} for a short explanation of this choice).

Let us first fix some notation. Given an $n$-dimensional (real) Lie algebra $\mathfrak{g}$ with basis $(X_1,\ldots,X_n)$, we will denote its dual space by $\mathfrak{g}^\ast$. This is the vector space of all (real-valued) linear functionals on the elements of $\mathfrak{g}$, also referred to as the space of 1-covectors.

\begin{definition}[Dilations]\label{def: dilations}
    Let \(G\) be a connected, simply connected Lie group with Lie algebra
\(\mathfrak g\). A family of dilations on \(\mathfrak g\) is a
one-parameter family of Lie algebra automorphisms
$\{ \delta_\lambda\}_{\lambda>0}$
of the form $\delta_\lambda=\operatorname{Exp}(A\ln\lambda)$,
where \(A:\mathfrak g\to\mathfrak g\) is a diagonalisable derivation
with strictly positive eigenvalues and $\operatorname{Exp}$ is the exponential on the space of morphisms of the vector space $\mathfrak{g}$. If \(\mathfrak g\) admits such a
family, then \(G\) is called a homogeneous Lie group.

    Let $(X_1,\ldots,X_n)$ be a basis of eigenvectors of $A$. In this basis, $\delta_\lambda$ is represented by the diagonal matrix $\operatorname{Mat}(\delta_\lambda)=\operatorname{diag}(\lambda^{\nu_1},\ldots,\lambda^{\nu_n})$, so that $\delta_\lambda X_i=\lambda^{\nu_i}X_i$ for each $i=1,\ldots,n$. The eigenvalues $\nu_1,\ldots,\nu_n$ are called the \textit{weights} of the dilations.
    \end{definition}

    Let $0<w_1<\cdots<w_s$
be the distinct weights, and denote by \(V_{w_j}\) the
eigenspace of \(A\) corresponding to \(w_j\). Then the Lie algebra $\mathfrak{g}$ decomposes into a direct sum
\begin{align}\label{eq: direct sum decomp}
        \mathfrak{g}=\bigoplus_{j=1}^sV_{w_j}\ \text{ such that }\ [V_{w_i},V_{w_j}]\subseteq V_{w_i+w_j}\ ,\ 1\le i,j\le s\,.
    \end{align}
    As a direct consequence, we get that $\mathfrak{g}$ (and hence also the homogeneous Lie group $G$ with Lie algebra $\mathfrak{g}$) is nilpotent. 

    Therefore, the Lie group $G$ we are considering is connected simply-connected and nilpotent, which implies that its exponential map $\exp\colon\mathfrak{g}\to G$ is a bijection and a global diffeomorphism. 
    The dilations of the Lie algebra therefore
induce Lie group automorphisms $\exp\circ\delta_\lambda\circ\exp^{-1}\colon G\longrightarrow G$. We will keep the same notation for the dilations on the group $\delta_\lambda\colon G\to G$, so that $\delta_\lambda\exp X=\exp\delta_\lambda X$ for any $X\in\mathfrak{g}$.

    When the weights of the dilations may be assumed to be integers, the group is said to be (positively) \textit{gradable}. If in addition one can take $w_1=1$ with $V_{1}$  generating the whole Lie algebra $\mathfrak{g}$ via successive brackets, the Lie group $G$ is called \textit{stratifiable} and we call the corresponding direct sum decomposition a \textit{stratification}. A Lie group $G$ may admit several non-equivalent homogeneous structures \cite{hakavuori2022gradings}.
    \begin{definition}[Carnot groups]\label{def: carnot group}
    We say that a group $G$ is \textit{Carnot} if it is stratifiable, we pick the homogeneous structure associated with the stratification, and we fix a scalar product on $V_1$. The scalar
product is extended by left translations to the horizontal
distribution $H_xG:=(dL_x)_eV_1$, at each $x\in G$
and determines a left-invariant Carnot-Carath\'eodory distance on \(G\). As such, Carnot groups represent the simplest examples of subRiemannian manifolds.
\end{definition}
    \begin{remark}
        When dealing with a Lie group $G$, one can consider the subcomplex of the de Rham complex consisting of the left-invariant differential forms $\Omega_L^\bullet(G)$. A left-invariant $k$-form is uniquely determined by its value at the identity, where it defines a linear map $\bigwedge^k\mathfrak{g}\to\R$, by identifying the tangent space at the identity with the Lie algebra $\mathfrak{g}$, and so $\Omega^k_L(G)\cong\bigwedge^k\mathfrak{g}^\ast$. Moreover,  we can identify the tangent space $T_xG$ to $G$ at any point $x\in G$ with $\mathfrak{g}$ by means of the isomorphism $(dL_x)_{e}$, where $L_x$ denotes the left-translation by $x\in G$. For a given
\(\xi\in\bigwedge^k\mathfrak g^*\), the corresponding left-invariant
form \(\xi^L\) is defined by
\[
\xi_x^L(v_1,\ldots,v_k)
=
\xi\bigl(
(dL_{x^{-1}})_xv_1,\ldots,
(dL_{x^{-1}})_xv_k
\bigr)\ \text{ for every }x\in G \text{ and }v_1,\ldots,v_k\in T_xG\,.
\]
Consequently, left translations give a natural identification
\[
\mathcal C^\infty(G)\otimes
\bigwedge\nolimits^k\mathfrak g^*
\longrightarrow
\Omega^k(G)\ ,\ 
\
f\otimes\xi\longmapsto f\,\xi^L\,,
\]
where $(f\,\xi^L)_x=f(x)\xi_x^L$ at each $x\in G$.
This is an isomorphism of \(\mathcal C^\infty(G)\)-modules.

In particular, after fixing a basis of \(\mathfrak g^*\), every smooth
differential form can be written uniquely as a finite sum of smooth
coefficient functions multiplied by left-invariant forms.
    \end{remark}

We now define the fibrewise action which determines the weight of a
smooth differential form. It is important to distinguish this action
from the geometric pullback by the group automorphism
\(\delta_\lambda\colon G\to G\).

For any given $f\in\mathcal{C}^\infty(G)$ and $\xi\in\bigwedge^k\mathfrak{g}^\ast$, we define
\begin{equation}
\label{eq:fibrewise-dilation-forms} \delta_\lambda(f\,\xi^L)
:=
f\,
\bigl(\delta_\lambda^*\xi\bigr)^L\ ,\ \text{ where }\bigl(\delta_\lambda^*\xi\bigr)
(X_1,\ldots,X_k)
=
\xi\bigl(
\delta_\lambda X_1,\ldots,
\delta_\lambda X_k
\bigr)
\end{equation}
and extend this definition linearly to \(\Omega^k(G)\). Moreover, these new dilation maps  respect the wedge product, since $\delta_\lambda(\alpha\wedge\beta)=\delta_\lambda\alpha\wedge\delta_\lambda\beta$ for each $\alpha,\beta\in\Omega^\bullet(G)$.

Notice that these \(\delta_\lambda\) only act on the covector part of a
differential form and leave its coefficient functions unchanged. In
contrast, the geometric pullback satisfies $\delta_\lambda^\ast(f\,\xi^L)=(f\circ\delta_\lambda)\,(\delta_\lambda^\ast\xi)^L$.
The two actions therefore agree on left-invariant forms, but not on
arbitrary smooth forms. Going forward, we will simplify the notation and simply write $f\xi=f\otimes\xi$ instead of $f\xi^L$.
\begin{definition}[Weights of forms]
\label{def: weights of forms}A smooth differential form \(\alpha\in\Omega^k(G)\) is said to have
pure weight \(p\) if $\delta_\lambda\alpha=\lambda^p\alpha$ for every $\lambda>0$.
In this case, we write $w(\alpha)=p$ and we denote the space of smooth \(k\)-forms of pure weight \(p\) by $\Omega^{p,k-p}(G)$.

Every smooth function has weight \(0\), because
\(\delta_\lambda\) acts trivially on the coefficient functions.
If $\mathfrak g=V_1\oplus\cdots\oplus V_s$
is a stratification, then the \textit{homogeneous dimension} of \(G\) is defined as $Q
=
\sum_{j=1}^s j\dim V_j$ and every nonzero top-degree form has pure weight \(Q\). For the
Carnot-Carathéodory distance fixed in
Definition~\ref{def: carnot group}, \(Q\) is also the Hausdorff
dimension of \(G\).

As a direct consequence we get the following direct sum decomposition of the space of smooth forms in each degree according to their weights, i.e.
\[
\Omega^k(G)
=
\bigoplus_{p=0}^Q\Omega^{p,k-p}(G)=\bigoplus_{p=0}^Q\C^\infty(G)\otimes\bigwedge\nolimits^{p,k-p}\mathfrak{g}^\ast\,.
\]
Finally, for brevity, we will also use the notation
\[
\Omega^{\geq p,k}(G)
:=
\bigoplus_{q\geq p}\Omega^{q,k-q}(G)\,.
\]
\end{definition}
Without loss of generality, one can consider a basis $(X_1,\ldots,X_n)$ adapted to the direct sum decomposition \eqref{eq: direct sum decomp}, whose expression can be simplified in the case of a stratification, so that
\begin{align*}
V_i=\operatorname{span}_\R\{X_{m_{i-1}+1},\ldots,X_{m_i}\}\ , \ \text{where }m_0=0\text{ and }m_i=\sum_{r=1}^i\operatorname{dim}V_r\ \text{ for }1\le i\le s\,.
\end{align*}
In particular, this implies that for any $i=1,\ldots,s$
\begin{align*}
    \delta_\lambda X_j=\lambda^iX_j \text{ whenever }m_{i-1}<j\le m_i\ \Longleftrightarrow X_j\in V_i\ \Longleftrightarrow\ w(X_j)=i\,.
\end{align*}
Likewise, if \((\theta_1,\ldots,\theta_n)\) is the dual basis,
regarded also as a left-invariant coframe on \(G\), then
\begin{align*}
    \delta_\lambda \theta_j=\lambda^i\theta_j \text{ whenever }m_{i-1}<j\le m_i\ \Longleftrightarrow \theta_j\in V_i^\ast\ \Longleftrightarrow\ w(\theta_j)=i\,.
\end{align*}
More generally, we have that $w\bigl(
f\,\theta_{j_1}\wedge\cdots\wedge\theta_{j_k}
\bigr)
=
w(\theta_{j_1})+\cdots+w(\theta_{j_k})$ for every $f\in\C^\infty(G)$. 
\begin{lemma}[Forms of different weight are linearly independent]\label{lem: different weights}
    Let us consider $\alpha_1,\alpha_2\in\Omega^k(G)$ two arbitrary $k$-forms. If they are both homogeneous with $w(\alpha_1)\neq w(\alpha_2)$, then they are linearly independent.
\end{lemma}
\begin{proof}
    
Let us first prove the result for left-invariant forms that is $\xi_1,\xi_2\in\bigwedge^k\mathfrak{g}^\ast$ with $w(\xi_1)\neq w(\xi_2)$.

    If $k=1$ and $w(\xi_1)=p_1\neq w(\xi_2)=p_2$, by definition of weights we have that $\xi_1^\ast\in V_{p_1}$ and $\xi_2^\ast\in V_{p_2}$ and hence they are linearly independent by the given direct sum decomposition.
    
    If $k>1$, given $\xi_1,\xi_2\in\bigwedge^k\mathfrak{g}^\ast$ with different weights, then without loss of generality we can assume $\xi_1=\theta_{i_1}\wedge\cdots\wedge\theta_{i_k}$ and $\xi_2=\theta_{j_1}\wedge\cdots\wedge\theta_{j_k}$ with
    \begin{align*}
        w(\xi_1)=w(\theta_{i_1})+\cdots+w(\theta_{i_k})\neq w(\xi_2)=w(\theta_{j_1})+\cdots+w(\theta_{j_k})\,. 
    \end{align*}
    This means that there exists at least one $l\in\{1,\ldots,k\}$ such that $w(\theta_{i_l})=p_1\neq w(\theta_{j_l})=p_2$, i.e. $\theta_{i_l}^\ast\in V_{p_1}$ and $\theta_{j_l}^\ast\in V_{p_2}$ belong to different subspaces of the direct sum decomposition.
    
    Finally, the claim follows from the fact that a smooth form $\alpha\in\Omega^k(G)$ is homogeneous of weight $p$ if $\alpha=\sum_{j}f_j\otimes\xi_j $ where $f_j\in \mathcal{C}^\infty(G)$ and $w(\xi_j)=p$ for each $j$.
\end{proof}
As a consequence, the space of smooth forms $\Omega^k(G)$ admits a direct sum decomposition given by the weight. Throughout this paper, we will express this decomposition using the following notation
\begin{align}\label{eq: direct sum decomp of forms according to weight}
    \Omega^k(G)=\Gamma\left(\bigwedge\nolimits^k\mathfrak{g}^\ast\right)=\bigoplus_{\substack{p+q=k\\0\le p\le Q}}\Gamma\left(\bigwedge\nolimits^{p,q}\mathfrak{g}^\ast\right)=\bigoplus_{\substack{p+q=k\\ 0\le p\le Q}}\mathcal{C}^\infty(G)\otimes\bigwedge\nolimits^{p,q}\mathfrak{g}^\ast=\bigoplus_{\substack{ p+q=k\\ 0\le p\le Q}}\Omega^{p,q}(G)\,,
\end{align}
where $\bigwedge^{p,q}\mathfrak{g}^\ast$ and $\Omega^{p,q}(G)$ denote the spaces of left-invariant and smooth forms of weight $p$ and degree $p+q=k$.

When working with differential forms on Carnot groups, $k$-forms of weight $p$ are typically denoted by $\Omega^{k,p}(G)$. However, since spectral sequence techniques play a central role in this paper, we adopt the notation that is standard in that context.

\begin{lemma}\label{lem: decomposing d according to weight increase}
    Given a homogeneous group $G$, the direct sum decomposition \eqref{eq: direct sum decomp} of its Lie algebra $\mathfrak{g}$ induces a decomposition of the exterior de Rham differential $d$ acting on smooth forms which can be easily expressed in terms of the weight increase. Given an arbitrary form $\alpha\in\Omega^{p,k-p}(G)$ of homogeneous weight $p$ and degree $k$, one can write
    \begin{align*}
        d\alpha=d_0\alpha+d_{w_1}\alpha+\cdots+d_{w_s}\alpha
    \end{align*}
    where $d_0\alpha\in\Omega^{p,k+1-p}(G)$ and $d_{w_i}\alpha\in\Omega^{p+w_i,k+1-w_i-p}(G)$ for each $i=1,\ldots,s$.
\end{lemma}
\begin{proof}
    Given an arbitrary $k$-form of weight $p$, we have $\alpha=\sum_{j}f_j\otimes\xi_j=\sum_jf_j\xi_j$ with $f_j\in \mathcal{C}^\infty(G)$ and $\xi_j\in\bigwedge^{p,k-p}\mathfrak{g}^\ast$. The exterior differential applied to $\alpha$ has the following expression\begin{equation}\label{explicit formula exterior derivative}
d\bigg(\sum_jf_j\xi_j\bigg)=\sum_j\left(df_j\wedge\xi_j+f_jd\xi_j\right)=\sum_jdf_j\wedge\xi_j+\sum_jf_jd\xi_j\,.
\end{equation}
If we consider a basis $(X_1,\ldots,X_n)$ adapted to the direct sum decomposition \eqref{eq: direct sum decomp}, we obtain a very explicit expression for the first summand, that is
\begin{align*}
    \sum_jdf_j\wedge\xi_j=\sum_j\sum_{l=1}^nX_lf_j\theta_l\wedge\xi_j=\sum_{i=1}^s\sum_{X_l\in V_{w_i}}\sum_jX_lf_j\theta_l\wedge\xi_j=\sum_{i=1}^s d_{w_i}\alpha\,.
\end{align*}
For each $i=1,\ldots,s$, we see that
\begin{align*}
    d_{w_i}\alpha=\sum_{X_l\in V_{w_i}}\sum_jX_lf_j\,\theta_l\wedge\xi_j\in\Omega^{p+w_i,k+1-p-w_i}(G)
\end{align*}
since each $X_l\in V_{w_i}$ and so $w(\theta_l)=w_i$.

Regarding the second summand, one needs to show that unless $d\xi_j$ vanishes, we have $w(d\xi_j)=w(\xi_j)=p$, i.e. it keeps the weight constant. In the case of homogeneous Lie groups, one can show this using the group's dilations. Indeed, this second summand coincides with the action of the exterior derivative on left-invariant forms, which can be seen as
\begin{align}\label{eq: formula for d_0}
    d\xi_j(X_1,\ldots,X_{k+1})=\sum_{1\le i<l\le k+1}(-1)^{i+l}\xi_j([X_i,X_l],X_1,\ldots,\hat{X}_i,\ldots,\hat{X}_l,\ldots,X_{k+1})\ \text{ for all }X_i\in\mathfrak{g}\,.
\end{align}
This formula, combined with the fact that dilations $\delta_\lambda$ are automorphisms of the Lie algebra $\mathfrak{g}$, readily implies that $d$ applied to $\bigwedge^k\mathfrak{g}^\ast$ commutes with the dilations. More explicitly, the fact that  $d(\delta_\lambda\xi_j)=\delta_\lambda(d\xi_j)$ implies that $w(d\xi_j)=w(\xi_j)=p$ and so
\begin{align*}
    d_0\alpha=\sum_jf_jd\xi_j\in\Omega^{p,k+1-p}(G)\,.
\end{align*}
\end{proof}
In order to apply the spectral sequence machinery developed in \cite{tripaldi2026spectralcomplexestruncatedmulticomplexes}, we first need to show that the cochain complex that we are working on, namely the de Rham complex, is indeed a truncated multicomplex.
\begin{definition}[Definition 2.1 in \cite{livernet2020spectral}] \label{def: multicomplex}Let $k$ be a commutative unital ground ring. An $s$-multicomplex (also known as twisted chain complex) is a $(\Z,\Z)$-graded $k$-module $\mathcal{C}$ equipped with $k$-linear maps $d_i\colon\mathcal{C}\to\mathcal{C}$ for $ i\ge 0$ of bidegree $\vert d_i\vert=(i,1-i)$ such that
\begin{align}\label{eq: rules maps}
    \sum_{i+j=n}d_id_j=0\ \text{ for all }n\ge 0\text{ and }d_k=0\text{ for all }k >s\,.
\end{align}
\end{definition}
We should also mention that in Definition \ref{def: multicomplex} we are choosing a cohomological sign and degree convention for the differentials $d_i$.

In order to simplify the notation of the following proposition, we will limit our considerations to Carnot groups, where the direct sum decomposition \eqref{eq: direct sum decomp} becomes a stratification $\mathfrak{g}=V_1\oplus\cdots\oplus V_s$ and so by Lemma \ref{lem: decomposing d according to weight increase} we have a simpler expression for the exterior derivative $d=d_0+d_1+\cdots+d_s$. Notice however that the exact same result holds for all positively graded Lie groups.
\begin{proposition}
    The de Rham complex $(\Omega^\bullet(G),d)$ on a Carnot group with stratification $\mathfrak{g}=V_1\oplus\cdots V_s$ is a truncated $s$-multicomplex.
\end{proposition}
\begin{proof}
    The direct sum decomposition according to weights given in \eqref{eq: direct sum decomp of forms according to weight} endows the space of smooth forms $\Omega^\bullet(G)$ with a bigrading. In particular, each $\Omega^{p,q}(G)$ is a $\mathbb R$-module of bidegree $(p,q)\in\N\times\Z$. Moreover, in Lemma \ref{lem: decomposing d according to weight increase} we established the existence of structure maps $d_i\colon\Omega^\bullet(G)\to\Omega^\bullet(G)$ of bidegree $\vert d_i\vert=(i,1-i)$. We are then left to prove that the formulae \eqref{eq: rules maps} also hold.

    These follow directly from the fact that $(\Omega^\bullet(G),d)$ is a complex, i.e. $d^2=0$, and that forms of different weight are linearly independent by Lemma \ref{lem: different weights}. Indeed, given a $k$-form $\alpha\in\Omega^{p,k-p}(G)$, if we expand the expression for $d^2\alpha$ according to the weight we get
    \begin{align*}
        d^2\alpha=&d(d_0\alpha+d_1\alpha+\cdots+d_s\alpha)=(d_0+d_1+\cdots+d_s)(d_0\alpha+d_1\alpha+\cdots+d_s\alpha)\\=&d_0^2\alpha+(d_0d_1+d_1d_0)\alpha+(d_0d_2+d_1^2+d_2d_0)\alpha+\cdots+d_s^2\alpha=\sum_{n=0}^{2s}\sum_{i+j=n}d_jd_i\alpha=0\,.
    \end{align*}
    For each $n=0,\ldots,2s$, we have $\sum_{i+j=n}d_jd_i\alpha\in\Omega^{p+n,k+2-p-n}(G)$, i.e. each summand has different weight. By Lemma \ref{lem: different weights}, this implies that each summand must be zero.
\end{proof}

\begin{remark}
    Each subspace $\Omega^{p,k-p}(G)\subset\Omega^k(G)$ is naturally a $\mathcal C^\infty(G)$-submodule. However, for $i>0$, the differential operators $d_i$ are only $\mathbb R$-linear. Thus, when considering the multicomplex structure on $\Omega^\bullet(G)$, we regard the spaces $\Omega^{p,k-p}(G)$ as real vector subspaces.
\end{remark}

Finally, since the spectral complexes can be defined in terms of the Rumin complex, and some of the results established below involve Rumin forms or the Rumin differential, we briefly recall the construction of this complex. We refer to \cite{rumin_grenoble,rumin_palermo,CompensatedCompactness,FT} for a more comprehensive treatment.

The spaces of Rumin forms $E_0^\bullet$ are realised as graded subspaces of $\Omega^\bullet(G)$ representing the cohomology of the complex $(\Omega^\bullet(G),d_0)$. In order to express them as subspaces rather than as quotient spaces $\ker d_0/\operatorname{Im}d_0$, we introduce an inner product on $\bigwedge^\bullet\mathfrak g^*$ that allows us to identify an orthogonal complement of $\operatorname{Im}d_0$ in $\ker d_0$.

Fix an inner product on $\mathfrak g$ adapted to the stratification, i.e. $V_i\perp V_j$ when $i\neq j$), obtained by extending the inner product on $V_1$ (see Definition \ref{def: carnot group}). We can use this scalar product to fix an ordered orthonormal basis $(X_1,\ldots,X_n)$ of $\mathfrak{g}$, then orient $\mathfrak{g}$ according to this  basis, so that the dual basis $(\theta_1,\ldots,\theta_n)$ is orthonormal and $\operatorname{vol}=\theta_1\wedge\cdots\wedge\theta_n$ is the associated unit volume form. Such a product on $\mathfrak{g}$ can be extended canonically to $\mathfrak{g}^\ast$ and then to the space of left-invariant $k$-forms $\Omega_L^k(G)\cong\bigwedge^k\mathfrak g^\ast$, and we denote the resulting inner product by $\langle\cdot,\cdot\rangle_k$. Since the scalar product $\langle\cdot,\cdot\rangle_k$ defined on each $\bigwedge^k\mathfrak{g}^\ast$ is adapted to the stratification, we have that covectors of different weight are orthogonal, i.e. for each degree 
\begin{align*}
    \text{ given }\xi_1\in\bigwedge\nolimits^{p_1,k-p_1}\mathfrak{g}^\ast\ \text{ and }\ \xi_2\in\bigwedge\nolimits^{p_2,k-p_2}\mathfrak{g}^\ast\ ,\ \text{ if }p_1\neq p_2\text{ then }\langle\xi_1,\xi_2\rangle=0\,.
\end{align*}
It is then possible to define the Hodge-$\star$ operator as the linear map acting on $\bigwedge^\bullet\mathfrak{g}^\ast$ as
\begin{align*}
    \star\colon\bigwedge\nolimits^k\mathfrak{g}^\ast\xrightarrow[]{\cong}\bigwedge\nolimits^{n-k}\mathfrak{g}^\ast\ ,\ \xi_1\wedge\star\xi_2:=\langle\xi_1,\xi_2\rangle_k\operatorname{vol}\ \text{ for any }\xi_1,\xi_2\in\bigwedge\nolimits^k\mathfrak{g}^\ast\,.
\end{align*}

The inner product $\langle\cdot,\cdot\rangle_k$ on $\bigwedge^k\mathfrak g^\ast$ extends pointwise to a $\mathcal{C}^\infty(G)$-bilinear product on $\Omega^k(G)$. More precisely, under the identification $\Omega^k(G)
\cong
\mathcal C^\infty(G)\otimes\bigwedge\nolimits^k\mathfrak g^\ast$, we define
\begin{equation}\label{eq: def pointwise scalar product}
\left\langle
\sum_i f_i\xi_i,\sum_j g_j\eta_j
\right\rangle_k
:=
\sum_{i,j}f_ig_j\langle\xi_i,\eta_j\rangle_k
\in\mathcal C^\infty(G)\,.
\end{equation}
The Hodge-\(\star\) operator also extends accordingly to a
\(\mathcal C^\infty(G)\)-linear map
$$
\star\colon\Omega^k(G)\xrightarrow[]{\cong}\Omega^{n-k}(G)\ ,\ 
\alpha_1\wedge\star\alpha_2
=
\langle\alpha_1,\alpha_2\rangle_k\operatorname{vol}\,.
$$

For a linear subspace $W\subseteq\bigwedge^k\mathfrak{g}^\ast$, let $\operatorname{pr}_W\colon\bigwedge^k\mathfrak{g}^\ast\to W$ denote the corresponding finite-dimensional orthogonal projection. It extends $\C^\infty(G)$-linearly to the space of smooth forms by the map $\operatorname{Id}_{C^\infty(G)}\otimes\operatorname{pr}_W\colon\Omega^k(G)\to\C^\infty(G)\otimes W$. Whenever $S=\C^\infty(G)\otimes W\subseteq\Omega^k(G)$ is a submodule of this form, we denote this fibrewise orthogonal projection by $\operatorname{pr}_S$. The orthogonal complements and projections below associated with $\operatorname{Im}d_0$, $\operatorname{Im}\delta_0$, and $E_0^\bullet$ are understood in this fibrewise sense.

We can use this inner product $\langle\cdot,\cdot\rangle_k$ just introduced to define the fibrewise, or algebraic, adjoint of $d_0$.

\begin{definition}[The adjoint of $d_0$]\label{def: delta_0} The algebraic adjoint of $d_0$ is the map $\delta_0\colon\Omega^{p,k-p}(G)\to\Omega^{p,k-1-p}(G)$ defined by
\begin{align*}
    \langle d_0\alpha_1,\alpha_2\rangle_{k}=\langle\alpha_1,\delta_0\alpha_2\rangle_{k-1}\ \text{ for any }\alpha_1\in\Omega^{k-1}(G)\text{ and }\alpha_2\in\Omega^{k}(G)\,.
\end{align*}    
\end{definition}
The operator $\delta_0$ preserves the weight of forms, namely $\delta_0\big(\Omega^{p,q}(G)\big)\subset\Omega^{p,q-1}(G)$. Indeed, this follows directly from the fact that forms of different weights are orthogonal and that $d_0$ preserves weight. Moreover, the associated Hodge-Laplacian 
\begin{align*}
    \Box_0:=d_0\delta_0+\delta_0d_0\colon\Omega^{p,k-p}(G)\to\Omega^{p,k-p}(G)\,.
\end{align*}
is a fibrewise symmetric algebraic operator preserving both the weight and the degree of forms.

\begin{definition}
    The space of Rumin forms $E_0^\bullet\subset \Omega^\bullet(G)$ is defined as the orthogonal complement of $\operatorname{Im}d_0$ within $\ker d_0$, and coincides with the space of harmonic forms for $\Box_0$:
    \begin{align*}
        E_0^\bullet:=\ker d_0\cap\big(\operatorname{Im}d_0\big)^\perp=\ker d_0\cap\ker \delta_0=\ker\Box_0\,.
    \end{align*}
\end{definition}

\begin{proposition}[Hodge decomposition for $\Box_0$]
For every degree $k$, the space of smooth $k$-forms admits the fibrewise orthogonal decomposition
\begin{align}\label{eq: Hodge decomp Box_0 degree}
\Omega^k(G)
=
\operatorname{Im}\bigl(
d_0\colon\Omega^{k-1}(G)\to\Omega^k(G)
\bigr)
\oplus
E_0^k
\oplus
\operatorname{Im}\bigl(
\delta_0\colon\Omega^{k+1}(G)\to\Omega^k(G)
\bigr).
\end{align}
Equivalently,
\begin{align}\label{eq: Hodge decomp Box_0}
\Omega^\bullet(G)
=
\operatorname{Im}d_0
\oplus
\ker\Box_0
\oplus
\operatorname{Im}\delta_0.
\end{align}
\end{proposition}
\begin{proof}
All orthogonal complements in this proof are understood fibrewise, and kernels and images are taken in the appropriate degrees. Since $d_0$ acts fibrewise as a linear map between finite-dimensional spaces, we have 
$(\ker d_0)^\perp=\operatorname{Im}\delta_0$.
Consequently, for every degree $k$,
        \begin{align*}
            \Omega^k(G)=&\ker d_0\oplus(\ker d_0)^\perp=\ker d_0\oplus\operatorname{Im}\delta_0\\=&\ker d_0\cap\ker \delta_0\oplus\ker d_0\cap(\ker \delta_0)^\perp\oplus\operatorname{Im}\delta_0\\=&\ker\Box_0\oplus\ker d_0\cap\operatorname{Im}d_0\oplus\operatorname{Im}\delta_0=\ker\Box_0\oplus\operatorname{Im}d_0\oplus\operatorname{Im}\delta_0\,,
        \end{align*}
        where in the last equality we used $d_0^2=0$, and hence $\operatorname{Im}d_0\subseteq\ker d_0$.
    \end{proof}
We now describe the orthogonal projection onto $E_0^\bullet = \ker \Box_0$. This projection will subsequently be used in the definition of the Rumin differential $d_c$. We first observe that, since $(\ker d_0)^\perp = \operatorname{Im} \delta_0$, for every $p,k$, the restriction
\begin{align*}
d_0\big\vert_{\operatorname{Im}\delta_0\cap\Omega^{p,k-1-p}(G)}
\colon
\operatorname{Im}\delta_0\cap\Omega^{p,k-1-p}(G)
\xrightarrow{\ \cong\ }
\operatorname{Im}d_0\cap\Omega^{p,k-p}(G)
\end{align*}
is an isomorphism of $\mathcal{C}^\infty(G)$-modules. Hence, this restriction map is invertible and its inverse defined a partial inverse of $d_0$.
\begin{definition}[The partial inverse $d_0^{-1}$]\label{def: the partial inverse of d0}
    The partial inverse of $d_0$ is the $\C^\infty(G)$-linear map
\begin{align*}
d_0^{-1}\colon
\Omega^{p,k-p}(G)
&\longrightarrow
\Omega^{p,k-1-p}(G)\ ,\
d_0^{-1}
:=
\left(
d_0\big\vert_{\operatorname{Im}\delta_0}
\right)^{-1}
\circ\operatorname{pr}_{\operatorname{Im}d_0}\,,
\end{align*}
where $\operatorname{pr}_{\operatorname{Im}d_0}$ denotes the fibrewise orthogonal projection onto $\operatorname{Im}d_0$.
Thus, on $\operatorname{Im}d_0$, the operator $d_0^{-1}$ coincides with the inverse of $d_0\big|_{\operatorname{Im}\delta_0}$, while it vanishes on $(\operatorname{Im}d_0)^\perp$.
\end{definition}
As an immediate consequence of the definition, we also get that  $(d_0^{-1})^2=0$ and $\ker d_0^{-1}=\ker\delta_0=(\operatorname{Im}d_0)^\perp$.
\begin{definition}[Orthogonal projection onto $E_0^\bullet$]\label{def: Pi_0}
We define
\begin{align*}
\Pi_0
&:=
\operatorname{Id}
-d_0d_0^{-1}
-d_0^{-1}d_0
\colon
\Omega^\bullet(G)
\longrightarrow
\Omega^\bullet(G)\,.
\end{align*}

\end{definition}
\begin{lemma}
    The operator $\Pi_0$ defined above is the fibrewise orthogonal projection onto $E_0^\bullet$, i.e.  $\Pi_0 = \operatorname{pr}_{E_0}$. Moreover, it preserves both weight and degree:
$\Pi_0\bigl(\Omega^{p,k-p}(G)\bigr)
\subseteq
E_0^k\cap\Omega^{p,k-p}(G)$.
\end{lemma}
\begin{proof}
    By Definition \ref{def: the partial inverse of d0}, 
$
d_0d_0^{-1}
=
\operatorname{pr}_{\operatorname{Im}d_0}$ and $
d_0^{-1}d_0
=
\operatorname{pr}_{\operatorname{Im}\delta_0}
$.
Thus, these operators are the fibrewise orthogonal projections onto $\operatorname{Im}d_0$ and $\operatorname{Im}\delta_0$, respectively. Moreover, it follows that $\Pi_0^2=(\operatorname{Id}-d_0d_0^{-1}-d_0^{-1}d_0)(\operatorname{Id}-d_0d_0^{-1}-d_0^{-1}d_0)=\operatorname{Id}-d_0d_0^{-1}-d_0^{-1}d_0=\Pi_0$, hence, $\Pi_0$ is the fibrewise orthogonal projection onto
$
\big(\operatorname{Im}d_0\oplus\operatorname{Im}\delta_0\big)^\perp$. Since $(\operatorname{Im}d_0)^\perp=\ker\delta_0$ and $(\operatorname{Im}\delta_0)^\perp=\ker d_0$,
we obtain $\operatorname{Im}\Pi_0
=
\ker d_0\cap\ker\delta_0
=
E_0^\bullet
$, and so $\Pi_0=\operatorname{pr}_{E_0}$. The second assertion follows directly from the fact that $d_0$ and $d_0^{-1}$ have bidegree $(0,1)$ and $(0,-1)$, respectively.
\end{proof}

We now turn to the Rumin differential $d_c$. To keep the exposition concise, we express $d_c$ directly in terms of the structure maps of the truncated multicomplex. The equivalence between the standard construction introduced by Rumin \cite{rumin_grenoble} and the formulation used here is proved in Lemma 2.19 of \cite{tripaldi2026spectralcomplexestruncatedmulticomplexes}; see also Lemma 3.7 of \cite{magnani2026stokestheorempositivelygraded}.

\begin{definition}[The operators $\partial_r$]\label{def: operators curly}
The operators $\partial_r\colon\Omega^k(G)\to\Omega^{k+1}(G)$ are defined recursively by
  \begin{align}\label{eq: operators curly}
        \partial_1=d_1\ \text{ and }\ \partial_r=d_r-\sum_{j=1}^{r-1}d_{r-j}d_0^{-1}\partial_j\ \text{ for }r\ge 2\,.
    \end{align}
\end{definition}

Since $d_0^{-1}$ has bidegree $(0,-1)$, it follows inductively that $\partial_r$ has bidegree  $\lvert\partial_r\rvert=(r,1-r)$, that is\begin{align}\label{eq: bidegree partial r}
        \text{ for any }\alpha\in\Omega^{p,k-p}(G)\ \text{ we have that }\ \partial_r\alpha\in\Omega^{p+r,k+1-p-r}(G)\,.
    \end{align}
Thus, $\partial_r$ increases the weight by $r$ and the degree by $1$.

\begin{proposition}[Expressing $d_c$ on $E_0^\bullet$]\label{prop: d_c on forms}
For every $\alpha\in E_0^\bullet$, the Rumin differential is given by
\begin{align}\label{eq: formulation of d_c}
d_c\alpha
&=
\Pi_0\sum_{r=1}^{Q}\partial_r\alpha
=
\sum_{r=1}^{Q}\partial_r\alpha
-d_0d_0^{-1}\sum_{r=1}^{Q}\partial_r\alpha
-d_0^{-1}d_0\sum_{r=1}^{Q}\partial_r\alpha\,.
\end{align}
\end{proposition}

Just as the space of Rumin forms decomposes according to weight, namely
\begin{align*}
E_0^k
=
\bigoplus_{p=0}^{Q}
\left(E_0^k\cap\Omega^{p,k-p}(G)\right),
\end{align*}
the Rumin differential admits a corresponding decomposition into components of different bidegree.

More precisely, fix $p$ and $k$ such that $E_0^k\cap\Omega^{p,k-p}(G)\neq\{0\}$.
Since not every space $E_0^{k+1}\cap\Omega^{p+j,k+1-p-j}(G)$ is necessarily nontrivial, let $I_{p,k}=\{\nu_1,\ldots,\nu_N\}\subseteq\mathbb N$, 
denote the set of admissible weight increments. Then the restriction of the Rumin differential to the weight-$p$ component of $E_0^k$ decomposes as
\begin{align}\label{eq: weight decomposition dc}
d_c
=
\sum_{j\in I_{p,k}}d_c^j
\colon
E_0^k\cap\Omega^{p,k-p}(G)
\longrightarrow
\bigoplus_{j\in I_{p,k}}
\left(
E_0^{k+1}\cap
\Omega^{p+j,k+1-p-j}(G)
\right),
\end{align}
where
\begin{align*}
d_c^j
:=
\Pi_0 \circ\partial_j\big|_{E_0^k\cap\Omega^{p,k-p}(G)}\ \text{ and it is not the trivial map for each }j\in I_{p,k}\,.
\end{align*}
Each $d_c^j$ is a differential operator that increases the weight by $j$ and the degree by $1$. In the terminology of multicomplexes, we say that $d_c^j$ has bidegree $|d_c^j|=(j,1-j)$.
More precisely, throughout the paper we use the notation \begin{align}\label{eq: definition Ipk} I_{p,k} := \left\{ j\geq 1: \Pi_0\partial_j \big|_{E_0^k\cap\Omega^{p,k-p}(G)} \neq 0 \right\}. \end{align} Thus, $I_{p,k}$ records precisely the nontrivial homogeneous components of the Rumin differential acting on Rumin forms of degree $k$ and weight $p$.

We now introduce the spectral complexes associated with the de Rham complex of $G$. Their construction combines two structures discussed above: the multicomplex structure of the de Rham complex and the decomposition of the Rumin differential according to the weight. Indeed, since $d =d_0 + d_1 + \cdots + d_s$ endows $\Omega^\bullet(G)$ with the structure of a $s$-multicomplex, its total complex carries a natural filtration by weight and hence gives rise to a spectral sequence \cite{livernet2020spectral,lerario2023multicomplexes}. We briefly recall the description of its pages and differentials that will be used throughout the paper. Following \cite{livernet2020spectral}, we first introduce the spaces of $r$-cycles and $r$-boundaries directly at the level of differential forms.

\begin{definition}[Definition 2.6 in \cite{livernet2020spectral}]\label{def: Z and B defined} Let $\alpha\in\Omega^{p,k-p}(G)$ and let $r\ge 1$. We define bigraded submodules $Z_r^{p,k-p}$ and $B_r^{p,k-p}$ of $\Omega^{p,k-p}$ as follows.
\begin{align*}
    \alpha\in Z_r^{p,k-p}\ \Longleftrightarrow&\ \text{for }1\le j\le r-1\,,\text{ there exists }z_{p+j}\in\Omega^{p+j,k-p-j}(G)\text{ such that}\\&\ d_0\alpha=0\text{ and }d_n\alpha=\sum_{i=0}^{n-1}d_iz_{p+n-i}\text{ for all }1\le n\le r-1\,.\\
    \alpha\in B_r^{p,k-p}\ \Longleftrightarrow&\ \text{for }0\le j\le r-1\text{ there exists }c_{p-j}\in\Omega^{p-j,k-1-p+j}(G)\text{ such that}\\&\ \alpha=\sum_{j=0}^{r-1}d_jc_{p-j}\text{ and }0=\sum_{j=l}^{r-1}d_{j-l}c_{p-j}\text{ for }1\le l\le r-1\,.
\end{align*}
\end{definition}

The $r$-th page of the spectral sequence is therefore represented by the bigraded quotient $E_r^{p,k-p} = Z_r^{p,k-p}/B_r^{p,k-p}$ and is endowed with a differential $\Delta_r$ of bidegree $(r,1-r)$. As usual, the quotients appearing at any given page are obtained by taking the cohomology of the preceding one, i.e. $E_{r+1}\cong H(E_r,\Delta_r)$. 

The differential $\Delta_r$ admits an explicit description in terms of the structure maps of the multicomplex. \begin{proposition}[Theorem 2.10 in \cite{livernet2020spectral}] \label{prop: Delta con diff} The $r$-th differential of the spectral sequence is given by \begin{align*} \Delta_r\colon Z_r^{p,k-p}/B_r^{p,k-p} &\longrightarrow Z_r^{p+r,k+1-p-r}/B_r^{p+r,k+1-p-r}, \\ \Delta_r([\alpha]) &= \left[ d_r\alpha - \sum_{i=1}^{r-1} d_i z_{p+r-i} \right], \end{align*} where $\alpha\in Z_r^{p,k-p}$ and the family $\{z_{p+j}\}_{1\leq j\leq r-1}$ satisfies the equations in Definition \ref{def: Z and B defined}. \end{proposition}

The ordinary spectral sequence only considers quotients in which the indices of the cycle and boundary spaces agree, namely $Z_r/B_r$. The construction of the spectral complexes is based on the observation that cycle and boundary spaces belonging to different pages can also be combined \cite[Lemma 4.1]{tripaldi2026spectralcomplexestruncatedmulticomplexes}.
Moreover, the differential associated with the $r$-th page admits suitable extensions to mixed quotients of the form
\begin{align*}\Delta_r\colon Z_r^{p,k-p}/B_m^{p,k-p} \longrightarrow Z_l^{p+r,k+1-p-r}/B_r^{p+r,k+1-p-r}, \end{align*}
Indeed, using a slight abuse of notation, one has
\begin{equation}\label{eq:image of Z_r by Delta_r}
    \Delta_r\bigl(Z_r^{p,k-p}\bigr) \subseteq B_{r+1}^{p+r,k+1-p-r}, 
\end{equation}
and 
\begin{equation}\label{eq:image of B_r by Delta_r}
    \Delta_r\bigl(B_m^{p,k-p}\bigr)  \subseteq B_r^{p+r,k+1-p-r}
\end{equation}
for every $m\geq 1$ (see proof of \cite[Proposition 4.2]{tripaldi2026spectralcomplexestruncatedmulticomplexes}). What we  mean by inclusions \eqref{eq:image of Z_r by Delta_r} and \eqref{eq:image of B_r by Delta_r} is that, given $\alpha\in Z_r^{p,k-p}$, for any choice of elements $z_{r+i}\in\Omega^{p+i,k-p-i}(G)$ that satisfy the equations in Definition \ref{def: Z and B defined}, we have that $d_r\alpha-\sum_{i=1}^{r-1}d_{i}z_{p+r-i}\in B_{r+1}^{p+r,k+1-p-r}$, and if furthermore $\alpha\in B_m^{p,k-p}$, then $d_r\alpha-\sum_{i=1}^{r-1}d_{i}z_{p+r-i}\in B_{r}^{p+r,k+1-p-r}$. Indeed, when we write
\begin{align}\label{eq: Delta_r on mixed quotients}
\Delta_r\colon E_{r,m}^{p,k-p} \longrightarrow E_{l,r}^{p+r,k+1-p-r}, 
\end{align}
we mean the map induced by 
\begin{equation}\label{eq: def representative-level map del Delta_r}
   \alpha
\longmapsto
d_r\alpha-\sum_{i=1}^{r-1}d_i z_{p+r-i}
\end{equation}
 followed by projection onto the quotient $ Z_l^{p+r,k+1-p-r}/B_r^{p+r,k+1-p-r}$.

We now use the Rumin complex to select, among all these possible mixed quotients and maps, those that assemble into the spectral complexes associated with the given truncated multicomplex. In particular, the collection of the indices $I_{p,k}$ defined in \eqref{eq: definition Ipk}, is used to determine which spectral-sequence differentials $\Delta_j$ are used to connect the corresponding quotients (see \cite[Proposition 3.8]{tripaldi2026spectralcomplexestruncatedmulticomplexes}). We use the following conventions to treat the cases of forms of degree 0 and top degree: $B_l^{0,0}=0$ and $Z_l^{Q,n-Q}=\Omega^n(G)$ for any $l\ge 1$.

\begin{definition}[Spectral complexes associated with $\Omega^\bullet(G)$] \label{def: spectral complexes}
Let $p,k$ be such that \[ E_0^k\cap\Omega^{p,k-p}(G)\neq\{0\}. \] For every pair $j,l\geq1$ satisfying $j\in I_{p,k}$ and $l\in I_{p-l,k-1}$, we define \begin{align}\label{eq: E_{j,l}} E_{j,l}^{p,k-p} := Z_j^{p,k-p}/B_l^{p,k-p}.
\end{align} 
At degree zero we set $B_l^{0,0}=0$ and we always have $j=1$ so that $E_{1,l}^{0,0}=Z_1^{0,0}=C^\infty(G)$, while at top degree $Z_j^{Q,n-Q}=\Omega^{n}(G)$ and $l=1$, so that $E_{j,1}^{Q,n-Q}=\Omega^n(G)$.

Moreover, if $i\in I_{p+j,k+1}$, then the corresponding differentials fit into a sequence 
\begin{align*} E_{l,m_1}^{p-l,k-1-p+l} \xrightarrow{\ \Delta_l\ } E_{j,l}^{p,k-p} \xrightarrow{\ \Delta_j\ } E_{i,j}^{p+j,k+1-p-j} \xrightarrow{\ \Delta_i\ } E_{m_2,i}^{p+j+i,k+2-p-j-i},
\end{align*} and satisfy
\[ \Delta_j\circ\Delta_l=0\ \text{ and } \ \Delta_i\circ\Delta_j=0\,. \] 
Here $m_1$ and $m_2$ denote any indices for which the corresponding neighbouring terms are nontrivial. Since this construction can be carried out in every degree and weight, it yields a family of bigraded spaces $E_{j,l}^{\bullet,\bullet}$ connected by the operators $\Delta_j$ of bidegree $(j,1-j)$. These spaces therefore form a collection of complexes, which we call the spectral complexes associated with the $s$-multicomplex $\bigl(\Omega^\bullet(G),d=d_0+d_1+\cdots+d_s\bigr)$. We denote this family by 
\begin{align*} \left\{ \left(E_{j,l}^{\bullet,\bullet},\Delta_j\right) \right\}_{j\in I_{\bullet,\bullet}}.
\end{align*} 
\end{definition}

\section{Pansu pullback of forms}\label{chapter pansu pullback}

On a Carnot group there are two distinct notions of differentiability: the classical ``Eulidean'' one arising from the smooth manifold structure, and an intrinsic notion determined by the group law together with the stratification of the Lie algebra. The latter, known as Pansu differentiability, is defined as follows.

\begin{definition}[Pansu derivative]\label{def: P deriv}
    Let $G_1$ and $G_2$ be Carnot groups, and denote by $\delta_{\lambda}$ the dilation of factor $\lambda>0$ in both groups. A map $\varphi:G_1 \rightarrow G_2$ is Pansu differentiable at a point $x \in G_1$ if there exists a group homomorphism $D_P\varphi(x): G_1 \rightarrow G_2$ such that 
    \begin{align}\label{Pansu derivative}
         \delta_{1/\lambda} \circ L^{-1}_{\varphi(x)} \circ \varphi \circ L_x \circ \delta_{\lambda}(\cdot) \xrightarrow[\lambda \rightarrow 0^+]{}   D_P\varphi(x)(\cdot),
    \end{align}
 with respect to the uniform convergence on compact sets.
    The map $D_P\varphi(x)$ is called the Pansu derivative of $\varphi$ at $x$.

    Whenever no ambiguity arises, we shall denote by the same symbol $D_P\varphi$ the map induced between the corresponding Lie algebras via the exponential identification.
    \end{definition}

As proved in \cite{Pansu-quasiconformal}, an analogue of Rademacher's theorem holds in this setting.

\begin{theorem}[Pansu-Rademacher Theorem]
    Let $\varphi : G_1 \rightarrow G_2$ be a Lipschitz map between Carnot groups, then $\varphi$ is Pansu differentiable at almost every point $x \in G_1$.
\end{theorem}

Let us recall some fundamental properties of the Pansu derivative.

Let $G_1$ and $G_2$ be Lie groups with stratified Lie algebras $\mathfrak{g}_1 = V_1^1 \oplus \cdots \oplus V_{s_1}^1$ and $\mathfrak{g}_2 = V_1^2 \oplus \cdots \oplus V_{s_2}^2$ and denote by $n_1$ and $n_2$ their (topological) dimension. 
Given $\varphi : G_1 \rightarrow G_2$ a Pansu differentiable map at $x\in G_1$, the first property is that the Pansu derivative  $D_P\varphi(x):G_1 \rightarrow G_2$ is a strata-preserving homomorphism. In particular, it commutes with dilations:
    \begin{equation*}
        D_P\varphi(x)(\delta_{\lambda} v) = \delta_{\lambda} D_P\varphi(x)(v) \quad \text{for all } \lambda>0\,.
    \end{equation*}
Equivalently, the induced map $D_P\varphi(x) : \mathfrak{g}_1 \rightarrow \mathfrak{g}_2$ between the Lie algebras preserves the stratification. Therefore, in coordinates, with respect to two bases adapted to the stratification, $D_P\varphi(x)$ is represented by a block diagonal matrix
    \begin{align*}
        D_P\varphi(x)=\begin{bmatrix}
            A_1^1 & 0 & \cdots  & 0\\
            0 & A_2^2 & \cdots  & 0\\
            \vdots &  & \ddots & \vdots\\
            0 & \cdots & 0 & A_{s_1}^{s_2}\\
            \end{bmatrix}\,,
    \end{align*}
    where for every $i, j$, the block $A_j^i$ represents a linear map from $V_i^1$ to $V_j^2$. Here the maps $A_i^j\colon V_i^1\to V_j^2$ with $A_{i}^j=0$ when $i\neq j$ are represented in matrix form as rectangular blocks.
    
The second important remark is that the mere existence of the limit in \eqref{Pansu derivative} does not automatically guarantee that the limit map is a group homomorphism from $G_1$ to $G_2$. However, if the map $\varphi$ is  Euclidean $\mathcal{C}^1$ on an open set, preservation of the horizontal distribution throughout that open set implied continuous Pansu differentiability. Conversely, Pansu differentiability at every point implies the contact condition (see \cite{Pansu-quasiconformal} and \cite[Theorem 1.1]{Magnani_2013}).
\medskip

From now on, we assume that $\varphi : G_1 \to G_2$ is a $\mathcal{C}^1$ Pansu differentiable map. In this framework both the classical differential $d(\varphi)_{\cdot}$ and the Pansu derivative $D_P\varphi(\cdot)$ are well defined at every point. Our goal is to clarify the relation between these two notions and to derive an explicit coordinate representation of the Pansu derivative.

\begin{remark}
As a consequence of \cite[Theorem 1.1]{Magnani_2013}, every $\mathcal{C}^1$ contact map is Pansu differentiable. Here, a map is said to be contact if its classical differential maps the horizontal subbundle into the horizontal subbundle; see \cite[Definition 2.9]{Magnani_2013}. Therefore, all the results of this section apply, in particular, to $\mathcal{C}^1$ contact maps.
\end{remark}

For $h=1,2$, fix a basis $\{X_1^h, \cdots, X_{n_h}^h\}$ of the Lie algebra $\mathfrak{g}_h$ adapted to the stratification. Identify $\mathfrak{g}_h$ and $\mathbb{R}^{n_h}$ via the isomorphism 
\begin{align*}
    \mathbb{R}^{n_h} &\longleftrightarrow \mathfrak{g}_h\\
    (x_1, \ldots, x_{n_h}) &\longmapsto \sum_{j = 1}^{n_h} x_j X_j^h.
\end{align*}
Through the exponential map, this identification induces global coordinates on the group $G_h$, namely the exponential coordinates of the first kind:
\begin{equation}\label{identification algebra group}
    \begin{aligned}
    \mathbb{R}^{n_h}\quad &\longrightarrow \quad G_h\\
    (x_1, \ldots, x_{n_h}) &\longmapsto \exp \big ( \sum_{j=1}^{n_h} x_j X_j^h\big ).
\end{aligned}
\end{equation}
Let $\phi := (\phi_1, \ldots, \phi_{n_2}) : \mathbb{R}^{n_1} \rightarrow \mathbb{R}^{n_2}$ denote the coordinate representation of $\varphi$ with respect to the exponential coordinates \eqref{identification algebra group}, that is 
\begin{equation*}
    \phi = \exp^{-1} \circ \varphi \circ \exp.
\end{equation*}
Recall that, in such coordinates, the group law on $G_h$ is given by the Baker-Campbell-Hausdorff formula. For all $x,y \in \mathbb R^{n_h}$,
\begin{equation*}
    x*y = \log \Big ( \exp \big ( \sum_{j=1}^{n_h} x_j X_j^h\big ) \exp \big ( \sum_{j=1}^{n_h} y_j X_j^h\big ) \Big ),
\end{equation*}
which defines a polynomial group law on $\mathbb R^{n_h}$.

Fix $x \in \mathbb R^{n_1}$. The Pansu derivative of $\phi$ at $x$, applied to $y \in \mathbb R^{n_1}$, is given by
\begin{equation*}
     D_P\phi(x) (y) =\lim_{\lambda \rightarrow 0^+} \delta_{1/{\lambda}}\circ L_{-\phi(x)} \circ \phi \circ L_x \circ \delta_{\lambda}(y),
 \end{equation*}
where $L_\bullet$ denotes the left multiplication with respect to the group law $*$ on $\mathbb R^{n_h}$, and the $\delta_\lambda$ are the induced homogeneous dilations on $\mathbb{R}^{n_h}$ (we use the same notation on $\mathbb R^{n_1}$ and $\mathbb R^{n_2}$ for simplicity).\\
Set $\widetilde{\phi}:= L_{-\phi(x)} \circ \phi \circ L_x$. 
We have the following commutative diagram
\begin{equation}\label{commutative diagram charts}
    \begin{tikzcd}[ampersand replacement=\&,cramped]
	{\mathbb{R}^{n_1}} \& {\mathbb{R}^{n_1}} \& {\mathbb{R}^{n_2}} \& {\mathbb{R}^{n_2}} \\
	{G_1} \& {G_1} \& {G_2} \& {G_2}
	\arrow["{L_x}", from=1-1, to=1-2]
	\arrow["\exp", from=1-1, to=2-1]
	\arrow["\phi", from=1-2, to=1-3]
	\arrow["\exp", from=1-2, to=2-2]
	\arrow["{L_{-\phi(x)}}", from=1-3, to=1-4]
	\arrow["\exp", from=1-3, to=2-3]
	\arrow["\exp", from=1-4, to=2-4]
	\arrow["{L_{\exp(x)}}"', from=2-1, to=2-2]
	\arrow["\varphi"', from=2-2, to=2-3]
	\arrow["{L_{\exp(-\phi(x))}}"', from=2-3, to=2-4]
\end{tikzcd}
\end{equation}

The map $\widetilde{\phi}$ can be interpreted as the representation of $\varphi$ with respect to the coordinates
\begin{equation*}
    \mathbb{R}^{n_1} \xrightarrow[]{L_x} \mathbb{R}^{n_1}\cong \mathfrak{g}_1 \xrightarrow[]{\exp} G_1 ,
\end{equation*}
and 
\begin{equation*}
    \mathbb{R}^{n_2} \xrightarrow[]{L_{\phi(x)}} \mathbb{R}^{n_2} \cong \mathfrak{g}_2 \xrightarrow[]{\exp} G_2.
\end{equation*}
The classic differential of the map $\varphi$ at the point $\exp(x)$ with respect to these coordinates is represented by $d(\widetilde{\phi})_0$. Differentiating the diagram in \eqref{commutative diagram charts} gives
\[\begin{tikzcd}[ampersand replacement=\&,cramped]
	{T_0\mathbb{R}^{n_1}} \&\& {T_x\mathbb{R}^{n_1}} \&\& {T_{\phi(x)}\mathbb{R}^{n_2}} \&\&\& {T_0\mathbb{R}^{n_2}} \\
	{T_{1_{G_1}}G_1} \&\& {T_{\exp(x)}G_1} \&\& {T_{\exp{\phi(x)}}G_2} \&\&\& {T_{1_{G_2}}G_2}
	\arrow["{d(L_x)_0}", from=1-1, to=1-3]
	\arrow["{d(\exp)_0}"', from=1-1, to=2-1]
	\arrow["{d(\phi)_x}", from=1-3, to=1-5]
	\arrow["{d(\exp)_x}", from=1-3, to=2-3]
	\arrow["{d(L_{-\phi(x)})_{\phi(x)}}", from=1-5, to=1-8]
	\arrow["{d(\exp)_{\phi(x)}}", from=1-5, to=2-5]
	\arrow["{d(\exp)_0}", from=1-8, to=2-8]
	\arrow["{d(L_{\exp(x)})_{1_{G_1}}}"', from=2-1, to=2-3]
	\arrow["{d(\varphi)_{\exp(x)}}"', from=2-3, to=2-5]
	\arrow["{d(L_{\exp(-\phi(x))})_{\exp(\phi(x))}}"', from=2-5, to=2-8]
\end{tikzcd}\]
In particular, 
\begin{equation*}
 d(\widetilde{\phi})_{0} = d(L_{-\phi(x)})_{\phi(x)} \circ d(\phi)_x \circ d(L_x)_0.
\end{equation*}
Let us compute $d(\widetilde{\phi})_{0}$ explicitly. To do so, recall that the fixed basis $\{X_1^h, \dots, X_{n_h}^h\}$ of $\mathfrak{g_h}$ induces left-invariant vector fields on $\mathbb{R}^{n_h}$ via the formula
\begin{equation*}
    x \mapsto d(L_x)_0 e_j\ , \quad j = 1, \dots, n_i
\end{equation*}
where $\{e_j\}$ is the standard basis of $\mathbb{R}^{n_h}$ (see \cite{LeDonneTripaldi2021}). We keep the notation $X_1^h, \dots, X_{n_h}^h$ for these vector fields.
The differential $(d\phi)_x$ is represented, with respect to the standard basis of $\mathbb R^{n_1}$ and $\mathbb R^{n_2}$, by the Jacobian matrix
\begin{equation*}
    J\phi(x) = \Big ( \partial_j \phi_i(x) \Big)_{ \substack{ i=1, \dots, n_2 \\ j=1, \dots, n_1 } }.
\end{equation*}
 Hence the matrix associated with $(d\phi)_x \circ (dL_x)_0: T_0\mathbb{R}^{n_1} \rightarrow T_{\phi(x)}\mathbb{R}^{n_2}$ is  
\begin{equation*}
    \left( (X_j^1\phi_i)(x) \right)_{\substack{i=1,\cdots,n_2\\ j = 1, \cdots,n_1}}.
\end{equation*} 
It remains to compute the differential $d(L_{-\phi(x)})_{\phi(x)}$.
Identifying $\mathfrak g_2$ with $\mathbb R^{n_2}$ via the fixed basis, and using exponential coordinates, one has
\[\begin{tikzcd}[ampersand replacement=\&,cramped]
	{G_2} \& {G_2} \\
	{\mathbb{R}^{n_2}} \& {\mathbb{R}^{n_2}.}
	\arrow["{L_{\exp(-y)}}", from=1-1, to=1-2]
	\arrow["{\exp^{-1}}", from=1-2, to=2-2]
	\arrow["\exp", from=2-1, to=1-1]
	\arrow["{L_{-y}}"', from=2-1, to=2-2]
\end{tikzcd}\]
Differentiating at $y$ gives
\begin{align*}
    d(L_{-y})_y &= d(\exp^{-1} \circ L_{\exp(-y)} \circ \exp)_y    = d(\exp^{-1})_{1_{G_2}} \circ d(L_{\exp(-y)})_{\exp(y)} \circ d(\exp)_y \\
    &= d(L_{\exp(-y)})_{\exp(y)} \circ d(\exp)_y\,.
\end{align*}
Using the standard formula for the differential of the exponential map (see \cite[Theorem 1.7, Chapter II]{Helgason2001}), one obtains
\begin{equation*}
    d(L_{-y})_y = \frac{1 - e^{-\text{ad}(y)}}{\text{ad}(y)} = \sum_{m=0}^{\infty}\frac{(-1)^m}{(m+1)!}(\text{ad}(y))^m\ \text{ where }\text{ad}(y)(z) = [y,z]\,.
\end{equation*}
 Therefore, the Jacobian matrix associated with
 \begin{align*}
     d(\widetilde{\phi})_0: \mathbb{R}^{n_1} = T_0\mathbb{R}^{n_1} \rightarrow T_0\mathbb{R}^{n_2} = \mathbb{R}^{n_2}
 \end{align*}
 has entries
\begin{equation}\label{formula esplicita differenziale classico}
    a_{i,j} = \Pi_{i} \Big ( \sum_{m=0}^{\infty}\frac{(-1)^m}{(m+1)!}(\text{ad}(\phi(x)))^m ( \sum_{k=1}^{n_2} (X_j^1\phi_k)(x) e_k )\Big)
\end{equation}
where $\Pi_i : \mathbb{R}^{n_2}  \to \mathbb{R}$ denotes the projection onto the 
$i$-th coordinate.

On the other hand, 

\begin{align*}
    a_{i,j} = \Pi_i(d(\widetilde{\phi})_0(e_j)) = \Pi_i(d(L_{-\phi(x)} \circ \phi \circ L_x)_0(e_j)),
\end{align*}
and
\begin{align*}
    d(L_{-\phi(x)} \circ \phi \circ L_x)_0(e_j) &= \lim_{\lambda\rightarrow 0} \frac{1}{\lambda} \left ( L_{-\phi(x)} \circ \phi \circ L_x(\lambda e_j) - L_{-\phi(x)} \circ \phi \circ L_x (0)  \right)\\ 
    &=  \lim_{\lambda\rightarrow 0} \frac{1}{\lambda} \left ( L_{-\phi(x)} \circ \phi \circ L_x(\lambda e_j) \right).
\end{align*}

Let $l_j^h \in \mathbb{N}$ such that $\delta_{\lambda}(X_j^h) = \lambda^{l_j^h}X_j^h$. Then the $(i,j)$-th entry of the matrix associated with the Pansu derivative
$D_P\phi(x)$ (with respect to the fixed stratified bases) is given by
\begin{align*}
\Pi_i\bigl(D_P\phi(x)(e_j)\bigr)
&=
\lim_{\lambda\to0^+}
\frac{1}{\lambda^{l_i^2}}
\Pi_i\Bigl(
(L_{-\phi(x)}\circ\phi\circ L_x)
(\lambda^{l_j^1}e_j)
\Bigr)\\
&=
\lim_{\lambda\to0^+}
\frac{1}{\lambda^{l_i^2-l_j^1}}
\cdot
\frac{1}{\lambda^{l_j^1}}
\Pi_i\Bigl(
(L_{-\phi(x)}\circ\phi\circ L_x)
(\lambda^{l_j^1}e_j)
\Bigr)\\
&=
\lim_{\lambda\to0^+}
\frac{1}{\lambda^{l_i^2-l_j^1}}
\left[
\Pi_i\left(
\sum_{m=0}^{\infty}
\frac{(-1)^m}{(m+1)!}
(\operatorname{ad}(\phi(x)))^m
\left(
\sum_{k=1}^{n_2}
(X_j^1\phi_k)(x)e_k
\right)
\right)
+o(1)
\right]\,.
\end{align*}

 Therefore, we distinguish the following three cases.
 \begin{itemize}
     \item[When] $l_j^1 = l_i^2$.\\
     In this case the scaling factor is $\lambda^0$, hence the $(i,j)$-th entry of the Pansu derivative coincides with the corresponding entry of the standard differential in left-translated coordinates:
     \begin{align*}
        \Pi_i( D_P\phi(x) (X_j^1)) =  \Pi_{i} \Big ( \sum_{m=0}^{\infty}\frac{(-1)^m}{(m+1)!}(\text{ad}(\phi(x)))^m ( \sum_{k=1}^{n_2} (X_j^1(x)\phi_k)(x) e_k ) \Big).
    \end{align*}
    This occurs precisely when $X_j^1$ and $X_i^2$ belong to the same layer of their respective stratifications. These entries form the diagonal blocks of the matrix of $D_P\phi(x)$.
     \item[When] $l_j^1 < l_i^2$.\\ 
     Here $l_i^2-l_j^1>0$, so the factor $\lambda^{l_j^1 -l_i^2}$ diverges for $\lambda \rightarrow 0^+$. 
     Since both the standard and the Pansu differentials exist, the only possibility is that
    \begin{equation}\label{contact equation in charts}
        \Pi_{i} \Big ( \sum_{m=0}^{\infty}\frac{(-1)^m}{(m+1)!}(\text{ad}(\phi(x)))^m ( \sum_{k=1}^{n_2} (X_j^1\phi_k)(x) e_k ) \Big) = 0\,.
    \end{equation}
    Moreover, since the Pansu derivative preserves the stratification, $D_P\phi(x)(e_j)\in V_{l_j}^2$, and since $l_i^2>l_j^1$, the $i$-th coordinate must vanish, i.e.
    \begin{align*}
        \Pi_i( D_P\phi(x) (e_j)) = 0\,.
    \end{align*}
    These entries correspond to the blocks below the diagonal.
     \item[When] $l_j^1 > l_i^2$.\\ 
     In this case $l_i^2-l_j^1<0$, hence again
    \begin{align*}
        \Pi_i( D_P\phi(x) (e_j)) = 0\,.
    \end{align*}
    In this case, the calculation imposes no vanishing
condition on the coefficient $a_{i,j}$. These entries correspond to the blocks above the diagonal blocks.
 \end{itemize}

\begin{remark}
The system of first-order partial differential equations in 
\eqref{contact equation in charts}, obtained by letting $X_j^1$ vary in the first layer of the stratification of $V_1$, coincides with the contact equations introduced by Magnani in \cite[Theorem 1.1]{Magnani_2013}.
Moreover, as a consequence of \cite[Theorem 1.1]{Magnani_2013}, the full system obtained by letting $X_j^1$ vary in a basis of $\mathfrak g_1$ is equivalent to the subsystem corresponding only to the vectors in the first layer.
\end{remark}
This discussion can be summarised by the following proposition.
\begin{proposition}[The classical differential of a Pansu differentiable map]
\label{prop:triangularity-classical-differential}
Let \(G_1\) and \(G_2\) be Carnot groups with stratified Lie algebras
\[
\mathfrak g_1=V_1^1\oplus\cdots\oplus V_{s_1}^1
\ \text{ and }\ 
\mathfrak g_2=V_1^2\oplus\cdots\oplus V_{s_2}^2\,.
\]
Let $\phi\colon\mathbb R^{n_1}\longrightarrow\mathbb R^{n_2}$
be the expression in exponential coordinates of a Euclidean
\(\mathcal C^1\) map
\(\varphi:G_1\to G_2\) and assume \(\varphi\) is Pansu
differentiable at $\exp(x)\in G_1$.

Set
\[
\widetilde\phi
:=
L_{-\phi(x)}\circ\phi\circ L_x
\]
and let
\[
a_{i,j}
:=
\Pi_i\bigl(d(\widetilde\phi)_0(e_j)\bigr),
\]
where \(e_j\) belongs to the $l_j^1$-th layer of $\mathfrak{g}_1$, while
the \(i\)-th coordinate in the target belongs to the $l_i^2$-th layer.

Then
\[
a_{i,j}=0
\qquad\text{whenever}\qquad
l_i^2>l_j^1\,.
\]
Moreover,
\[
\Pi_i\bigl(D_P\phi(x)(e_j)\bigr)
=
\begin{cases}
a_{i,j}\ , & l_i^2=l_j^1\,,\\[2mm]
0\ ,       & l_i^2\neq l_j^1\,.
\end{cases}
\]
Equivalently, the classical differential in left-translated
coordinates satisfies
\[
d(\widetilde\phi)_0(V_q^1)
\subseteq
\bigoplus_{p\leq q}V_p^2\ \text{ for every }q=1,\ldots,s_1\,,
\] and
\[
D_P\phi(x)\big|_{V_q^1}
=
\operatorname{pr}_{V_q^2}
\circ d(\widetilde\phi)_0\big|_{V_q^1}.
\]
Thus, with respect to the bases adapted to the stratifications,
the matrix representing \(d(\widetilde\phi)_0\) is block upper
triangular, and the Pansu derivative \(D_P\phi(x)\) is given
precisely by the block diagonal part.
\end{proposition}

\begin{proof}
Since $\widetilde\phi(0)=0$ and we are assuming \(\widetilde\phi\) is Euclidean differentiable at \(0\), for every
\(e_j\) we have
\[
\widetilde\phi(\lambda^{l_j^1}e_j)
=
\lambda^{l_j^1}d(\widetilde\phi)_0(e_j)
+
o(\lambda^{l_j^1})
\ \text{ as }\ \lambda\to0^+.
\]
Taking the \(i\)-th coordinate gives
\begin{equation}
\label{eq:euclidean-expansion-coordinate}
\Pi_i\bigl(\widetilde\phi(\lambda^{l_j^1}e_j)\bigr)
=
\lambda^{l_j^1}a_{i,j}
+
o(\lambda^{l_j^1})\,.
\end{equation}
On the other hand, by Pansu differentiability,
\[
\delta_{1/\lambda}\circ\widetilde\phi\circ\delta_\lambda(e_j)
\longrightarrow
D_P\phi(x)(e_j)\ 
\text{ as }\ \lambda\to0^+.
\]
Since \(e_j\) has weight \(l_j^1\), we have $\delta_\lambda(e_j)=\lambda^{l_j^1}e_j$.
Moreover, the dilation of \(G_2\) acts on the \(i\)-th coordinate by
multiplication by \(\lambda^{l_i^2}\). Consequently,
\begin{equation}
\label{eq:pansu-component-limit}
\lambda^{-l_i^2}
\Pi_i\bigl(\widetilde\phi(\lambda^{l_j^1}e_j)\bigr)
\longrightarrow
\Pi_i\bigl(D_P\phi(x)(e_j)\bigr)\,.
\end{equation}
The Pansu derivative commutes with the homogeneous dilations and
therefore preserves the stratifications, that is $D_P\phi(x)(V_q^1)\subseteq V_q^2$.
It follows that
\begin{equation}
\label{eq:off-diagonal-pansu}
\Pi_i\bigl(D_P\phi(x)(e_j)\bigr)=0
\ \text{ whenever }\ 
l_i^2\neq l_j^1\,.
\end{equation}
One should then distinguish the following three cases.

Suppose first that \(l_i^2>l_j^1\). By
\eqref{eq:pansu-component-limit} and
\eqref{eq:off-diagonal-pansu},
\[
\Pi_i\bigl(\widetilde\phi(\lambda^{l_j^1}e_j)\bigr)
=
o(\lambda^{l_i^2})\,.
\]
Combining this with
\eqref{eq:euclidean-expansion-coordinate}, we obtain $\lambda^{l_j^1}a_{i,j}+o(\lambda^{l_j^1})
=
o(\lambda^{l_i^2})$, so when
dividing by \(\lambda^{l_j^1}\) we get
\[
a_{i,j}+o(1)
=
o\bigl(\lambda^{l_i^2-l_j^1}\bigr)\,.
\]
Since \(l_i^2-l_j^1>0\), letting \(\lambda\to0^+\) yields $a_{i,j}=0$.

Suppose next that \(l_i^2=l_j^1\). Dividing
\eqref{eq:euclidean-expansion-coordinate} by \(\lambda^{l_i^2}\), we find $
\lambda^{-l_i^2}
\Pi_i\bigl(\widetilde\phi(\lambda^{l_j^1}e_j)\bigr)
=
a_{i,j}+o(1)$.
Comparing this with \eqref{eq:pansu-component-limit} gives
\[
\Pi_i\bigl(D_P\phi(x)(e_j)\bigr)=a_{i,j}\,.
\]

Finally, if \(l_i^2<l_j^1\), then
\eqref{eq:euclidean-expansion-coordinate} gives $\lambda^{-l_i^2}
\Pi_i\bigl(\widetilde\phi(\lambda^{l_j^1}e_j)\bigr)
=
\lambda^{l_j^1-l_i^2}\bigl(a_{i,j}+o(1)\bigr)$.
Since \(l_j^1-l_i^2>0\), the right-hand side converges to zero.
Therefore,
\[
\Pi_i\bigl(D_P\phi(x)(e_j)\bigr)=0\,.
\]
Notice that, unlike the case when $l_i^2>l_j^1$, no condition is imposed on \(a_{i,j}\).

We have therefore proved that $a_{i,j}=0$ whenever $l_i^2>l_j^1$, and that the Pansu derivative retains precisely the entries for which $l_i^2=l_j^1$. This is equivalent to
\[
d(\widetilde\phi)_0(V_q^1)
\subseteq
\bigoplus_{p\leq q}V_p^2
\ \text{ and }\ 
D_P\phi(x)\big|_{V_q^1}
=
\operatorname{pr}_{V_q^2}
\circ d(\widetilde\phi)_0\big|_{V_q^1}\,.
\]
\end{proof}

\medskip
In view of the relation between the classical differential and the Pansu derivative established above, it is natural to ask whether, for $\mathcal{C}^1$  Pansu differentiable maps, the two differentials actually coincide.
The answer is in general negative. We present below an explicit example.
\begin{example}\label{example comparison Pansu and classical differential}
    Let $\mathfrak{h}^1 \times \mathbb{R}$ denote the direct product of the first Heisenberg Lie algebra $\mathfrak{h}^1$, generated by $\{X_1,X_2,T\}$ with non-trivial Lie bracket $[X_1,X_2]=T$, and the one-dimensional abelian Lie algebra spanned by $X_3$. In particular, the only non-vanishing bracket relation is $[X_1,X_2]=T$. Hence $\mathfrak{h}^1 \times \mathbb{R}$ is a two-step stratified Lie algebra with first layer $V_1 = \mathrm{span}\{X_1,X_2,X_3\}$ and second layer $V_2 = \mathrm{span}\{T\}.$
    Let $\mathbb{H}^1 \times \mathbb{R}$ be the connected, simply connected four-dimensional Lie group associated with $\mathfrak{h}^1 \times \mathbb{R}$. We use exponential coordinates of the first kind with respect to the ordered basis $(X_1,X_2,X_3,T)$ to identify $\mathbb{H}^1 \times \mathbb{R}$ with $\mathbb{R}^4$. We keep the same notation $(X_1,X_2,X_3,T)$ for the left invariant vector fields induced in $\mathbb{R}^4$ (see \cite{LeDonneTripaldi2021} for the explicit formulae).\\
    Define the map $\varphi : \mathbb{H}^1 \times \mathbb{R} \to \mathbb{H}^1 \times \mathbb{R}$ by
\begin{equation*}
\varphi(x_1,x_2,x_3,t)
=
\Big(e^{t}x_1,\; e^{t}x_2,\; x_3,\; \frac{1}{2} e^{2t}\Big).
\end{equation*}
The map $\varphi$ is smooth and Pansu differentiable at every point. Indeed, writing $\varphi=(\varphi_1,\varphi_2,\varphi_3,\varphi_4)$, one verifies easily that the contact equations
   \begin{equation*}
\begin{cases}
    X_1\varphi_4 + \frac{1}{2}(\varphi_2 X_1\varphi_1 - \varphi_1 X_1\varphi_2) = 0 \\
    X_2\varphi_4 + \frac{1}{2}(\varphi_2 X_2\varphi_1 - \varphi_1 X_2\varphi_2) = 0\\
    X_3\varphi_4 + \frac{1}{2}(\varphi_2 X_3\varphi_1 - \varphi_1 X_3\varphi_2) = 0
\end{cases}.
\end{equation*}
are satisfied (see \cite[Theorem 1.1]{Magnani_2013}).

According to formula \eqref{formula esplicita differenziale classico}, the entries $(i,4)$ for $i=1,2,3$ of the classical differential with respect to the basis $(X_1(x), X_2(x), X_3(x), T(x))$ and $(X_1(\varphi(x)), X_2(\varphi(x)), X_3(\varphi(x)), T(\varphi(x)))$ are given by
\begin{equation*}
    a_{i,4} = T\varphi_i.
\end{equation*}
A direct computation shows that outside the plane $(0,0,x_3,t)$, we have $(a_{1,4},a_{2,4},a_{3,4}) \not= (0,0,0)$. On the other side, the corresponding entries of $D_P\varphi(x)$ are zero for all $x \in \mathbb{H}^1 \times \mathbb{R}$.
\end{example}

\medskip
We now come to a key object in our analysis, the Pansu pullback. Even though all the considerations done so far regarding Pansu differentiability only require a map $\varphi$ which is Euclidean $\C^1$, when considering the pullback of smooth forms, we will always assume Euclidean smoothness.

The Pansu derivative $D_P\varphi(x)$ can be regarded either as a homogeneous group homomorphism between $G_1$ and $G_2$, or, via the exponential identification,as a graded Lie algebra homomorphism $D_P\varphi(x) : \mathfrak{g}_1 \rightarrow \mathfrak{g}_2$.
This algebraic interpretation allows us to transport smooth forms from $G_2$ to $G_1$ as follows (see also \cite[Definition 3.14]{kleiner2020pansu}). 

\begin{definition}[Pansu pullback]\label{definition pansu pullback}
Let $\alpha = \sum_j f_j \,\xi_j$ be a $k$-form on $G_2$, where $f_j \in \mathcal{C}^{\infty}(G_2)$ for all $j$, and $\xi_j$ are left-invariant $k$-forms. The Pansu pullback of $\alpha$ by $\varphi$ is the $k$-form on $G_1$ defined by
\begin{equation*}
    \varphi_P^*\alpha(x)(X_1, \cdots, X_k) = \sum_{j} (f_j \circ \varphi )(x)  \xi_j (D_P\varphi(x) X_1, \cdots, D_P\varphi(x) X_k)
\end{equation*}
for every $x \in G_1$ and $X_1, \cdots, X_k \in \mathfrak{g}_1$.
\end{definition}
This definition mimics the classical pullback of differential forms, namely
\begin{equation*}
    \varphi^*\alpha(x)(X_1, \cdots, X_k) = \sum_{j} (f_j \circ \varphi )(x) \xi_j ((d\varphi)_x X_1(x), \cdots, (d\varphi)_x X_k(x))
\end{equation*}
with the classical differential $(d\varphi)_x$ being replaced by the Pansu derivative $D_P\varphi$.

Although the two definitions are formally analogous, their behaviour
with respect to weights is different. The classical pullback need not
preserve each individual weight space: it may produce components of
strictly larger weight. Nevertheless, it cannot decrease weight and
therefore preserves the filtration by weight. The Pansu pullback, on
the other hand, preserves each weight space and coincides precisely
with the weight-preserving part of the classical pullback.

\begin{proposition}[Weight behaviour of the classical and Pansu pullbacks]
\label{prop:weight-behaviour-pullbacks}
Let \(G_1\) and \(G_2\) be Carnot groups, and let $\varphi:G_1\longrightarrow G_2$
be a Euclidean smooth map which is Pansu differentiable at
every point.
Then,
\begin{equation*}
    \varphi^* \big(\Omega^{\geq p, k}(G_2)\big) \subset \Omega^{\geq p, k} (G_1) \quad \text{and} \quad \varphi_P^*\bigl(\Omega^{p,k-p}(G_2)\bigr)
\subseteq
\Omega^{p,k-p}(G_1).
\end{equation*}
Moreover, for each $h=1,2$, let
\[
\pi_p^h:\Omega^k(G_h)\longrightarrow\Omega^{p,k-p}(G_h)
\]
denote the projection onto the component of weight \(p\). Then, for every
\(\alpha\in\Omega^{p,k-p}(G_2)\), one has $$\varphi_P^*\alpha
=
\pi_p^1\bigl(\varphi^*\alpha\bigr)\,.$$ 
Therefore, the Pansu pullback of a form of pure weight \(p\) is precisely the
weight-\(p\) component of its classical pullback
\[
\varphi_P^*
=
\sum_p\pi_p^1\circ\varphi^*\circ\pi_p^2\,.
\]
In this sense, the Pansu pullback is the weight-preserving part of the
classical pullback.
\end{proposition}

\begin{proof}
Fix \(x\in G_1\). Using left translations, write
\[
A_x:=d(\widetilde\phi_x)_0:
\mathfrak g_1\longrightarrow\mathfrak g_2
\]
for the classical differential in left-translated coordinates, where $\widetilde\phi_x
:=
L_{-\phi(x)}\circ\phi\circ L_x$.
By Proposition~\ref{prop:triangularity-classical-differential}, for
every \(q\) one has
\[
A_x(V_q^1)\subseteq\bigoplus_{r\leq q}V_r^2\
\text{ and }
\
D_P\varphi(x)\big|_{V_q^1}
=
\operatorname{pr}_{V_q^2}
\circ A_x\big|_{V_q^1}.
\]

We first consider the induced maps on covectors. Let $\theta\in(V_p^2)^\ast$, then if \(v\in V_q^1\) and \(q<p\), then
\[
A_xv\in\bigoplus_{r\leq q}V_r^2
\subseteq
\bigoplus_{r<p}V_r^2\ \text{ which implies that }\ 
(A_x^*\theta)(v)=\theta(A_xv)=0\,.
\]
In other words \(A_x^*\theta\) has no components of weight strictly
smaller than \(p\), and hence $A_x^*\theta
\in
\bigoplus_{q\geq p}(V_q^1)^*$.
Furthermore, the component of \(A_x^*\theta\) of weight \(p\) is
exactly \(D_P\varphi(x)^*\theta\). Indeed, if \(v\in V_p^1\), then
\[
(D_P\varphi(x)^*\theta)(v)
=
\theta(D_P\varphi(x)v)
=
\theta\bigl(\operatorname{pr}_{V_p^2}(A_xv)\bigr)
=
\theta(A_xv)
=
(A_x^*\theta)(v)\,.
\]
On the other hand, \(D_P\varphi(x)^*\theta\) vanishes on \(V_q^1\) whenever
\(q\neq p\), because \(D_P\varphi(x)\) preserves the layers. Consequently,
\begin{equation}
\label{eq:dual-diagonal-part}
D_P\varphi(x)^*\theta
=
\pi_p^1(A_x^*\theta)\,.
\end{equation}

Now let $\xi=\theta_1\wedge\cdots\wedge\theta_k\in\bigwedge^k\mathfrak{g}_2^\ast$
be a homogeneous \(k\)-covector
where $\theta_j\in(V_{p_j}^2)^*$, so that $w(\xi)=
p=p_1+\cdots+p_k$.
For every \(j\), the covector \(A_x^*\theta_j\) has only components of
weight greater than or equal to \(p_j\), and so
\[
A_x^*\xi
=
A_x^*\theta_1\wedge\cdots\wedge A_x^*\theta_k\in\bigoplus_{ r\ge p}\bigwedge\nolimits^{r,k-r}\mathfrak{g}_1^\ast\,.
\]

Its component of weight exactly \(p\) can only be obtained by taking,
in each term \(A_x^*\theta_j\), the component of weight \(p_j\).
Using \eqref{eq:dual-diagonal-part}, we obtain
\[
\pi_p^1(A_x^*\xi)
=
D_P\varphi(x)^*\theta_1\wedge\cdots\wedge D_P\varphi(x)^*\theta_k
=
D_P\varphi(x)^*\xi\,.
\]

By linearity, the same conclusion holds for every homogeneous
\(k\)-covector \(\xi\) of weight \(p\). Finally, multiplication by the
coefficient functions of a differential form does not affect its
weight, and so for every
\(\alpha\in\Omega^{p,k-p}(G_2)=\C^\infty(G_2)\otimes\bigwedge^{p,k-p}\mathfrak{g}_2^\ast\),
\[
\varphi^*\alpha
\in
\bigoplus_{r\geq p}\Omega^{r,k-r}(G_1)\ \text{ and }\
\pi_p^1(\varphi^*\alpha)
=
\varphi_P^*\alpha\,.
\]

The filtration-preserving property follows immediately by summing
over all weights \(r\geq p\). Finally, applying the preceding identity
to every homogeneous component
\(\alpha_p=\pi_p^2\alpha\) gives
\[
\varphi_P^*\alpha
=
\sum_p\pi_p^1\bigl(\varphi^*\alpha_p\bigr)\,.
\]
\end{proof}

\medskip
A second fundamental difference concerns the interaction with the exterior derivative. The classical pullback satisfies the naturality property of the exterior derivative, that is
\begin{equation*}
    d\varphi^*\alpha = \varphi^*d\alpha
\end{equation*}
for every smooth differential form $\alpha$. In contrast, the Pansu pullback does not, in general, commute with the exterior derivative.

\begin{example}\label{example not commutativity}
    Let $\mathbb{H}^1$ be the $3$-dimensional Heisenberg group and denote by $\mathfrak{h}^1 = V_1 \oplus V_2$ its stratified Lie algebra. Let $\{X_1, X_2\}$ be a basis of the first layer $V_1$ and set $T:= [X_1, X_2]$. Denote by $\{\theta_1, \theta_2, \tau\}$ the corresponding dual left-invariant $1$-forms.
    
    Let $\varphi : \mathbb{H}^1 \rightarrow \mathbb{H}^1$ be a $\mathcal{C}^1$ Pansu differentiable map. Identifying $\mathbb{H}^1$ with its Lie algebra via the exponential map, we write $\varphi = (\varphi_1, \varphi_2, \varphi_3)$. 
    Since $\varphi$ is Pansu differentiable, its components satisfy the system of PDEs
    \begin{equation*}
\begin{cases}
    X_1\varphi_3 + \frac{1}{2}(\varphi_2 X_1\varphi_1 - \varphi_1 X_1\varphi_2) = 0 \\
    X_2\varphi_3 + \frac{1}{2}(\varphi_2 X_2\varphi_1 - \varphi_1 X_2\varphi_2) = 0
\end{cases}.
\end{equation*}
    The Pansu derivative  has the form
\begin{align*}
D_P\varphi=\begin{bmatrix}
            X_1\varphi_1 & X_2\varphi_1 & 0\\
        X_1\varphi_2 & X_2\varphi_2 & 0 \\ 0 & 0 & T\varphi_3 + \frac{1}{2}(\varphi_2 T\varphi_1 - \varphi_1 T\varphi_2)\end{bmatrix}\,.
    \end{align*}
Let $g\in \Omega^0(G)$ be a $0$-form. Then
    \begin{align*}
        &\varphi_P^* dg = \varphi_P^* (X_1g \theta_1 + X_2g \theta_2 + Tg \tau) =\\ 
        &=(X_1g \circ \varphi)(X_1\varphi_1 \theta_1 + X_2\varphi_1 \theta_2) + (X_2g \circ \varphi) (X_1\varphi_2 \theta_1 + X_2\varphi_2 \theta_2) + (Tg \circ \varphi) \left(T\varphi_3 + \frac{1}{2}(\varphi_2 T\varphi_1 - \varphi_1 T\varphi_2)\right)\tau
    \end{align*}
    On the other hand, 
    \begin{equation*}
        d\varphi_P^* g = X_1(g \circ \varphi) \,\theta_1 + X_2(g \circ \varphi) \,\theta_2 + T(g \circ \varphi)\, \tau
    \end{equation*}
    To compute derivatives of compositions along the vector fields $X_i$, we use the chain rule
    \begin{align*}
    X_i(g\circ f)(x)=\sum_{j=1}^3\partial_jg(f(x)) X_if_j(x)\,.
\end{align*}
Using the contact equations, we obtain for $i=1,2$
\begin{align*}
    X_i(g \circ \varphi)(x) &= \partial_1g(\varphi(x)) X_i\varphi_1(x)+\partial_2g(\varphi(x)) X_i\varphi_2(x)+\partial_3g(\varphi(x)) X_i\varphi_3(x)\\ 
    &= \partial_1g(\varphi(x))X_i\varphi_1(x)+\partial_2g(\varphi(x)) X_i\varphi_2(x) + \frac{1}{2}\partial_3g(\varphi(x))\left ( \varphi_1 X_i\varphi_2 - \varphi_2 X_i\varphi_1 \right )\\
    &= X_1g(\varphi(x))X_i\varphi_1(x)+X_2g(\varphi(x)) X_i\varphi_2(x).
\end{align*}
In contrast
\begin{align*}
    T(g \circ \varphi)(x) &= \partial_1g(\varphi(x)) T\varphi_1(x)+\partial_2g(\varphi(x)) T\varphi_2(x)+\partial_3g(\varphi(x)) T\varphi_3(x)\\ 
    &= X_1g(\varphi(x))T\varphi_1(x)+X_2g(\varphi(x)) T\varphi_2(x) + Tg(\varphi(x))\left( T\varphi_3 + \frac{1}{2}(\varphi_2 T\varphi_1 - \varphi_1 T\varphi_2) \right).
\end{align*}
Therefore,
\begin{equation*}
      d \varphi_P^* g - \varphi_P^* dg = \left(X_1g(\varphi(x))T\varphi_1(x)+X_2g(\varphi(x)) T\varphi_2(x)\right) \tau.
\end{equation*}
In particular, for suitable choices of $g$ and $\varphi$, the difference $d \varphi_P^* g - \varphi_P^* dg $ does not vanish in general.
\end{example}

Nevertheless, a result due to Kleiner, Müller, and Xie establishes a weak form of commutativity between the exterior derivative and the Pansu pullback. In \cite{kleiner2020pansu}, the authors prove a distributional commutation formula under suitable regularity assumptions. In order to state their result, we need to introduce the weight of a non-homogeneous form (as opposed to the one given in Definition \ref{def: weights of forms}).
\begin{definition}
    If a $k$-form $\alpha$ is not homogeneous with respect to weight, but instead decomposes as a sum of homogeneous components of different weights, we adopt the same convention as in \cite{kleiner2020pansu}. Namely,
     \begin{align*}
         \text{ if }\alpha=\sum_{j\in J}\alpha_j \text{ where }\alpha_j\in\Omega^{j,k-j}(G)\ \text{ then }w(\alpha)=\min J\,.
     \end{align*}
     In other words, the weight of $\alpha$ is defined as the minimal weight among its nonvanishing homogeneous components.
\end{definition}

\begin{theorem}[Theorem 4.2 in \cite{kleiner2020pansu}]\label{theorem KMX}
    Let $U_1 \subset G_1$ and $U_2 \subset G_2$ be open sets, and let $Q_1$ denote the homogeneous dimension of $G_1$. Suppose that $\varphi : U_1 \rightarrow U_2$ is a $W_{\text{loc}}^{1,q}$-map for some $q>Q_1$. Let $\alpha \in \Omega^k(U_2)$ be a continuous $k$-form with continuous distributional exterior derivative $d\alpha$ and let $\eta \in \Omega^{n_1 - k -1}(U_1)$ be a smooth compactly supported form. Assume that
    \begin{equation}\label{condizione pesi KMX}
       \min \left\{ w(\alpha) + w(d\eta), w(d\alpha) + w(\eta) \right\} \geq Q_1.
    \end{equation}
    Then the following identity holds:
    \begin{equation}
        \int_{U_1} \varphi^*_P d\alpha \wedge \eta + (-1)^{k} \int_{U_1}\varphi^*_P\alpha \wedge d\eta = 0.
    \end{equation}
\end{theorem}

\begin{remark}
    In the specific setting of Example \ref{example not commutativity}, the weight constraint implies that admissible test forms must have the form $\eta = f_1 \theta_1 \wedge \tau + f_2 \theta_2 \wedge \tau$ for some functions $f_1, f_2 \in \mathcal{C}^{\infty}_c(\mathbb{H}^1)$, so that the vertical component associated with $\tau$ is effectively not detected.
\end{remark}

\medskip
We conclude this section by briefly discussing the commutativity between the Pansu pullback and the Rumin differential $d_c$ associated with the Rumin complex $(E_0^\bullet,d_c)$.

In this setting, only partial results are available. Weak commutativity has been established for Heisenberg groups $\mathbb H^n$ for all $n$ (see \cite{kleiner2021sobolevmappingsrumincomplex}). As for strong (pointwise) commutativity, once we assume enough Euclidean regularity, it is known to hold in the Heisenberg case. Indeed, this is a direct consequence of the commutativity between the standard pullback of a smooth contact map and the Rumin differentials in Heisenberg groups \cite{canarecci2019insightrumincohomologyorientability}. Moreover in \cite[Theorem 3.16]{FranchiTesi}, the authors prove that the classical pullback of a homogeneous group homomorphism whose transpose is a homogenous group homomorphism commutes with the Rumin differential on Carnot groups. In this case, the classical pullback and the Pansu pullback coincide (see \eqref{formula esplicita differenziale classico}), and therefore also the Pansu pullback of a homogenous group homomorphism whose transpose is a homogenous group homomorphism commutes with the Rumin differential. However, for arbitrary Pansu differentiable maps between Carnot groups where this transpose property fails, commutativity does not hold as shown in the example below.

\begin{example}\label{example dc not commutes with Pansu pullback}
    Let $\mathbb{H}^1 \times \mathbb{R}$ be the Carnot group associated with the Lie algebra $\mathfrak{h}^1 \times \mathbb{R}$ as in Example \ref{example comparison Pansu and classical differential}.\\
    Let $\{X_1,X_2,X_3,T\}$ be a basis of  $\mathfrak{h}^1 \times \mathbb{R}$ and denote by $\{\theta_1, \theta_2, \theta_3, \tau\}$ its dual basis.
    The Rumin complex associated with $\mathbb{H}^1 \times \mathbb{R}$ is 
    \begin{align*}
        \mathcal{C}^{\infty}(\mathbb{H}^1 \times \mathbb{R}) \xlongrightarrow{d_c} \text{span}_{\mathcal{C}^{\infty}(\mathbb{H}^1 \times \mathbb{R})}\{\theta_i \,|\, i=1,2,3\} \xlongrightarrow{d_c} \text{span}_{\mathcal{C}^{\infty}(\mathbb{H}^1 \times \mathbb{R})}\{\theta_i \wedge\theta_3, \theta_i \wedge \tau\,|\,i=1,2\} \longrightarrow \,\, \cdots\\
        \cdots \,\,\, \xlongrightarrow{d_c} \text{span}_{\mathcal{C}^{\infty}(\mathbb{H}^1 \times \mathbb{R})}\{\theta_i\wedge \theta_j \wedge \tau\,|\,1\leq i <j\leq 3\} \xlongrightarrow{d_c} \text{span}_{\mathcal{C}^{\infty}(\mathbb{H}^1 \times \mathbb{R})}\{\theta_1 \wedge \theta_2 \wedge \theta_3 \wedge \tau\}.
    \end{align*}

    Here, the differential $d_c$ has a nontrivial and rather involved definition (see \cite[Theorem 3.11]{FranchiTesi}). However, for the purposes of the computations below, we only need its explicit action on certain $1$-forms. More precisely, for every $f \in \mathcal{C}^{\infty}(\mathbb{H}^1 \times \mathbb{R})$,
    \begin{align*}
        d_c(f \,\theta_2) = -X_3f\, \theta_2 \wedge \theta_3 + X_1^2f \,\theta_1 \wedge \tau + (X_2X_1f-Tf)\,\theta_2 \wedge \tau 
    \end{align*}
    and
    \begin{align*}
        d_c(f\,\theta_3) = X_1f \, \theta_1 \wedge \theta_3 + X_2f \, \theta_2 \wedge \theta_3.
    \end{align*}
     Define the map $\varphi : \mathbb{H}^1 \times \mathbb{R} \to \mathbb{H}^1 \times \mathbb{R}$ by
\begin{equation*}
\varphi(x_1,x_2,x_3,t)
=
\Big(x_1,\; x_2,\; x_2+x_3,\; t\Big),
\end{equation*}
    where we identify $\mathbb{H}^1 \times \mathbb{R}$ with $\mathbb{R}^4$ via the exponential coordinates associated with the basis $\{X_1, X_2, X_3, T\}$. We use the same notation $\{X_1, X_2, X_3,T\}$ for the corresponding left-invariant vector fields in exponential coordinates (see \cite{LeDonneTripaldi2021} for the explicit formulae).
    A straightforward computation shows that $\varphi = (\varphi_1,\varphi_2, \varphi_3, \varphi_4)$ satisfies the contact equations
     \begin{equation*}
\begin{cases}
    X_1\varphi_4 + \frac{1}{2}(\varphi_2 X_1\varphi_1 - \varphi_1 X_1\varphi_2) = 0 \\
    X_2\varphi_4 + \frac{1}{2}(\varphi_2 X_2\varphi_1 - \varphi_1 X_2\varphi_2) = 0\\
    X_3\varphi_4 + \frac{1}{2}(\varphi_2 X_3\varphi_1 - \varphi_1 X_3\varphi_2) = 0
\end{cases}.
\end{equation*}
Hence $\varphi$ is Pansu differentiable by \cite[Theorem 1.1]{Magnani_2013} with Pansu derivative $D_P\varphi$ given by
 \begin{align*}
        D_P\varphi=\begin{bmatrix}
        1 & 0 & 0& 0\\
        0 & 1 & 0 & 0\\ 
        0 & 1 & 1 & 0\\ 
        0 & 0 & 0 & 1\end{bmatrix}\,,
    \end{align*}
    so that, 
    \begin{align*}
        \varphi_P^*\theta_1 =\theta_1, \quad \varphi_P^*\theta_2 =\theta_2, \quad \varphi_P^*\theta_3 =\theta_2 + \theta_3, \quad \varphi_P^*\tau =\tau.
    \end{align*}
    We now compare the two operators $d_c \circ \varphi_P^*$ and $\varphi_P^* \circ d_c$ acting on the Rumin $1$-form $f\theta_3$.
    On the one hand,
\begin{align*}
    d_c \circ \varphi_P^*(f\theta_3) =& d_c(f\circ\varphi \, \theta_2 + f \circ \varphi \,\theta_3)
    = X_1(f\circ \varphi)\, \theta_1 \wedge \theta_3 + (X_2(f\circ \varphi) - X_3(f \circ \varphi))\, \theta_2 \wedge \theta_3 +\\ &+ X_1^2(f\circ \varphi)\,\theta_1 \wedge \tau + (X_2X_1(f \circ \varphi) - T(f \circ \varphi))\,\theta_2 \wedge \tau = X_1f\circ \varphi \,\theta_1 \wedge \theta_3 +\\&+ X_2f \circ \varphi \, \theta_2 \wedge \theta_3 + (X_2X_1f \circ \varphi + X_3X_1f \circ \varphi - Tf \circ \varphi)\, \theta_2 \wedge \tau+ X_1^2f \circ \varphi \, \theta_1 \wedge \tau
    ,
\end{align*}
where in the last step we used the chain rule.
On the other hand
\begin{align*}
    \varphi_P^*\circ d_c(f \,\theta_3) &=\varphi_P^*(X_1f \, \theta_1 \wedge \theta_3 + X_2f \, \theta_2 \wedge \theta_3)\\ &= X_1f \circ \varphi \, \theta_1 \wedge \theta_2 + X_1f \circ \varphi \, \theta_1 \wedge \theta_3 + X_2f\circ \varphi \, \theta_2 \wedge \theta_3.
\end{align*}
Thus the two expressions cannot coincide in general. For instance, choosing $f(x) = x_1^2$, one obtains explicit formulas showing that the component $\theta_1 \wedge \tau$ do not vanish in the first expression, while it is absent in the second. Similarly the component $\theta_1 \wedge \theta_2$ non-zero in the second expression is absent in the first.

Moreover because of the presence of the component $\theta_1 \wedge \theta_2$  in the Pansu pullback of the Rumin form $d_c(f \theta_3)$ we have that the pullback of a Rumin form is not, in general, a Rumin form.

Indeed, we have that
\begin{align*}
    \varphi_P^\ast\circ d_c(f\theta_3)=\underbrace{d_c\circ\varphi_P^\ast(f\theta_3)}_{\text{weight 2}}-\underbrace{(X_2X_1f \circ \varphi + X_3X_1f \circ \varphi - Tf \circ \varphi)\, \theta_2 \wedge \tau- X_1^2f \circ \varphi \, \theta_1 \wedge \tau}_{\text{weight 3}}+\operatorname{Im}d_0
\end{align*}
which means that even after projecting onto the space of Rumin forms, the two operators $d_c \circ \varphi_P^*$ and $\varphi_P^* \circ d_c$ still produce different results.
\end{example}

\section{The Pansu pullback on spectral complexes }\label{chapter commutativity}

The previous section showed that, on arbitrary Carnot groups, neither the
exterior derivative nor the Rumin differential commutes in general with the
Pansu pullback. The purpose of this section is to show that the differentials
arising in the spectral complexes recover such a commutativity property. More precisely, let $\{(E_{j,l}^{\bullet,\bullet},\Delta_j)\}_{j\in I_{\bullet,\bullet}}$ be the family of spectral complexes associated with the de Rham complex $(\Omega^\bullet(G),d=d_0+d_1+\cdots+d_s)$ viewed as a truncated multicomplex. We prove that the Pansu pullback induced
by a smooth Pansu differentiable map between Carnot groups is compatible with
the differentials $\Delta_j$.

Throughout this section, as well as in the next one, all Pansu differentiable
maps are assumed to be smooth in the Euclidean sense. This ensures, in
particular, that both the classical pullback and the Pansu pullback of smooth
differential forms are well defined in the smooth category. Moreover, since we are working with Pansu-differentiable maps between two different Carnot groups $\varphi\colon G_1\to G_2$, we will be considering forms on both groups. To make the computations clearer, we will make the dependence on the underlying group explicit and
write
\begin{align*}
Z_r^{\bullet,\bullet}(G_i)\ \text{ and }\ B_r^{\bullet,\bullet}(G_i)\ ,\  i=1,2\,.
    \end{align*}
Whenever the group is clear from the context, we omit it from the notation.

It is worth stressing that, throughout this section, we formulate and prove all results for maps defined on the whole group.
Since the arguments use only pullback, differentiation, and weight projections, they apply directly to maps between arbitrary open subsets. Therefore, the same Pansu pullback commutativity result holds for smooth Pansu-differentiable maps $\varphi\colon U_1\subset G_1\to U_2\subset G_2$ defined on an open subset $U_1\subset G_1$, upon replacing $\Omega^\bullet(G_i)$ with $\Omega^\bullet(U_i)$.

The proofs of the main results rely on a more streamlined description
of the spaces $Z_r^{p,k-p}$ and $B_r^{p,k-p}$, and of the spectral differential
$\Delta_r$, given directly in terms of the exterior derivative and weights of forms (as opposed to those presented in Definition \ref{def: Z and B defined} and Proposition \ref{prop: Delta con diff}). We collect these
characterisations in the following lemma.

\begin{lemma}\label{remark: characterising Z and B with d}
    Let $r\geq 1$ and let $\alpha\in\Omega^{p,k-p}(G)$. Then
    \begin{align*}
    \alpha\in Z_r^{p,k-p}\ \Longleftrightarrow&\ \text{for }1\le j\le r-1\,,\text{ there exists }z_{p+j}\in\Omega^{p+j,k-p-j}(G)\text{ such that}\\&\ d\big( \alpha-\sum_{j=1}^{r-1}z_{p+j} \big) \in \Omega^{\geq p+r,k+1}(G)\,.\\
    \alpha\in B_r^{p,k-p}\ \Longleftrightarrow&\ \text{for }1\le j\le r\text{ there exists }c_{p-r+j}\in\Omega^{p- r+j,k-1-p+r-j}(G) \text{ such that }\\&\ d\big( c_{p-r+1} +c_{p-r+2} +\cdots +c_p \big) - \alpha \in \Omega^{\geq p+1,k}(G)\,.
\end{align*}   
    Finally, let $m,l\in\mathbb N$ be admissible indices for which the corresponding spaces of spectral forms are defined, and consider \begin{align*} 
    \Delta_r\colon Z_r^{p,k-p}/B_m^{p,k-p} \longrightarrow Z_l^{p+r,k+1-p-r}/B_r^{p+r,k+1-p-r}. 
    \end{align*}
    If $\alpha\in Z_r^{p,k-p}$ and $z_{p+1},\ldots,z_{p+r-1}$ are chosen as in Definition \ref{def: Z and B defined}, then
    \begin{align}\label{eq:scrittura delta in termini di d}
        \Delta_r ([\alpha]) = d\big( \alpha - \sum_{j=1}^{r-1} z_{p+j}\big) \,\,\operatorname{mod} B_r^{p+r,k+1-p-r} + \Omega^{\geq p+r+1, k+1}(G)\,.
    \end{align}
Therefore, $\Delta_r ([\alpha])$ is represented by the component of weight $p+r$ of the exterior derivative of the form $\alpha-\sum_{j=1}^{r-1}z_{p+j}$.

\end{lemma}
\begin{proof}
    We first prove the characterisation of $Z_r^{p,k-p}$ and $B_r^{p,k-p}$. This can easily be achieved by expanding the exterior derivative according to the weight of forms together with the equations that appear in Definition \ref{def: Z and B defined}. 

    If $\alpha\in Z_r^{p,k-p}$, then there exist $z_{p+j}\in\Omega^{p+j,k-p-j}(G)$ with $j=1,\ldots,r-1$, such that 
    \begin{equation*}
    \begin{aligned}
         d(\alpha-\sum_{j=1}^{r-1}z_{p+j})&=\underbrace{{d_0\alpha}}_{=0}+\underbrace{d_1\alpha-d_0z_{p+1}}_{=0}+\underbrace{d_2\alpha-d_1z_{p+1}-d_0z_{p+2}}_{=0}+\cdots+\\ &+\underbrace{d_{r-1}\alpha-\sum_{i=0}^{r-2}d_iz_{p+r-1-i}}_{=0}+d_r\alpha-\sum_{i=1}^{r-1}d_iz_{p+r-i}+\Omega^{\ge p+r+1,k+1}(G)\,.
    \end{aligned}    
    \end{equation*}
    In other words, $d(\alpha-\sum_{j=1}^{r-1}z_{p+j})\in\Omega^{\ge p+r,k+1}(G)$ as claimed. 
    
    Moreover, if we focus on the summand in weight $p+r$ in the explicit formula above, we can also recognise  part of the expression appearing in \eqref{eq:scrittura delta in termini di d}, and so this alternative characterisation follows easily once we apply the formula given in Proposition \ref{prop: Delta con diff}, so that
    \begin{align*}
        \Delta_r\big[\alpha\big]=\left[\pi_{p+r}d\bigg(\alpha-\sum_{j=1}^{r-1}z_{p+j}\bigg)\right]_{B_r^{p+r,k+1-p-r}}\,.
    \end{align*}
    
    Let us now take $\alpha\in B_r^{p,k-p}$, then there exist $c_{p-r+j}\in\Omega^{p-r+j,k-1-p+r-j}(G)$ such that
    \begin{align*}
        d\big( c_{p-r+1} +c_{p-r+2} +\cdots +c_p \big)=&\underbrace{d_0c_{p-r+1}}_{=0}+\underbrace{d_1c_{p-r+1}+d_0c_{p-r+2}}_{=0}+\cdots+\\+&\underbrace{\sum_{j=1}^{r-1}d_{r-1-j}c_{p-r+j}}_{=0}+\underbrace{\sum_{j=1}^{r}d_{r-j}c_{p-r+j}}_{=\alpha}+\Omega^{\ge p+1,k}(G)\,,
    \end{align*}
    which can easily be re-expressed as $d\big( c_{p-r+1} +c_{p-r+2} +\cdots +c_p \big) - \alpha \in \Omega^{\ge p+1,k}(G) $ as claimed.

    The converse implications are straightforward checks.
\end{proof}

Before addressing the commutativity with the spectral differentials, we must show that the Pansu pullback descends to the quotients $Z_r^{p,k-p}/B_m^{p,k-p}$. We first consider the case $r=m=1$, which involves only the algebraic differential $d_0$.

\begin{lemma}\label{lem: d_0 commutes}
Let $\varphi\colon G_1\to G_2$ be a smooth Pansu differentiable map. Then the
Pansu pullback commutes with the algebraic differential $d_0$, namely 
\begin{align*}
    \varphi_P^\ast d_0\alpha=d_0\varphi_P^\ast\alpha\ \text{ for all }\alpha\in\Omega^\bullet(G_2)\,.
\end{align*}
\end{lemma}

\begin{proof}
This is a direct consequence of the fact that the Pansu derivative is a Lie algebra homomorphism $D_P\varphi\colon\mathfrak{g}_1\to\mathfrak{g}_2$, i.e. for any $X_1,X_2\in\mathfrak{g}_1$ we have $[(D_P\varphi) X_1,(D_P\varphi) X_2]=(D_P\varphi)[X_1,X_2]$.

Let us first show this explicitly for an arbitrary 1-form. By the linearity of both $d_0$ and the Pansu pullback, we can assume without loss of generality that $\alpha=f\xi\in\Omega^1(G_2)$, so that for every $X_1,X_2\in\mathfrak{g}_1$
\begin{align*}
    (\varphi_P^\ast d_0\alpha)(X_1,X_2)=&\varphi_P^\ast (fd\xi)(X_1,X_2)=(f\circ \varphi)\, d\xi((D_P\varphi)X_1,(D_P\varphi)X_2)=-(f\circ\varphi)\,\xi([(D_P\varphi)X_1,(D_P\varphi)X_2])\\=&-(f\circ\varphi)\, \xi\big((D_P\varphi)[X_1,X_2]\big)=-(\varphi^\ast_P\alpha)([X_1,X_2])=(d_0\varphi_P^\ast\alpha)(X_1,X_2)\,.
\end{align*}

In general, given a smooth form $\alpha=f\xi\in\Omega^k(G_2)$, one can use the formula \eqref{eq: formula for d_0} to get that
\begin{align*}
    &(\varphi_P^\ast d_0\alpha)(X_1,\ldots,X_{k+1})=(f\circ\varphi)\,d\xi\big((D_P\varphi)X_1,\ldots,(D_P\varphi)X_{k+1}\big)\\&=\sum_{1\le i<j\le k+1}(-1)^{i+j}(f\circ \varphi)\,\xi\big([(D_P\varphi )X_i,(D_P\varphi)X_j],(D_P\varphi)X_1,\ldots,\widehat{(D_P\varphi)X_i},\ldots,\widehat{(D_P\varphi)X_j},\ldots,(D_P\varphi)X_{k+1}\big)\\&=\sum_{1\le i<j\le k+1}(-1)^{i+j}(f\circ\varphi)\,\xi\big((D_P\varphi)[X_i,X_j],(D_P\varphi)X_1,\ldots,\widehat{(D_P\varphi)X_i},\ldots,\widehat{(D_P\varphi)X_j},\ldots,(D_P\varphi)X_{k+1}\big)\\&=\sum_{1\le i<j\le k+1}(-1)^{i+j}\varphi_P^\ast\alpha([X_i,X_j],X_1,\ldots,\widehat{X_i},\ldots,\widehat{X_j},\ldots,X_{k+1})=(d_0\varphi_P^\ast\alpha)(X_1,\ldots,X_{k+1})
\end{align*}
holds for any $X_1,\ldots,X_{k+1}\in\mathfrak{g}_1$. In the computation above, we have written $D_P\varphi$ in place of
$D_P\varphi(x)$ to simplify the notation.
\end{proof}

\begin{proposition}\label{prop: commutativity at page 1}
    Let $\varphi\colon G_1\to G_2$ be a smooth Pansu differentiable map. Then,
    \begin{align*}
        \varphi_P^\ast Z_1^{p,k-p}(G_2)\subset Z_1^{p,k-p}(G_1)\ \text{ and }\varphi_P^\ast B_1^{p,k-p}(G_2)\subset B_1^{p,k-p}(G_1)\ \text{ for every }p,k\in\mathbb{N}\,.
    \end{align*}
\end{proposition}
\begin{proof}
    Following Definition \ref{def: Z and B defined} when $r=1$, we get
    \begin{align*}
        \alpha\in Z_1^{p,k-p}(G_2)\Longleftrightarrow d_0\alpha=0\,,
    \end{align*}
    hence, by Lemma \ref{lem: d_0 commutes}, if $\alpha\in Z_1^{p,k-p}(G_2)$ we get that $d_0\varphi_P^\ast\alpha=\varphi_P^\ast d_0\alpha=0$ and so $\varphi_P^\ast\alpha\in Z_1^{p,k-p}(G_1)$.

    On the other hand, 
    \begin{align*}
        \alpha\in B_1^{p,k-p}(G_2)\ \Longleftrightarrow\ \alpha=d_0\beta_p\ \text{ for some }\beta_p\in\Omega^{p,k-1-p}(G_2)
    \end{align*}
    and so, again by Lemma \ref{lem: d_0 commutes}, if $\alpha\in B_1^{p,k-p}(G_2)$, then $\varphi_P^\ast\alpha=\varphi_P^\ast d_0\beta_p=d_0\varphi^\ast_P\beta_p\in B_1^{p,k-p}(G_1)$.
\end{proof}
\begin{remark}
    Proposition \ref{prop: commutativity at page 1} gives a well-defined map between the quotients $Z_1^{p,k-p}/B_1^{p,k-p}$, but does not imply that the Pansu pullback preserves their harmonic representatives, i.e. the inclusion $\varphi_P^\ast\big(\ker \Box_0\cap\Omega^{p,k-p}(G_2)\big)\subset \ker\Box_0\cap\Omega^{p,k-p}(G_1)$ need not hold. Indeed, in general, we will only have the inclusion
    \begin{align*}
        \varphi_P^\ast\left(\ker\Box_0\cap\Omega^{p,k-p}(G_2)\right)\subset Z_1^{p,k-p}(G_1)\,.
    \end{align*}
    This can be easily seen already in the case of a Pansu differentiable map $\varphi\colon\mathbb H^1\times\R\to\mathbb H^1\times\R$ such as the one presented in Example \ref{example dc not commutes with Pansu pullback}.

The identity in Lemma \ref{lem: d_0 commutes} can also be viewed as the weight-preserving component of the classical identity $d\varphi^*=\varphi^*d$. To describe the corresponding identities at higher weights, we decompose the classical pullback according to its weight increase.
\end{remark}
\begin{lemma}[Weight components of the classical pullback]
\label{lem:weight-components-classical-pullback}
Let \(\varphi:G_1\to G_2\) be a Euclidean smooth Pansu
differentiable map. For every \(j\geq0\), define
\[
\Phi_j:\Omega^{p,k-p}(G_2)
\longrightarrow
\Omega^{p+j,k-p-j}(G_1)\ ,\ 
\Phi_j\alpha
:=
\pi_{p+j}^1\bigl(\varphi^*\alpha\bigr)\,.
\]
Then for every $\alpha\in\Omega^{p,k-p}(G_2)$, we have $\varphi^*\alpha=\sum_{j\geq0}\Phi_j\alpha$ and $\Phi_0=\varphi_P^*$.
Moreover, for every \(n\geq0\),
\begin{equation}
\label{eq:components-filtered-chain-map}
\sum_{a+b=n}d_a\Phi_b
=
\sum_{a+b=n}\Phi_a d_b\,.
\end{equation}
\end{lemma}

\begin{proof}
By Proposition~\ref{prop:weight-behaviour-pullbacks}, the classical
pullback does not decrease weight. Thus, if
\(\alpha\in\Omega^{p,k-p}(G_2)\), then
\[
\varphi^*\alpha
\in
\bigoplus_{j\geq0}\Omega^{p+j,k-p-j}(G_1)\ \text{
which gives }
\
\varphi^*\alpha=\sum_{j\geq0}\Phi_j\alpha\,.
\]
The same proposition also implies that $
\Phi_0\alpha
=
\pi_p^1(\varphi^*\alpha)
=
\varphi_P^*\alpha$.

Let \(\alpha\in\Omega^{p,k-p}(G_2)\), then we have the following expressions in terms of the weight of forms
\[
d\varphi^*\alpha
=
\sum_{a,b\geq0}d_a\Phi_b\alpha\ \text{ and }\
\varphi^*d\alpha
=
\sum_{a,b\geq0}\Phi_a d_b\alpha\,.
\]
Using the classical commutativity result $d\varphi^\ast=\varphi^\ast d$ and taking the component of weight \(p+n\) on both sides yields
\[
\sum_{a+b=n}d_a\Phi_b\alpha
=
\sum_{a+b=n}\Phi_a d_b\alpha.
\]
Moreover, taking \(n=0\)
and using \(\Phi_0=\varphi_P^*\) gives an alternative proof of Lemma \ref{lem: d_0 commutes}.

This commutation identity can also be obtained pointwise
from the naturality of the exterior derivative under the
graded group homomorphism $D_P\varphi(x)$, for each fixed
$x\in G_1$ (see also the proof of
\cite[Lemma 4.8]{kleiner2020pansu}). Indeed, pullback by this homomorphism preserves
left-invariant forms, on which $d$ coincides with $d_0$, so that
$$d_0\bigl((D_P\varphi(x))^*\xi\bigr)
=
(D_P\varphi(x))^*(d_0\xi)\ \text{ for every }\xi\in\bigwedge\nolimits^\bullet\mathfrak g_2^\ast\,.$$
The identity $d_0\varphi_P^*=\varphi_P^*d_0$ then follows
by the $\C^\infty(G_h)$-linearity of $d_0$.
\end{proof}

We now extend Proposition \ref{prop: commutativity at page 1} to arbitrary cycle and boundary spaces.

\begin{proposition}[Preservation of the spaces \(Z_r\) and \(B_r\)]
\label{prop:Pansu-preserves-Zr-Br}
Let \(\varphi:G_1\to G_2\) be a Euclidean smooth Pansu
differentiable map. Then, for every \(r\geq1\),
$$
\varphi_P^*
\bigl(Z_r^{p,k-p}(G_2)\bigr)
\subseteq
Z_r^{p,k-p}(G_1)
\ \text{ and }\ 
\varphi_P^*
\bigl(B_r^{p,k-p}(G_2)\bigr)
\subseteq
B_r^{p,k-p}(G_1)\,.
$$
Consequently, for every $r,m\ge 1$, since 
$B_m^{p,k-p}\subseteq Z_r^{p,k-p}$ by \cite[Lemma 4.1]{tripaldi2026spectralcomplexestruncatedmulticomplexes}, the Pansu pullback induces a
well-defined map
\begin{align}\label{eq: def pullback on quotients}
\frac{Z_r^{p,k-p}(G_2)}{B_m^{p,k-p}(G_2)}
\longrightarrow
\frac{Z_r^{p,k-p}(G_1)}{B_m^{p,k-p}(G_1)} \quad ,\quad 
[\alpha]
\mapsto
[\varphi_P^*\alpha]\,.
\end{align}
\end{proposition}

\begin{proof}
We prove the two inclusions separately.

Assume first that $\alpha\in Z_r^{p,k-p}(G_2)$.
By Lemma~\ref{remark: characterising Z and B with d}, there exist forms $z_{p+j}\in\Omega^{p+j,k-p-j}(G_2)$ for $1\le i\le r-1$ such that
\[
d\left(\alpha-\sum_{j=1}^{r-1}z_{p+j}\right)\in\Omega^{\geq p+r,k+1}(G_2)\,.
\]
Since the ordinary pullback preserves the weight filtration (by Proposition \ref{prop:weight-behaviour-pullbacks}) and it
also commutes with the exterior derivative,
\[
d\varphi^*\left(\alpha-\sum_{j=1}^{r-1}z_{p+j}\right)
=
\varphi^*d\left(\alpha-\sum_{j=1}^{r-1}z_{p+j}\right)
\in
\Omega^{\geq p+r,k+1}(G_1)\,.
\]

Because
\[
\varphi^*\left(\alpha-\sum_{j=1}^{r-1}z_{p+j}\right)-\sum_{j=0}^{r-1}
\pi_{p+j}^1\left(\varphi^*\left(\alpha-\sum_{j=1}^{r-1}z_{p+j}\right)\right)
\in
\Omega^{\geq p+r,k}(G_1)
\]
and \(d\) does not decrease weight, it follows that
\begin{align}\label{eq: then Z_r}
d\left(\sum_{j=0}^{r-1}
\pi_{p+j}^1\left(\varphi^*\left(\alpha-\sum_{j=1}^{r-1}z_{p+j}\right)\right)\right)
\in
\Omega^{\geq p+r,k+1}(G_1)\,.
\end{align}
Moreover, since
\[
\pi_p^1\left(\varphi^*\left(\alpha-\sum_{j=1}^{r-1}z_{p+j}\right)\right)
=
\pi_p^1(\varphi^*\alpha)
=
\varphi_P^*\alpha\,,
\] we have
$$\sum_{j=0}^{r-1}
\pi_{p+j}^1\left(\varphi^*\left(\alpha-\sum_{j=1}^{r-1}z_{p+j}\right)\right)=\varphi_P^\ast\alpha-\sum_{j=1}^{r-1}z'_{p+j}$$
where $z'_{p+j}
:=
-\pi_{p+j}^1(\varphi^*(\alpha-\sum_{i=1}^{r-1}z_{p+i}))$ for \(1\leq j\leq r-1\).

Therefore, \eqref{eq: then Z_r} implies that, given the form $\varphi_P^\ast\alpha\in\Omega^{p,k-p}(G_1)$, there exist $z'_{p+j}\in\Omega^{p+j,k-p-j}(G_1)$ for $1\le j\le r-1$ such that
\[
d
\left(
\varphi_P^*\alpha
-
\sum_{j=1}^{r-1}z'_{p+j}
\right)
\in
\Omega^{\geq p+r,k+1}(G_1)\,.
\]
In other words, $\varphi_P^*\alpha\in Z_r^{p,k-p}(G_1)$ there the relative correction forms $z'_{p+j}\in\Omega^{p+j,k-p-j}(G_1)$ are given by
$$
z'_{p+j}
=
\sum_{h=1}^{j}\Phi_{j-h}z_{p+h}
-
\Phi_j\alpha\ ,
\
1\leq j\leq r-1\,.$$
We now prove the statement for \(B_r\). Assume that $\alpha\in B_r^{p,k-p}(G_2)$.
By Definition~\ref{def: Z and B defined}, there exist $c_{p-j}
\in
\Omega^{p-j,k-1-p+j}(G_2)$ for $0\leq j\leq r-1$
such that
$$d\left(\sum_{j=0}^{r-1}c_{p-j}\right)
=
\alpha+\gamma\ \text{ for some }
\
\gamma\in\Omega^{\geq p+1,k}(G_2)\,.$$
Applying the ordinary pullback gives
\[
d\left(\varphi^*\left(\sum_{j=0}^{r-1}c_{p-j}\right)\right)
=
\varphi^*\left(d\left(\sum_{j=0}^{r-1}c_{p-j}\right)\right)
=
\varphi^*\alpha+\varphi^*\gamma
=
\varphi_P^*\alpha
+
\Omega^{\geq p+1,k}(G_1)\,.
\]
Since \(\varphi^*\) does not decrease weight, we have
\[
\varphi^*\left(\sum_{j=0}^{r-1}c_{p-j}\right)-\sum_{j=0}^{r-1}
\pi_{p-j}^1\left(\varphi^*\left(\sum_{j=0}^{r-1}c_{p-j}\right)\right)
\in
\Omega^{\geq p+1,k-1}(G_1)\,,
\]
and so
\[
d \left( \sum_{j=0}^{r-1}
\pi_{p-j}^1\left(\varphi^*\left(\sum_{j=0}^{r-1}c_{p-j}\right)\right)\right)
=
\varphi_P^*\alpha
+
\Omega^{\geq p+1,k}(G_1)\,.
\]
By the characterisation of $B_r$ presented in Lemma \ref{remark: characterising Z and B with d}, this proves that $\varphi_P^*\alpha\in B_r^{p,k-p}(G_1)$.

The two inclusions show that $[\alpha]\longmapsto[\varphi_P^*\alpha]$
is well-defined on every admissible quotient
\(Z_r^{p,k-p}/B_m^{p,k-p}\).
\end{proof}

To avoid cumbersome notation, we will again use $\varphi_P^*$ to denote the maps induced by $\varphi_P^*$ on the quotients in \eqref{eq: def pullback on quotients}. 
Having established their well-definedness, we now prove their compatibility with the spectral differentials.

\begin{theorem}[Commutativity of the Pansu pullback with the spectral differentials]
\label{thm: commutativity}
Let $\varphi:G_1\longrightarrow G_2$
be a Euclidean smooth Pansu differentiable map between two Carnot
groups. Fix \(p,k\in\mathbb N\) and \(i\geq1\) such that $i\in I_{p,k}$.
Then, for every $\alpha\in Z_i^{p,k-p}(G_2)$,
one has
\begin{equation}
\label{eq:Pansu-commutativity-mod-Bi}
\varphi_P^*\bigl(\Delta_i\alpha\bigr)
=
\Delta_i\bigl(\varphi_P^*\alpha\bigr)
\quad
\operatorname{mod}
B_i^{p+i,k+1-p-i}(G_1)\,.
\end{equation}
More precisely, if \(\omega_2\) is any representative of
\(\Delta_i\alpha\) and \(\omega_1\) is any representative of
\(\Delta_i(\varphi_P^*\alpha)\), then
\[
\varphi_P^*\omega_2-\omega_1
\in
B_i^{p+i,k+1-p-i}(G_1)\,.
\]
With this notation, the following diagram is commutative:
\[
\begin{tikzcd}
Z_i^{p,k-p}(G_2)
\arrow[r,"\Delta_i"]
\arrow[d,"\varphi_P^*"']
&
B_{i+1}^{p+i,k+1-p-i}(G_2)/
B_i^{p+i,k+1-p-i}(G_2)
\arrow[d,"\varphi_P^*"]
\\
Z_i^{p,k-p}(G_1)
\arrow[r,"\Delta_i"']
&
B_{i+1}^{p+i,k+1-p-i}(G_1)/
B_i^{p+i,k+1-p-i}(G_1)
\end{tikzcd}
\]
\end{theorem}

\begin{proof}

Let $\alpha\in Z_i^{p,k-p}(G_2)$.
By Definition~\ref{def: Z and B defined}, there exist forms $z_{p+j}\in\Omega^{p+j,k-p-j}(G_2)$ for $1\leq j\leq i-1$
such that
\[
d\left(\alpha-\sum_{j=1}^{i-1}z_{p+j}\right)
\in
\Omega^{\geq p+i,k+1}(G_2)\,.
\]
Set
\[
\omega_2
:=
\pi_{p+i}^2\left(d\left(\alpha-\sum_{j=1}^{i-1}z_{p+j}\right)\right).
\]
Then $\omega_2
=
d_i\alpha
-
\sum_{j=1}^{i-1}
d_{i-j}z_{p+j}$
and \(\omega_2\) represents \(\Delta_i\alpha\).
We now construct compatible correction forms for
\(\varphi_P^*\alpha\). 

Since the ordinary pullback preserves the filtration by weight,
\[
\varphi^*\left(\alpha-\sum_{j=1}^{i-1}z_{p+j}\right)-\sum_{j=0}^{i-1}
\pi_{p+j}^1\left(\varphi^*\left(\alpha-\sum_{j=1}^{i-1}z_{p+j}\right)\right)
\in
\Omega^{\geq p+i,k}(G_1)\,.
\]
Moreover, the ordinary pullback commutes with the exterior
derivative, and therefore
\[
d\left(\varphi^*\left(\alpha-\sum_{j=1}^{i-1}z_{p+j}\right)\right)
=
\varphi^*\left(d\left(\alpha-\sum_{j=1}^{i-1}z_{p+j}\right)\right)
\in
\Omega^{\geq p+i,k+1}(G_1)\,.
\]
Because \(d\) does not decrease weight, it follows that
\[
d\left(\sum_{j=0}^{i-1}
\pi_{p+j}^1\left(\varphi^*\left(\alpha-\sum_{j=1}^{i-1}z_{p+j}\right)\right)
\right)\in
\Omega^{\geq p+i,k+1}(G_1)\,.
\]

Moreover,
\[
\sum_{j=0}^{i-1}
\pi_{p+j}^1\left(\varphi^*\left(\alpha-\sum_{j=1}^{i-1}z_{p+j}\right)\right)
=\pi_p^1\varphi^\ast\alpha+\sum_{j=1}^{i-1}
\pi_{p+j}^1\left(\varphi^*\left(\alpha-\sum_{j=1}^{i-1}z_{p+j}\right)\right)=
\varphi_P^*\alpha
-
\sum_{j=1}^{i-1}z'_{p+j}
\]
for suitable $z'_{p+j}
\in
\Omega^{p+j,k-p-j}(G_1)$.

Consequently, these forms constitute an admissible family of
correction forms for \(\varphi_P^*\alpha\). Therefore,
\[
\omega_1
:=
\pi_{p+i}^1\left(d\sum_{j=0}^{i-1}
\pi_{p+j}^1\left(\varphi^*\left(\alpha-\sum_{j=1}^{i-1}z_{p+j}\right)\right)\right)
\]
represents $\Delta_i\bigl(\varphi_P^*\alpha\bigr)$.

Since
\[
\varphi^*\left(\alpha-\sum_{j=1}^{i-1}z_{p+j}\right)
=
\sum_{j=0}^{i-1}
\pi_{p+j}^1\left(\varphi^*\left(\alpha-\sum_{j=1}^{i-1}z_{p+j}\right)\right)+
\pi_{p+i}^1\left(\varphi^*\left(\alpha-\sum_{j=1}^{i-1}z_{p+j}\right)\right)+\Omega^{\geq p+i+1,k}(G_1)\,,
\]
applying \(d\) and taking the component of weight \(p+i\)
gives
\begin{align}
\label{eq:weight-pi-pullback-u}
\pi_{p+i}^1\left(d\left(\varphi^*\left(\alpha-\sum_{j=1}^{i-1}z_{p+j}\right)\right)\right)=
\omega_1+d_0\left(\pi_{p+i}^1\left(\varphi^*\left(\alpha-\sum_{j=1}^{i-1}z_{p+j}\right)\right)\right)\,.
\end{align}

On the other hand, since $d\varphi^\ast=\varphi^\ast d$
and the weight-preserving component of the classical pullback is the
Pansu pullback, we obtain
\begin{equation}
\label{eq:weight-pi-Pansu}
\pi_{p+i}^1\left(d\varphi^*\left(\alpha-\sum_{j=1}^{i-1}z_{p+j}\right)\right)
=\pi_{p+i}^1\left(\varphi^\ast d\left(\alpha-\sum_{j=1}^{i-1}z_{p+j}\right)\right)=\pi_{p+i}^1\big(\varphi^\ast\big(\omega_2+\Omega^{\ge p+i+1,k+1}(G_2)\big)\big)=
\varphi_P^*\omega_2\,.
\end{equation}
Combining
\eqref{eq:weight-pi-pullback-u} and
\eqref{eq:weight-pi-Pansu}, we find $\varphi_P^*\omega_2-\omega_1
\in
B_1^{p+i,k+1-p-i}(G_1)
\subseteq
B_i^{p+i,k+1-p-i}(G_1)$, as claimed in
\eqref{eq:Pansu-commutativity-mod-Bi}.
Finally, recall from \eqref{eq:image of Z_r by Delta_r} that representatives of $\Delta_i\alpha$, for $\alpha\in Z_i^{p,k-p}$, belong to $B_{i+1}^{p+i,k+1-p-i}$. This gives the horizontal arrows in the stated diagram. Proposition \ref{prop:Pansu-preserves-Zr-Br} gives the left-hand vertical arrow and, since it also preserves both $B_{i+1}^{p+i,k+1-p-i}$ and $B_i^{p+i,k+1-p-i}$, the induced right-hand vertical arrow. The identity just proved establishes the commutativity of the diagram. 
\end{proof}

\begin{corollary}
\label{cor: commutativity spectral}
    Let $\varphi\colon G_1\to G_2$ be a Pansu differentiable map between two Carnot groups with respective spectral complexes $\{(E_{i,j}^{\bullet,\bullet}(G_1),\Delta_i)\}_{i\in I_{\bullet,\bullet}}$ and $\{(E_{i,j}^{\bullet,\bullet}(G_2),\Delta_{i})\}_{i\in I_{\bullet,\bullet}}$, then the Pansu pullback commutes with the differentials $\Delta_i$ for each non trivial choice of weight, degree, and order $i$, i.e. the diagram
\begin{equation}\label{diagram commutativity}
\begin{tikzcd}
Z_i^{p,k-p}(G_2) /B_{m_2}^{p,k-p}(G_2)\arrow[r, "\Delta_i"] \arrow[d, "\varphi_P^\ast"'] 
  & Z_{l_2}^{p+i,k+1-p-i}(G_2)/B_i^{p+i,k+1-p-i}(G_2) \arrow[d, "\varphi_P^\ast"] \\
Z_i^{p,k-p}(G_1)/B_{m_1}^{p,k-p}(G_1) \arrow[r, "\Delta_i"'] 
  & Z_{l_1}^{p+i,k+1-p-i}(G_1)/B_i^{p+i,k+1-p-i}(G_1)
\end{tikzcd}
\end{equation}
commutes for any (non trivial) choice of $p,k,l_1,l_2,m_1,m_2\in\mathbb{N}$.
\end{corollary}

\begin{proof}
\label{rem:arbitrary-indices-commutativity} The formulation of Theorem~\ref{thm: commutativity} is independent of
the auxiliary indices used to realise the spectral complexes.

Indeed, by
\cite[Proposition 4.2]{tripaldi2026spectralcomplexestruncatedmulticomplexes},
for every \(m,l\in\mathbb N\), the differential \(\Delta_i\) factors
as
\[
\frac{Z_i^{p,k-p}}{B_m^{p,k-p}}
\xrightarrow{\ \Delta_i\ }
\frac{B_{i+1}^{p+i,k+1-p-i}}
     {B_i^{p+i,k+1-p-i}}
\hookrightarrow
\frac{Z_l^{p+i,k+1-p-i}}
     {B_i^{p+i,k+1-p-i}}\,.
\]
Here, the operator \(\Delta_i\) vanishes on \(B_m^{p,k-p}\)
because $\Delta_i\bigl(B_m^{p,k-p}\bigr)
\subseteq
B_i^{p+i,k+1-p-i}$,
and hence \(B_m^{p,k-p}\) is sent to zero in the quotient by
\(B_i^{p+i,k+1-p-i}\).

It follows that, for arbitrary \(m_1,m_2,l_1,l_2\), both expressions
\[
\Delta_i\bigl(\varphi_P^*\alpha\bigr)
\qquad\text{and}\qquad
\varphi_P^*\bigl(\Delta_i\alpha\bigr)
\]
are well defined modulo
\(B_i^{p+i,k+1-p-i}(G_1)\), independently of the quotient
representatives and of the chosen indices \(m_1,m_2,l_1,l_2\).
Theorem~\ref{thm: commutativity} asserts that these two expressions
define the same class.

For arbitrary \(m_1,m_2,l_1,l_2\), however, the diagram
\eqref{diagram commutativity} should only be regarded as schematic.
In fact, the expression
\[
[\alpha]_{B_{m_2}(G_2)}
\longmapsto
[\varphi_P^*\alpha]_{B_{m_1}(G_1)}
\]
does not in general define a map unless $\varphi_P^*\bigl(B_{m_2}(G_2)\bigr)
\subseteq
B_{m_1}(G_1)$.
Similarly, the right-hand vertical arrow need not be defined on the
entire displayed quotient when \(l_1\neq l_2\).

If \(m_1=m_2\) and \(l_1=l_2\), then both vertical arrows are
well-defined and \eqref{diagram commutativity} is a genuine
commutative diagram. More generally, using the standard inclusions $B_m\subseteq B_{m+1}$ and 
$Z_{l+1}\subseteq Z_l$,
the diagram is genuine whenever $m_2\leq m_1$ and $l_2\geq l_1$.

\end{proof}

\subsection{Spectral complexes as subspaces}
Theorem \ref{thm: commutativity} and Corollary \ref{cor: commutativity spectral} establish commutativity at the level of the quotients defining the spectral complexes. For the applications to central extensions in Section 5, we need to express these results in terms of representatives belonging to subspaces of smooth differential forms. The fibrewise inner product alone does not provide a general orthogonal realisation of the quotients $Z_r^{p,k-p}/B_m^{p,k-p}$. An $L^2$-based approach would require additional choices of function spaces and domains for the differential operators; see Subsection \ref{subsect: Orthogonal realisations of the spectral complexes}. For our purposes, however, it suffices to consider spectral differentials acting on horizontal $1$-forms, whose images will be compared with left-invariant $2$-cocycles in Section 5. Two features make this case particularly simple. First, Rumin $1$-forms are precisely horizontal $1$-forms. Second, in the target bidegree of $\Delta_r$, the boundary space $B_r$ coincides with $B_1=\operatorname{Im}d_0$ (see Lemma \ref{lem:d0-injective-higher-weight-one-forms} and Proposition \ref{prop: Bi é B1 per le 2 forme}).

We first isolate the part of the construction that relies only on the fibrewise Hodge decomposition. This removes the $d_0$-exact component from any spectral quotient and, when the denominator is $B_1$, identifies the quotient with a subspace of Rumin forms.

\begin{lemma}[Harmonic representatives of the spectral quotients]\label{lem: Harmonic reduction of the spectral quotients}
    Let $G$ be a Carnot group, equipped with the fibrewise inner product fixed in Section 2. For every $r,m \geq 1$
    \begin{align*}
        Z_r^{p,k-p} &= B_1^{p,k-p} \oplus \big( Z_r^{p,k-p} \cap \ker \Box_0 \big)\\
        B_m^{p,k-p} &= B_1^{p,k-p} \oplus \big( B_m^{p,k-p} \cap \ker \Box_0 \big)\,.
    \end{align*}
    Therefore, the projection $\Pi_0$ induces an isomorphism 
    \begin{align*}
\frac{Z_r^{p,k-p}}{B_m^{p,k-p}}
\longrightarrow
\frac{Z_r^{p,k-p}\cap \ker \Box_0}{B_m^{p,k-p}\cap \ker \Box_0} \quad ,\quad 
[\alpha]
\mapsto
[\Pi_0\alpha]\,.
\end{align*}
In particular,
\begin{align*}
\frac{Z_r^{p,k-p}}{B_1^{p,k-p}}
\xrightarrow[]{\cong}
Z_r^{p,k-p} \cap \ker \Box_0\ \text{ where }\ [\alpha]\mapsto\Pi_0\alpha \,.
\end{align*}
\end{lemma}
\begin{proof}
For every $\alpha\in Z_r^{p,k-p}$, the fibrewise Hodge decomposition gives $
\alpha=\Pi_0\alpha+d_0d_0^{-1}\alpha$. Since $d_0d_0^{-1}\alpha\in B_1^{p,k-p}\subseteq Z_r^{p,k-p}$, we also have $\Pi_0\alpha\in Z_r^{p,k-p}$. The same argument, applied to $\alpha\in B_m^{p,k-p}$, shows that $\Pi_0\alpha\in B_m^{p,k-p}$.
\end{proof}

\begin{corollary}\label{cor: pullback in the subspaces}
    Let $\varphi:G_1 \rightarrow G_2$ be a smooth Pansu differentiable map. Under the identifications of Lemma \ref{lem: Harmonic reduction of the spectral quotients} the map induced by the Pansu pullback is represented by
    \begin{align*}
\frac{Z_r^{p,k-p}(G_2)\cap E_0^k(G_2)}{B_m^{p,k-p}(G_2)\cap  E_0^k(G_2)} \rightarrow \frac{Z_r^{p,k-p}(G_1)\cap  E_0^k(G_1)}{B_m^{p,k-p}(G_1)\cap  E_0^k(G_1)}  \quad ,\quad 
[\alpha] \longmapsto \big[\Pi_0 \varphi_P^* \alpha\big]\,.
\end{align*}
In particular, when the denominator is $B_1^{p,k-p}(G_h)$, this is the map 
\begin{equation*}
    \Pi_0 \varphi_P^* : Z_r^{p,k-p}(G_2) \cap  E_0^k(G_2) \rightarrow Z_r^{p,k-p}(G_1) \cap E_0^k(G_1)\,.
\end{equation*}
\end{corollary}
\begin{proof}
    By Proposition \ref{prop:Pansu-preserves-Zr-Br}, the Pansu pullback preserves the relevant cycle and boundary spaces. By Lemma \ref{lem: Harmonic reduction of the spectral quotients}, the projection $\Pi_0$ also preserves each of these spaces. Hence $\Pi_0^1\varphi_P^*$ maps harmonic cycles to harmonic cycles and harmonic boundaries to harmonic boundaries, so the displayed map is well defined. Moreover, for $\alpha\in Z_r^{p,k-p}(G_2)$,
\[
\varphi_P^*\alpha-\Pi_0\varphi_P^*\alpha
\in B_1^{p,k-p}(G_1)
\subseteq B_m^{p,k-p}(G_1).
\]
Thus the projected form represents the same quotient class as the Pansu pullback.
\end{proof}

We next express the cycle and boundary conditions, as well as the spectral differentials, in terms of the homogeneous components of the Rumin differential. We write $d_c=\sum_{j\geq1}d_c^j$, where only finitely many components are nonzero. The following formulation combines \cite[Propositions 3.4, 3.6 and 3.8]{tripaldi2026spectralcomplexestruncatedmulticomplexes} with Proposition \ref{prop:Pansu-preserves-Zr-Br}.

\begin{proposition}\label{prop: formulating main thm with Rumin differentials}
    Let $\varphi\colon G_1\to G_2$ be a Pansu differentiable map between two Carnot groups with respective spectral complexes $\{(E_{i,j}^{\bullet,\bullet}(G_1),\Delta_i)\}_{i\in I_{\bullet,\bullet}}$ and $\{(E_{i,j}^{\bullet,\bullet}(G_2),\Delta_i)\}_{i\in I_{\bullet,\bullet}}$. Let $\alpha \in \ker \Box_0$, then (assuming $r\ge 2$) 
   \begin{align*}
       \alpha\in Z_r^{p,k-p}(G_h)\Longleftrightarrow &\exists\,  \omega_{p+i}\in\Omega^{p+i,k-p-i}(G_h)\cap\ker\Box_0 \text{ with }i=1,\ldots,r-2\text{ such that }\\
       &\ d_c^i\alpha=\sum_{j=1}^{i-1}d_c^{i-j}\omega_{p+j}\ \text{ for each }i=1,\ldots,r-1\,.\\
       \alpha\in B_r^{p,k-p}(G_h)\Longleftrightarrow&\exists\, c_{p-r+i}\in\Omega^{p-r+i,k-1-p+r-i}(G_h)\cap\ker\Box_0\text{ with }i=1,\ldots,r-1\text{ such that }\\
       & d_c^ic_{p-r+1}=\sum_{j=1}^{i-1}d_c^{i-j}c_{p-r+1+j}\text{ for each }i=1,\ldots,r-2\text{ and }\\
       &\alpha=d_c^{r-1}c_{p-r+1}-\sum_{i=1}^{r-2}d_c^{r-1-i}c_{p-r+1+i}\,.
   \end{align*}
   Hence, the property $\varphi_P^\ast(Z_r^{p,k-p}(G_2))\subset Z_r^{p,k-p}(G_1)$ can be expressed as
   \begin{align*}
       \text{given }\alpha\in Z_r^{p,k-p}(G_2)\text{ there exist }&\xi_{p+i}\in\Omega^{p+i,k-p-i}(G_1)\cap\ker\Box_0\text{ with } i=1,\ldots,r-2\\ \text{ such that }
       d_c^i \Pi_0\varphi_P^\ast\alpha&=\sum_{j=1}^{i-1}d_c^{i-j}\xi_{p+j}\ \text{ for each }i=1,\ldots,r-1\,.
   \end{align*}

   Likewise, the fact that $\varphi_P^\ast(B_r^{p,k-p}(G_2))\subset B_r^{p,k-p}(G_1)$ can be expressed as
   \begin{align*}
       \text{given }\alpha\in B_r^{p,k-p}(G_2)\text{ there exist }\gamma_{p-r+i}\in\Omega^{p-r+i,k-1-p+r-i}(G_1)\cap\ker\Box_0\text{ with }i=1,\ldots,r-1
       \text{ such that }\\d_c^i\gamma_{p-r+1}=\sum_{j=1}^{i-1}d_c^{i-j}\gamma_{p-r+1+j}\text{ for each }i=1,\ldots,r-2\text{ and }\Pi_0\varphi_P^\ast\alpha=d_c^{r-1}\gamma_{p-r+1}-\sum_{i=1}^{r-2}d_c^{r-1-i}\gamma_{p-r+1+i}\,.
   \end{align*}
   Finally, given an arbitrary $\alpha\in Z_r^{p,k-p}(G_h) \cap \ker\Box_0$ for each $h=1,2$, we have that
   \begin{align*}
       \Delta_r\alpha=d_c^r\alpha-\sum_{i=2}^{r-1}d_c^i\omega_{p+r-i}+B_r^{p+r,k+1-p-r}(G_h)
   \end{align*}
   where $\omega_{p+r-i}\in\ker\Box_0\cap\Omega^{p+r-i,k-p-r+i}(G_h)$ such that $d_c^j\alpha=\sum_{i=1}^{j-1}d_c^i\omega_{p+j-i}$ for $j=1,\ldots,r-1$.
\end{proposition}

\begin{remark}
    The correction forms $\xi_{p+j}$ on $G_1$ should not, in general, be identified with the Pansu pullbacks of the original correction forms on $G_2$. The Pansu pullback need not preserve elements in $\ker\Box_0$. The statements above require only the existence of suitable correction forms that belong to $\ker\Box_0$.
\end{remark}

We now specialise these descriptions to horizontal $1$-forms. In this degree, the harmonic correction forms of higher weight vanish, and the target boundary spaces reduce to $B_1$. The algebraic fact underlying both simplifications is the following.

\begin{lemma}\label{lem:d0-injective-higher-weight-one-forms}
Let $G$ be a Carnot group with stratified Lie algebra $\mathfrak g=V_1\oplus\cdots\oplus V_s$.
Then, for every $p\ge 2$, the map
$$d_0\colon
\Omega^{p,1-p}(G)
\longrightarrow
\Omega^{p,2-p}(G)$$
is injective. Consequently $E_0^1 = \Omega^{1,0}(G)$.
\end{lemma}
\begin{proof}
    Let $\alpha \in \Omega^{p,1-p}(G)$ be such that $d_0 \alpha = 0$. Fix $x \in G$ and use the left translation to regard $\alpha_x$ as an element in $\bigwedge\nolimits^{p,1-p}\mathfrak g^* = V_p^*$. For every $X_0 \in V_1$ and $X_1 \in V_{p-1}$, we have
    \begin{equation*}
        0 = (d_0\alpha)_x(X_0, X_1) = -\alpha_x([X_0,X_1])\,.
    \end{equation*}
    Since $[V_1, V_{p-1}] = V_p$, the element $\alpha_x \in V_p^*$ vanishes on $V_p$ and hence $\alpha_x = 0$ for any $x \in G$, and so $\alpha = 0$.
\end{proof}

\begin{proposition}\label{prop: Bi é B1 per le 2 forme}
    Let $G$ be a Carnot group. For every $i \geq 2$,
    \begin{equation*}
        Z_i^{1,0} = \{\alpha \in \Omega^{1,0}(G) \mid d_c^j\alpha = 0 \text{ for } 1 \leq j \leq i-1\}
    \end{equation*}
    and 
    \begin{equation*}
        B_i^{i+1,1-i} =  B_1^{i+1,1-i} 
    \end{equation*}
Consequently, for each $i\in I_{1,1}$ the expression of the spectral complex differential simplifies to
\begin{align*}
    \Delta_i\colon Z_i^{1,0}/B_1^{1,0}\longrightarrow Z_m^{1+i,1-i}/B_1^{1+i,1-i} \ \text{ for some }m\in I_{1+i,2}\,,
\end{align*}
where $m$ is an admissible index for the target term.
\end{proposition}

\begin{proof} 
By Lemma \ref{lem:d0-injective-higher-weight-one-forms}, for every $p \geq 2$, $E_0^1 \cap \Omega^{p,1-p} = \{0\}$. Thus, all correction forms in Proposition \ref{prop: formulating main thm with Rumin differentials} vanish when $p=k=1$.

By the standard set of inclusions $B_i^{p,q}\subseteq B_{i+1}^{p,q}$ shown in \cite[Lemma 4.1]{tripaldi2026spectralcomplexestruncatedmulticomplexes}, we immediately get that the inclusion $B_1^{1+i,1-i}\subseteq B_i^{1+i,1-i}$ holds for every $i\ge 1$.

The reverse inclusion is also immediate when \(i=1\), so assume that \(i\geq2\). Given $\alpha\in B_i^{1+i,1-i}$, there exist 1-forms $c_{i+1-j}\in\Omega^{i+1-j,j-i}(G)$ for $0\le j\le i-1$, such that
\begin{equation}\label{eq:Bi-degree-two-expression}
\alpha
=
\sum_{j=0}^{i-1}d_jc_{i+1-j}
=
d_0c_{i+1}+d_1c_i+\cdots+d_{i-1}c_2,
\end{equation}
and
\begin{equation}\label{eq:Bi-degree-two-relations}
0
=
\sum_{j=l}^{i-1}d_{j-l}c_{i+1-j}\ ,
\
\text{for every }1\leq l\leq i-1\,.
\end{equation}
Taking $l=i-1$ in
\eqref{eq:Bi-degree-two-relations}, we obtain $d_0c_2=0$ where $c_2\in\Omega^{2,-1}(G)$, and so by
Lemma~\ref{lem:d0-injective-higher-weight-one-forms} we get that $c_2=0$.
Taking now $l=i-2$ in \eqref{eq:Bi-degree-two-relations}, we obtain $d_0c_3+d_1c_2=d_0c_3=0$, which also implies $c_3=0$ since $c_3\in\Omega^{3,-2}(G)$.

Proceeding inductively, suppose that $c_2=\cdots=c_{r-1}=0$ for some $3\le r\le i$. Then taking $l=i+1-r$ in 
\eqref{eq:Bi-degree-two-relations}, all the terms involving $c_2,\ldots,c_{r-1}$ vanish, and we are left with $d_0c_r=0$, where $c_r\in\Omega^{r,1-r}(G)$ and since $r\ge 2$, Lemma \ref{lem:d0-injective-higher-weight-one-forms} implies $c_r=0$. 

Therefore $c_2=c_3=\cdots=c_i=0$, and so by \eqref{eq:Bi-degree-two-expression} we have that $\alpha=d_0c_{i+1}$, that is $\alpha\in  B_1^{1+i,1-i}$, i.e. $B_i^{1+i,1-i}\subseteq B_1^{1+i,1-i}$.
\end{proof}

\begin{corollary}\label{cor: Delta_i on 1-forms}
    Under the identifications
\begin{align*}
    Z_i^{1,0}\big/ B_1^{1,0} = Z_i^{1,0}\ \text{ and }\
    Z_m^{i+1,1-i}\big / B_1^{i+1,1-i} &\cong  Z_m^{i+1,1-i} \cap E_0^2\,,
\end{align*}
the differential 
\begin{align*}
    \Delta_i\colon Z_i^{1,0}/B_1^{1,0}\longrightarrow Z_m^{1+i,1-i}/B_1^{1+i,1-i}
\end{align*}
is represented by
\begin{align*}
    d_c^i \colon Z_i^{1,0}\longrightarrow Z_m^{1+i,1-i}\cap E_0^2\,.
\end{align*}
\end{corollary}
\begin{proof}
    The first identification follows from $B_1^{1,0} = 0$ and the second is given by Lemma \ref{lem: Harmonic reduction of the spectral quotients}. For $\alpha \in Z_i^{1,0}$, the Rumin expression for $\Delta_i$ contains no higher-weight harmonic correction and hence $\Delta_i([\alpha]) = [d_c^i \alpha]$. Since $B_i^{i+1, 1-i} = B_1^{i+1, 1-i}$ and $d_c^i \alpha$ is harmonic the claim follows.
\end{proof}

The Pansu pullback preserves horizontal $1$-forms, so we do not need to project by $\Pi_0$. In the target, the induced map is represented by $\Pi_0 \varphi_P^*$ as in Corollary \ref{cor: pullback in the subspaces}. Theorem \ref{thm: commutativity} therefore takes the following form.

\begin{corollary}\label{cor: pullback commuta con diff rumin 1-forms}
    Let $\varphi : G_1 \rightarrow G_2$ be a smooth Pansu differentiable map. For every $\alpha \in Z_r^{1,0}(G_2)$, one has
    \begin{equation*}
        \Pi_0 \varphi_P^* d_c^r \alpha = d_c^r \varphi_P^* \alpha \,.
    \end{equation*}
    Equivalently, there exist $\beta_r \in \Omega^{r+1,-r}(G_1)$ such that 
    \begin{equation*}
        \varphi_P^* d_c^r \alpha - d_c^r \varphi_P^*\alpha =d_0 \beta_r 
    \end{equation*}
\end{corollary}

We conclude this section with a further application of Theorem \ref{thm: commutativity} to the Rumin complex. While the preceding results concern horizontal $1$-forms on any Carnot group, we now consider a new condition that allows the same approach to be applied in arbitrary degrees.

\begin{definition}[Carnot groups with the non-splitting property]
    We say that a Carnot group $G$ satisfies the non-splitting property if, for each degree $0\leq k\leq\dim G$, the space $E_0^k(G)$ is concentrated in a single weight. Thus, for each $k$, there exists a weight $p$, such that $ E_0^k \subset \Omega^{p,k-p}(G)$.
\end{definition}
Examples of such groups include Heisenberg groups $\mathbb H^n$ for every $n\geq1$, and the five-dimensional Cartan group (see \cite[$N_{5,2,3}$]{LeDonneTripaldi2021})

\begin{corollary}\label{cor: comm d_c non splitting}
    Let $\varphi : G_1 \rightarrow G_2$ be a smooth Pansu differentiable map. Suppose that, for some degree $k$, weight $p$ and integer $r\geq 1$,
    \begin{align*}
        E_0^k(G_h) \subset \Omega^{p,k-p}(G_h)\  \text{ and }\  E_0^{k+1}(G_h) \subset \Omega^{p+r,k+1-p-r}(G_h) \quad \text{ for } h=1,2\,.
    \end{align*}
    Then
    \begin{equation*}
        \Pi_0 \varphi_P^* d_c \alpha = d_c \Pi_0 \varphi_P^* \alpha \quad \text{for every } \alpha \in E^k_0(G_2)\,.
    \end{equation*}
    In particular, if $G_1$ and $G_2$ satisfy the non-splitting property with the same sequence of weights, this conclusion holds in every degree and 
    \begin{equation*}
        \Pi_0 \varphi_P^* : \big(E_0^\bullet (G_2), d_c\big) \rightarrow \big(E_0^\bullet(G_1), d_c \big) 
    \end{equation*}
    is a morphism of complexes.
\end{corollary}
\begin{proof}
    Since the Rumin differential acting on $E_0^k(G_h)$ has only the component that increase the weight by $r$, Proposition \ref{prop: formulating main thm with Rumin differentials} gives $Z_r^{p,k-p}(G_h) = Z_1^{p,k-p}(G_h)$. Similarly, $B_r^{p+r,k+1-p-r}(G_h) = B_1^{p+r,k+1-p-r}$ and the differential $\Delta_r$ is represented by $d_c = d_c^r$. The equality then follows by Theorem \ref{thm: commutativity}.
\end{proof}

\begin{remark}
    We emphasise that, as explained in \cite[Section 4.1]{tripaldi2026spectralcomplexestruncatedmulticomplexes}, under the non-splitting assumption, the family of spectral complexes collapses to a single complex, which coincides with the Rumin complex. Therefore, in the case where $G$ is a Carnot group with the non-splitting property, given $\varphi\colon G\to G$ a smooth Pansu differentiable map, Corollary \ref{cor: comm d_c non splitting} implies the commutativity between the Rumin differential and the Pansu pullback, i.e. $d_c \Pi_0\varphi_P^\ast=\Pi_0\varphi_P^\ast d_c$ in each degree.

\end{remark}

\subsubsection{On $L^2$ realisations of spectral complexes}\label{subsect: Orthogonal realisations of the spectral complexes}
The fibrewise realisations used above do not extend formally to all the quotients $Z_r^{p,k-p}/B_m^{p,k-p}$. A natural approach would be to introduce an $L^2$-inner product, but this would require additional choices of domains and further analytic results. We briefly describe these requirements.

The pointwise inner product introduced in \eqref{eq: def pointwise scalar product} induces a scalar-valued inner product on compactly supported smooth forms:
\begin{align*}
    (\alpha_1,\alpha_2)_{L^2}:=\int_G\langle\alpha_1,\alpha_2\rangle_k\operatorname{vol}=\int_G\alpha_1\wedge\star\alpha_2 \quad \alpha_1,\alpha_2\in\Omega_c^k(G)\,.
\end{align*}
 The completion of $\Omega_c^k(G)$ with respect to the associated norm is the Hilbert space $L^2\Omega^k(G)$. Since the layers of $\mathfrak g$ are mutually
orthogonal, forms of different weights are orthogonal, and therefore
$$
L^2\Omega^k(G)
=
\bigoplus_{p=0}^Q L^2\Omega^{p,k-p}(G) \quad\text{where} \quad L^2\Omega^{p,k-p}(G)\cong L^2(G) \otimes \bigwedge\nolimits^{p,k-p} \mathfrak{g}^*
$$
is an orthogonal direct sum. 

The first issue concerns the exterior derivative, which is not defined as an $L^2$-valued operator on all $L^2\Omega^k(G)$. For $\alpha\in L^2\Omega^k(G)$, its distributional
exterior derivative is defined by
$$\langle d\alpha,\eta\rangle_{\mathrm{dist}}
    :=
    (-1)^{k+1}\int_G\alpha\wedge d\eta,
    \qquad
    \eta\in\Omega_c^{n-k-1}(G)\,,$$
where the left-hand side denotes the distributional pairing
with a test form of complementary degree. A natural choice is to restrict to the maximal domain

\begin{equation*}
    \Omega_{\flat}^k(G) = \big\{ \alpha \in L^2\Omega^k(G) \mid d\alpha \in L^2\Omega^{k+1}(G) \text{ distributionally }\big\}\,.
\end{equation*}
Thus, as $d^2 = 0$, $\big(\Omega_{\flat}^\bullet(G), d\big)$ is a complex. 
This does not yet provide the multicomplex structure required by the construction in Section \ref{chapter truncated multicomplex}. One would also have to choose a bigraded domain preserved by each component $d_j$. These properties do not follow merely from the condition $d\alpha \in L^2\Omega^{k+1}(G)$. Indeed terms arising from different weight components of $\alpha$ can contribute to the same component of $d\alpha$ and their sum being in $L^2$ does not control each term separately.
A smaller invariant domain may therefore be necessary. 

Suppose that a compatible domain has been specified. Denote it by $\widetilde{\Omega}^\bullet_{\flat}(G)$ and construct the corresponding cycle and boundary spaces $Z_r^{p,k-p}$ and $B_m^{p,k-p}$.

For any linear subspace $S \subset \widetilde{\Omega}_{\flat}^{p,k-p}(G)$, one has $
S^\perp=(\overline S^{\,L^2})^\perp
$. Therefore, an orthogonal-complement realisation cannot distinguish
a subspace from its $L^2$-closure.
To obtain an orthogonal representative of every class modulo $\overline{B_m^{p,k-p}}^{\,L^2}$, one must first verify if
\begin{equation}\label{eq:cond for subspaces L2}
    \overline{B_m^{p,k-p}}^{\,L^2} \subset Z_r^{p,k-p}\,.
\end{equation}
If, for example, $Z_r^{p,k-p}$ is $L^2$-closed, Condition \eqref{eq:cond for subspaces L2} is automatically verified. Nevertheless, the spectral cycle conditions involve the existence of correction forms, whose solvability requires separate analysis.

Under Condition \eqref{eq:cond for subspaces L2}, one has
\begin{equation*}
    Z_r^{p,k-p}
\cap
\left(B_m^{p,k-p}\right)^\perp \cong Z_r^{p,k-p}/
\overline{B_m^{p,k-p}}^{\,L^2}
\end{equation*}
If the boundary space is not closed, this construction represents a reduced quotient, rather than the algebraic quotient \(Z_r^{p,k-p}(G)/B_m^{p,k-p}(G)\).

Finally, the pullback introduces further requirements. Proposition \ref{prop:Pansu-preserves-Zr-Br} establishes preservation of the smooth cycle and boundary spaces, but does not by itself establish preservation of the domain $\widetilde{\Omega}^\bullet_{\flat}(G_h)$.
Even assuming that the Pansu pullback admits a bounded extension preserving $Z_r^{p,k-p}$ and $B_m^{p,k-p}$, the orthogonal
complements may not be preserved. Therefore, one
does not usually have
$$
\varphi_P^*
\left(
\left(B_m^{p,k-p}(G_2)\right)^\perp
\right)
\subseteq
\left(B_m^{p,k-p}(G_1)\right)^\perp\,.
$$
Thus, the induced Pansu pullback is represented by
$$
\alpha
\longmapsto
\operatorname{pr}_{(\overline{B_m(G_1)}^{\,L^2})^\perp}
\bigl(\varphi_P^*\alpha\bigr).
$$

These considerations explain why we retain the smooth quotient formulation and use the fibrewise harmonic reduction in the cases needed for Section 5. In those cases, the relevant boundary spaces equal $B_1=\operatorname{Im}d_0$, so no global integrability condition or Hilbert-space completion is required.

\section{Lifting Pansu derivatives to central extensions}\label{section application}

The result of Theorem \ref{thm: commutativity} on the commutativity between the Pansu pullback and the differentials $\Delta_i$ appearing in the spectral complexes has a very direct consequence when applied to central extensions of nilpotent Lie groups. We remind the reader that throughout this section the Pansu differentiable maps considered are assumed to be (Euclidean) smooth. A standard reference for all definitions and statements of the classical results contained in this section is Chapter 1, Section 4 of \cite{FUKS}.

\begin{definition}[Central extensions: definition and main properties]\label{def: central ext}
    A central extension of a Lie algebra $\mathfrak{g}$ by a vector space $V$ is a short exact sequence
    \begin{align*}
        0\longrightarrow V\xrightarrow[]{\iota_\mathfrak{g}}\hat{\mathfrak{g}}\xrightarrow[]{\pi_\mathfrak{g}}\mathfrak{g}\longrightarrow 0
    \end{align*}
    such that the image of the homomorphism $\iota_\mathfrak{g}\colon V\longrightarrow\hat{\mathfrak{g}}$ is contained in the centre $Z(\hat{\mathfrak{g}})$ of the Lie algebra $\hat{\mathfrak{g}}$ and the linear subspace $V$ is considered as an abelian Lie algebra.

Choosing a linear splitting of $\pi_{\mathfrak g}$
identifies $\hat{\mathfrak g}$, as a vector space,
with $V\oplus\mathfrak g$, with the standard inclusion
and projection. The Lie bracket in the vector space $V\oplus\mathfrak{g}$ can be defined by the formula
    \begin{align*}
        \left[(u,X),(v,Y)\right]_{\hat{\mathfrak{g}}}:=\left(\omega(X,Y),[X,Y]_{\mathfrak{g}}\right)\ \text{ for all }u,v\in V\text{ and }X,Y\in\mathfrak{g}\,,
    \end{align*}
    where $\omega$ is a skew-symmetric bilinear function on $\mathfrak{g}$ which takes values in the vector space $V$ and $[\cdot,\cdot]_{\mathfrak{g}}$ denotes the Lie bracket of the Lie algebra $\mathfrak{g}$. One can verify directly that the Jacobi identity for the new bracket $[\cdot,\cdot]_{\hat{\mathfrak{g}}}$ is equivalent to the condition that the bilinear function is a cocycle, i.e. the following equality holds identically
    \begin{align*}
        \omega([X,Y]_{\mathfrak{g}},Z)+\omega([Y,Z]_{\mathfrak{g}},X)+\omega([Z,X]_{\mathfrak{g}},Y)=0\ \text{ for all }X,Y,Z\in\mathfrak{g}\,.
    \end{align*}
    Two extensions are called equivalent if there is an isomorphism of Lie algebras $\Phi\colon\hat{\mathfrak{g}}_1\longrightarrow\hat{\mathfrak{g}}_2$ such that the following diagram is commutative
    \[
\begin{tikzcd}
0 \arrow[r, " "] \arrow[d, " "] 
  & V \arrow[d, "\operatorname{Id}"]\arrow[r, " \iota_1"] & \hat{\mathfrak{g}}_1 \arrow[d, "\Phi" ]\arrow[r, "\pi_1"] & \mathfrak{g} \arrow[d, "\operatorname{Id}"]\arrow[r, " "] & 0\arrow[d," "]\\
0 \arrow[r, " "] 
  & V\arrow[r, " \iota_2"] & \hat{\mathfrak{g}}_2\arrow[r, "\pi_2"]  & \mathfrak{g} \arrow[r, " "] & 0
\end{tikzcd}
\]
A cocycle $\omega$ is cohomologous to zero if there exists a linear mapping $\mu\colon\mathfrak{g}\longrightarrow V$ such that $\omega(X,Y)=-\mu([X,Y]_{\mathfrak{g}})$ for all $X,Y\in\mathfrak{g}$. In this situation, the cocycle is called a coboundary and we denote this as $\omega=d\mu$. 

Two cocycles $\omega$ and $\omega'$ are cohomologous if their difference is cohomologous to zero, i.e. $\omega-\omega'=d\mu$. Cohomologous cocycles define equivalent central extensions. To prove this, it suffices to verify that the linear mapping 
\begin{align*}
    \Phi=\operatorname{Id}+\mu\colon V\oplus\mathfrak{g}\longrightarrow V\oplus\mathfrak{g}\ ,\ \Phi(v,X)=\left(v+\mu(X),X\right)
\end{align*}
is an isomorphism of Lie algebras (as in the diagram above).
\end{definition}
\begin{remark}\label{remark: V is 1-dimensional}
    In the present discussion, we focus on 1-dimensional central extensions, that is we will only take $V=\operatorname{span}_\mathbb{R}\{T\}=\mathbb R$ (we are working with real Lie algebras). In this case, the entire central extension construction can be more efficiently rephrased in terms of left-invariant 2-forms. Indeed, $\omega$ being a skew-symmetric bilinear function on $\mathfrak{g}$ which takes values in $\mathbb R$ is equivalent to requiring $\omega\in\bigwedge^2\mathfrak{g}^\ast\cong\Omega_L^2(G)$. Furthermore, the cocycle condition on $\omega$ means that $d\omega=d_0\omega=0$, i.e. $\omega\in\ker d_0$, while $\omega$ being cohomologous to zero means that $\omega=d\beta=d_0\beta$ for some $\beta\in\mathfrak{g}^\ast$. Note again that an exact 2-covector $\omega=d_0\beta$ defines an extension $\hat{\mathfrak{g}}$ which is isomorphic to the direct sum of $\operatorname{span}_{\mathbb R}\{T\}\oplus\mathfrak{g}$ where $[T,X]=0$ for any $X\in\mathfrak{g}$. Such a central extension is called trivial.  

    In other words, non trivial 1-dimensional central extensions are classified by the Lie algebra cohomology classes in degree 2, i.e. $H^2(\mathfrak{g},\mathbb R)$. On the other hand, the action of $d$ on left-invariant forms coincides with $d_0$ and so (when specialised to the case of 2-forms) $E_0^2=\ker\Box_0\cap\Omega^{2}(G)\cong C^\infty(G)\otimes H^2(\mathfrak{g},\R)$.
\end{remark}

In our considerations, we are going to use the following classical result.
\begin{theorem}
    \label{thm: algebraic construction}
    Suppose we have two 1-dimensional central extensions
\begin{align*}
0\longrightarrow\mathbb R
\xrightarrow{\iota_1}\hat{\mathfrak g}_1
\xrightarrow{\pi_1}\mathfrak g_1
\longrightarrow0,\\
0\longrightarrow\mathbb R
\xrightarrow{\iota_2}\hat{\mathfrak g}_2
\xrightarrow{\pi_2}\mathfrak g_2
\longrightarrow0.
\end{align*}
        classified by the two corresponding cohomology classes $[\omega]\in H^2(\mathfrak{g}_1,\mathbb R)$ and $[\zeta]\in H^2(\mathfrak{g}_2,\mathbb R)$.

A Lie algebra homomorphism
$\varphi:\mathfrak g_1\to\mathfrak g_2$
lifts to a Lie algebra homomorphism
$\Phi:\hat{\mathfrak g}_1\to\hat{\mathfrak g}_2$
satisfying
\[
    \pi_2\circ\Phi=\varphi\circ\pi_1
    \quad \text{and}\quad
    \Phi\circ\iota_1=\iota_2
\]
if and only if
\[
    [\omega]=\varphi^*[\zeta]
    \quad\text{in }H^2(\mathfrak g_1,\mathbb R)\ \Longleftrightarrow\ \omega=\varphi^*\zeta+d\eta
    \quad\text{for some }\eta\in\mathfrak g_1^*.
\]
The condition $\Phi\circ\iota_1=\iota_2$ requires the lift
to induce the identity on the prescribed central kernels.
        
        If this is the case, then the map
        \begin{align*}
            \Phi\colon\hat{\mathfrak{g}}_1\cong\mathbb R\oplus\mathfrak{g}_1\longrightarrow\hat{\mathfrak{g}}_2\cong\mathbb R\oplus\mathfrak{g}_2\ ,\ \Phi(u,X):=\left(u+\eta(X),\varphi(X)\right)
        \end{align*}    
        gives rise to the lift we are looking for, since
        \begin{align*}
            &\Phi\left([(u,X),(v,Y)]_{\hat{\mathfrak{g}}_1}\right)=\Phi\left(\omega(X,Y),[X,Y]_{\mathfrak{g}_1}\right)=\left(\omega(X,Y)+\eta([X,Y]_{\mathfrak{g}_1}),\varphi([X,Y]_{\mathfrak{g}_1})\right)\text{ while }\\
            &\left[\Phi(u,X),\Phi(v,Y)\right]_{\hat{\mathfrak{g}}_2}=\left[(u+\eta(X),\varphi(X)),(v+\eta(Y),\varphi(Y))\right]_{\hat{\mathfrak{g}}_2}=\left(\zeta(\varphi(X),\varphi(Y)),[\varphi(X),\varphi(Y)]_{\mathfrak{g}_2}\right)
        \end{align*}
        and they coincide since $\varphi\colon\mathfrak{g}_1\to\mathfrak{g}_2$ is a Lie algebra homomorphism, and so $\varphi([X,Y]_{\mathfrak{g}_1})=[\varphi(X),\varphi(Y)]_{\mathfrak{g}_2}$, but also for any $X,Y\in\mathfrak{g}_1$ we have $\varphi^\ast\zeta(X,Y)=\zeta(\varphi(X),\varphi(Y))=\omega(X,Y)-d\eta(X,Y)=\omega(X,Y)+\eta([X,Y]_{\mathfrak{g}_1})$.

        One can check that the identities
$\pi_2\circ\Phi=\varphi\circ\pi_1$ and
$\Phi\circ\iota_1=\iota_2$
follow directly from the formula for $\Phi$.

Conversely, suppose that a Lie algebra homomorphism $\Phi$
satisfying both compatibility conditions exists.
Since $\Phi$ is linear, the identity
$\pi_2\circ\Phi=\varphi\circ\pi_1$ implies that
$\Phi(0,X)=(\eta(X),\varphi(X))$
for some $\eta\in\mathfrak g_1^*$.
The additional condition $\Phi\circ\iota_1=\iota_2$
then gives
\[
    \Phi(u,X)=(u+\eta(X),\varphi(X)).
\]
Comparing the central components in
\[
    \Phi([(0,X),(0,Y)]_{\hat{\mathfrak g}_1})
    =
    [\Phi(0,X),\Phi(0,Y)]_{\hat{\mathfrak g}_2}
\]
yields
\[
    \omega(X,Y)+\eta([X,Y]_{\mathfrak g_1})
    =
    \zeta(\varphi(X),\varphi(Y)).
\]
Since $d\eta(X,Y)=-\eta([X,Y]_{\mathfrak g_1})$,
we conclude that
$\omega-\varphi^*\zeta=d\eta$.
Thus $[\omega]=\varphi^*[\zeta]$, proving necessity.
\end{theorem}

We are interested in applying Theorem \ref{thm: algebraic construction} to the case of Pansu differentiable maps between two Carnot groups $\varphi\colon G_1\to G_2$ since by Definition \ref{def: P deriv} the Pansu derivative is a group homomorphism and hence the induced map $D_P\varphi(x)\colon\mathfrak{g}_1\to\mathfrak{g}_2$ is a Lie algebra homomorphism for every $x\in G_1$. 

Indeed, we would like to understand whether the Lie algebra homomorphism $D_P\varphi(x)\colon\mathfrak{g}_1\to\mathfrak{g}_2$ can lift to a Lie algebra homomorphism $\Phi_x\colon\hat{\mathfrak{g}}_1\to\hat{\mathfrak{g}}_2$ on the central extensions of $\mathfrak{g}_1$ and $\mathfrak{g}_2$. 

For this purpose, we just need to specialise Theorem \ref{thm: commutativity} to consider smooth horizontal 1-forms $\alpha\in\Omega^{1,0}(G_2)$ whose $\Delta_r\alpha$ (for a non-trivial choice of $r$) is a left-invariant 2-form.

\begin{remark}
    As stated in Theorem \ref{thm: algebraic construction}, 1-dimensional central extensions are classified via left-invariant 2-forms that belong to the cohomology of the Chevalley-Eilenberg complex $H^2(\mathfrak{g}_h,\mathbb R)$. Since we are assuming to be working with a scalar product on $\mathfrak{g}_h$ with $h=1,2$, this means that there is a direct correspondence between $\ker d_0\cap\left(\operatorname{Im}d_0\right)^\perp\cap \Omega_L^2(G_h)$ and central extensions (this is a direct consequence of the fact that $d=d_0$ on left-invariant forms):
    \begin{align*}
        \omega\in\Omega_L^2(G_h)=\bigwedge\nolimits^2\mathfrak{g}_h^\ast\text{ defines a non-trivial central extension }\hat{\mathfrak{g}}_h\Longleftrightarrow d_0\omega=0\ \text{ and }\Pi_0\omega\neq 0\,.
    \end{align*}
    By Proposition \ref{prop: Bi é B1 per le 2 forme}, we have that for each appropriate choice of $i,m\in\N
    $
    \begin{align*}
        E_{m,i}^{1+i,1-i}(G_h)\cong Z_m^{1+i,1-i}(G_h)\cap\left(B_1^{i+1,1-i}(G_h)\right)^\perp\subset Z_1^{i+1,1-i}(G_h)\cap\left(B_1^{i+1,1-i}(G_h)\right)^\perp=\ker\Box_0\cap\Omega^{i+1,1-i}(G_h)\,.
    \end{align*}
    On the other hand, since we are limiting our considerations to the case of left-invariant 2-forms, we also have that for each non trivial choice of $p,j\in\N$
    \begin{align*}
        \omega\in\ker\Box_0\cap\Omega_L^{p,2-p}(G_l)\Longrightarrow \omega\in \ker d\cap\Omega^{p,2-p}(G_l)\Longrightarrow \omega\in Z_m^{p,2-p}(G_l)\text{ for every }m\ge 1\,.
    \end{align*}
This is true in general, since left invariant forms that belong to $\ker d_0\cap\Omega^\bullet(G_h)=Z_1^{\bullet,\bullet}(G_h)$ also belong to all possible $Z_r^{\bullet,\bullet}(G_h)$:
\begin{align}\label{eq: left-invariant closed}
    \text{ for all }\omega\in Z_1^{\bullet,\bullet}(G_h)\cap\Omega^\bullet_L(G_h) \text{ we have }\omega\in\bigcap_{r=1}^\infty Z_r^{\bullet,\bullet}(G_h)\,.
\end{align}
The statement can be easily checked by taking $z_{p+i}=0\in\Omega^{p+i,\bullet}(G_h)$ for any $1\le i< r$ (see Definition \ref{def: Z and B defined}), since $\omega\in\Omega_L^\bullet(G_h)$ implies $\omega\in\ker (d-d_0)$.
    
    We are going to need this for considerations on closedness vs. exactness properties of 2-forms.
\end{remark}

\begin{proposition}[Lifting the Pansu derivative to central extensions]\label{prop: extensions}
Let $G_1$ and $G_2$ be two Carnot groups and let $\varphi\colon G_1\to G_2$ be a Pansu differentiable map with $D_P\varphi(x)\colon\mathfrak{g}_1\to\mathfrak{g}_2$ its Pansu derivative in $x \in G_1$, and consider a left-invariant 2-form $\omega\in\ker\Box_0\cap\Omega^2_L(G_2)$ such that $\omega=\omega_{p_1}+\cdots+\omega_{p_N}\in\oplus_{i=1}^N\Omega^{p_i,2-p_i}(G_2)$ giving rise to the central extension $\hat{\mathfrak{g}}_2$.
Set $P=\{p_1,\ldots,p_N\}$, hence for each $s\in P$, define
$\alpha_{s-1}\in Z_{s-1}^{1,0}(G_2)$ satisfying
$d_c^{s-1}\alpha_{s-1}=\omega_s$, and set
$$\zeta:=
    \sum_{s\in P}
    d_c^{s-1}(\varphi_P^*\alpha_{s-1})\,.$$
Assume that $\zeta$ is left-invariant.
Then $\zeta$ is a harmonic $2$-cocycle and defines
a central extension $\hat{\mathfrak g}_1$.

For every $x\in G_1$, the homomorphism $D_P\varphi(x)$
admits a lift
$$\Phi_x:\hat{\mathfrak g}_1\to\hat{\mathfrak g}_2$$
inducing the identity on the central kernels.
More precisely, setting $\beta:=d_0^{-1}(\varphi_P^*\omega)$,
one such lift is $
    \Phi_x(u,X)
    =
    \bigl(u-\beta_x(X),D_P\varphi(x)X\bigr)$.

\end{proposition}

\begin{proof}
   The set up of spectral complexes greatly simplifies when considering the space of 1-forms on a Carnot group $G_l$. As shown in Corollary \ref{cor: Delta_i on 1-forms}, we can identify 
   $$\Delta_i\colon Z_i^{1,0}\longrightarrow Z_m^{1+i,1-i}/B_1^{1+i,1-i}\ \text{ with } \ d_c^i\colon Z_i^{1,0}\longrightarrow Z_m^{1+i,1-i}\cap E_0^2\,.$$
   Therefore, if the space of Rumin 2-forms is spanned by homogeneous forms of weight within the range $P=\{p_1,p_2,\ldots,p_N\}\subset\{n\in\N\mid n\ge 2\} $ , then the only non-trivial operators $d_c^i$ acting on 1-forms will be for $i\in P-1= \{p_1-1,p_2-1,\ldots,p_N-1\}$, that is
   \begin{align*}
       d_c^i\colon E_{i,1}^{1,0}(G_h)=Z_i^{1,0}(G_h)\longrightarrow E_{m_i,i}^{i+1,1-i}(G_h)\cong Z_{m_i}^{i+1,1-i}(G_h)\cap E_0^2
   \end{align*}
   for any appropriate choice of $m_i\in\N$.

   By Corollary \ref{cor: pullback commuta con diff rumin 1-forms}, Theorem \ref{thm: commutativity} takes the simplified form
   \begin{align}\label{eq: thm comm on 1-forms}
       \text{for all }\alpha\in E_{i,1}^{1,0}(G_2)\text{ we have that }\varphi_P^\ast d_c^i\alpha=d_c^i\varphi_P^\ast\alpha+d_0\beta\ \text{ for some }\beta\in\Omega^{i+1,-i}(G_1)\,.
   \end{align}
   Finally, we are left to apply \eqref{eq: thm comm on 1-forms} to our case, where we are considering a closed left-invariant 2-form $\omega=\omega_{p_1}+\cdots+\omega_{p_N}\in\bigwedge^2\mathfrak{g}_2^\ast$. As already shown in \eqref{eq: left-invariant closed}, for each $s\in P$ we have that $\omega_s\in\cap_{r=1}^\infty Z_r^{s,2-s}(G_2)$. By the theory of spectral sequences, we have that the cohomology of the underlying complex (in this case the de Rham cohomology) is carried by all the weights at page $E_\infty^{\bullet,\bullet}$. In other words, for each $s\in P$ we have
   \begin{align*}
       \omega_s\in Z_\infty^{s,2-s}(G_2)\ \Longrightarrow \ \omega_s\in B^{s,2-s}_\infty(G_2)
   \end{align*}
   where the implication comes from the fact that the de Rham cohomology of any Carnot group in degree 2 is trivial (i.e. closed 2-forms are exact).
   Since $B_\infty^{s,2-s}(G_2)=B_{s}^{s,2-s}(G_2)$ we have that \begin{align*}
       \text{for any }\omega_s\in\ker\Box_0\cap\bigwedge\nolimits^{s,2-s}\mathfrak{g}_2^\ast\ \text{ there exists }\alpha_{s-1}\in E_{s-1,1}^{1,0}(G_2)\text{ such that }d_c^{s-1}\alpha_{s-1}=\omega_s\,.
   \end{align*}
   Therefore by \eqref{eq: thm comm on 1-forms} we have that
   \begin{align*}
       \varphi_P^\ast\omega_s=\varphi_P^\ast d_c^{s-1}\alpha_{s-1}=d_c^{s-1}\varphi_P^\ast\alpha_{s-1}+d_0\beta_s\text{ for some }\beta_s\in\Omega^{s,1-s}(G_1)\,.
   \end{align*}
   Since this formula holds for each $s\in P$, by the linearity of the differentials and the Pansu pullback, we get that if we impose $\zeta:=\Pi_0\varphi_P^\ast\omega$, we have that $\varphi_P^\ast\omega=\zeta+d_0\beta$ with $\beta=d_0^{-1}(\varphi_P^\ast\omega)$.
   
   The final claim then follows directly from Theorem \ref{thm: algebraic construction}, since for each $x\in G_1$, $\zeta=(D_P\varphi(x))^\ast\omega-d_0\beta_x$ is left-invariant by assumption. Note that $\zeta$ does not depend on the choice of $\alpha$.
\end{proof}

\begin{proposition}[Lifting of graded homomorphisms to central extensions]\label{prop: Xiangdgong}
    Let $G_1$ and $G_2$ be two Carnot groups with respective Lie algebras $\mathfrak{g}_1$ and $\mathfrak{g}_2$, and let $\psi\colon G_1\to G_2$ be a graded homomorphism. Consider a left-invariant 2-form $\omega\in\ker\Box_0\cap\Omega^2_L(G_2)$ such that $\omega=\omega_{p_1}+\cdots+\omega_{p_N}\in\oplus_{i=1}^N\Omega^{p_i,2-p_i}(G_2)$ giving rise to the central extension $\hat{\mathfrak{g}}_2$ of $\mathfrak{g}_2$. Let $\hat{\mathfrak{g}}_1$ be the central extension of $\mathfrak{g}_1$ determined by the closed 2-form $\zeta=\psi^\ast\omega\in\Omega^2_L(G_1)$. Then we can lift $d(\psi)_0\colon\mathfrak{g}_1\to\mathfrak{g}_2$ to a Lie algebra homomorphism $d(\Psi)_0\colon \hat{\mathfrak{g}}_1\to\hat{\mathfrak{g}}_2$ as long as
    \begin{align*}
        \sum_{s=p_1}^{p_N}d_c^{s-1}\psi^\ast\alpha_{s-1}\in\Omega^2_L(G_1)=\bigwedge\nolimits^2\mathfrak{g}_1^\ast\ \text{ where each }\alpha_{s-1}\in E_{s-1,1}^{1,0}(G_2)\text{ satisfies }d_c^{s-1}\alpha_{s-1}=\omega_s\,.
    \end{align*}
    In this case, one can recover the formula for $d(\Psi)_0\colon\hat{\mathfrak{g}}_1\to\hat{\mathfrak{g}}_2$ from Theorem \ref{thm: algebraic construction}.
\end{proposition}
\begin{proof}
    Note that the proof of this statement need not rely on Theorem \ref{thm: commutativity}. Indeed, in the case where $\psi\colon G_1\to G_2$ is a graded homomorphism, the Pansu and the classical Frech\'et derivative coincide, and so the Pansu pullback is indeed a standard pullback. It is therefore sufficient to mimic the same reasoning as in Proposition \ref{prop: extensions} to first show that $d_c^{s-1}=\Delta_{s-1}$. Since the differentials of the spectral sequence are defined using only the exterior derivative (see Proposition \ref{prop: Delta con diff}), one can simply use the classical result that $d\psi^\ast=\psi^\ast d$.
\end{proof}
\begin{remark}
    Even though Proposition \ref{prop: extensions} holds in full generality, the question of whether a particular choice of left-invariant closed two form $\omega\in\bigwedge^\bullet\mathfrak{g}^\ast$ gives rise to a non trivial lift $\Phi_x\colon\hat{\mathfrak{g}}_1\to\hat{\mathfrak{g}}_2$ is a case-by-case study. Indeed, we know that the Pansu pullback necessarily keeps the weight of the 2-form constant, which means that $w(\varphi_P^\ast\omega)=w(\omega)$. Therefore, if we are picking a form such that $w(\omega)$ is greater than any of the weights of the 2-forms in $\ker\Box_0\cap\Omega^{2}(G_1)$ then necessarily $\Pi_0\varphi_P^\ast\omega$ will vanish and hence $\hat{\mathfrak{g}}_1$ will be the trivial central extension.

    This also goes hand in hand with other problems linked to central extensions. Indeed, it may well happen that two Lie algebras $\hat{\mathfrak{g}}_1$ and $\hat{\mathfrak{g}}_2$ that are central extensions of $\mathfrak{g}$ corresponding to non-cohomologous closed two forms $\omega_1,\omega_2\in\bigwedge^2\mathfrak{g}^\ast$ are nevertheless isomorphic.   
\end{remark}
\begin{corollary}[Lifting the Pansu derivative to stratified
central extensions]
\label{cor: extensions to stratifiable}
Let $G_1$ and $G_2$ be two Carnot groups, and let
$\varphi:G_1\to G_2$ be a Euclidean smooth Pansu
differentiable map.
Given a non-zero $\omega\in
    E_0^2(G_2)\cap\Omega_L^{p,2-p}(G_2)$ and an $\alpha\in Z_{p-1}^{1,0}(G_2)$ such that $d_c^{p-1}\alpha=\omega$, assume that
    $$E_0^2(G_1)\cap\Omega^{q,2-q}(G_1)=\{0\}
    \ \text{ for every }q>p\ \text{ and }\ d_c^{p-1}(\varphi_P^*\alpha)
    \in\bigwedge\nolimits^2\mathfrak{g}_1^\ast\setminus\{0\}\,.$$
Denote by $\hat G_1$ and $\hat G_2$ the Carnot groups
associated with the central extensions defined by
$d_c^{p-1}\varphi^\ast_P\alpha$ and $\omega$, respectively, and by
$\pi_h:\hat G_h\to G_h$ their projections for $h=1,2$.

Then there exists a smooth contact lift
$\Psi:\hat G_1\to\hat G_2$ satisfying $\pi_2\circ\Psi=\varphi\circ\pi_1$.
In particular, $\Psi$ is Pansu differentiable with $D_P\Psi(\hat x)=\Phi_{\pi_1(\hat x)}$ for every $\hat x\in\hat G_1$,
where $\Phi_x$ is the graded algebraic lift described below.
\end{corollary}
\begin{proof}
By Corollary \ref{cor: pullback commuta con diff rumin 1-forms}, $d_c^{p-1}\varphi_P^\ast\alpha=\Pi_0(\varphi_P^*\omega)$, hence both $d_c^{p-1}\varphi_P^\ast\alpha$ and $\omega$ define stratified central extensions
with central generators of weight $p$. In particular, we have that $\varphi_P^\ast\omega=d_c^{p-1}\varphi_P^\ast\alpha+d_0\beta$ with $\beta=d_0^{-1}\varphi_P^\ast\omega\in\Omega^{p,1-p}(G_1)$.

The weight assumption on $E_0^2(G_1)$ and the simple connectedness
of $G_1$ allow us to apply
\cite[Theorem 1.4]{hakavuori2025smoothcontactliftscentral},
with the identity map between the central kernels.
This yields a smooth contact lift $\Psi$ of $\varphi$.

For each $x\in G_1$, the map
\[
    \Phi_x(u,X)
    =
    \bigl(u-\beta_x(X),D_P\varphi(x)X\bigr)
\]
is a Lie algebra homomorphism by
Theorem~\ref{thm: algebraic construction} that is graded because $\beta$ has weight $p$.

For $\hat x\in\hat G_1$ and $x=\pi_1(\hat x)$,
both $D_P\Psi(\hat x)$ and $\Phi_x$ are graded
homomorphisms covering $D_P\varphi(x)$, and since $p\geq2$,
the projections identify the horizontal layers
of the extensions with those of the base groups.
Thus the two homomorphisms agree on the horizontal
layer. As the horizontal layer bracket-generates $\hat{\mathfrak g}_1$,
we have $D_P\Psi(\hat x)=\Phi_x$.
\end{proof}

\begin{remark}
The assumption that $E_0^2(G_1)$ contains no forms of weight
greater than $p$ can be replaced by the conditions
$$d_c^j(\varphi_P^*\alpha)=0
    \ \text{ for every }j\geq p\,,$$
provided that $\alpha\in Z_{p-1}^{1,0}(G_2)$ with $d_c^{p-1}\alpha=\omega$. Indeed, as long as the other hypotheses of the corollary are satisfied, \cite[Theorem 1.1]{hakavuori2025smoothcontactliftscentral}
yields the same contact-lifting conclusion.
\end{remark}

In contrast to the more standard case described in Corollary \ref{cor: extensions to stratifiable}, Proposition \ref{prop: extensions} holds in full generality when $\omega$ is \textit{not} homogeneous. In this case, assuming we are not dealing with some trivial extensions, if we start from the Pansu derivative $D_P\varphi(x)\colon\mathfrak{g}_1\to\mathfrak{g}_2$ of a Pansu differentiable map $\varphi\colon G_1\to G_2$ between two Carnot groups, we obtain a lift which is a homomorphism between two non-stratifiable Lie algebras $\hat{\mathfrak{g}}_1$ and $\hat{\mathfrak{g}}_2$. 

An explicit example can be easily found when considering a Pansu differentiable map $\varphi\colon \mathbb H^1\times\R\to \mathbb H^1\times\R$ where we are lifting the Lie algebra homomorphism $D_P\varphi(x)\colon\mathfrak{h}^1\times\R\to\mathfrak{h}^1\times\R$ using the non-homogeneous left-invariant 2-form $\omega=\theta_2\wedge\theta_3+\theta_1\wedge\tau$.
In this case, both central extensions $\hat{\mathfrak{g}}$ are isomorphic to the non stratifiable Lie algebra generated by $\{X_1,X_2,X_3,T,W\}$ with non trivial brackets $[X_1,X_2]=T$, $[X_1,T]=[X_2,X_3]=W$. This opens up the question as to whether the lifted Lie algebra homomorphism can be rephrased as an appropriate Pansu derivative $D_P\Psi(0,x)\colon\hat{\mathfrak{g}}\to\hat{\mathfrak{g}}$ of a map $\Psi\colon G\to G$ between the two non stratifiable groups.

Finally, the construction above allows the Pansu pullback itself
to be non-left-invariant, provided that its harmonic
projection is left-invariant. A further question is how
to extend the construction when this harmonic projection
varies with the base point, so that a fixed source
extension is no longer supplied by the preceding
criterion. Connections on principal bundles provide a
possible geometric setting for such questions, although
interpreting the relevant forms as curvature requires
de Rham closedness, which does not follow from
$d_0$-closedness alone.

\bigskip
\bibliographystyle{amsplain}
\bibliography{bibliography}

\end{document}